\documentclass[10pt]{amsart}

\usepackage[margin=1in, marginparsep=1pt,
            marginparwidth=124pt]{geometry}

\usepackage[english]{babel}
\usepackage[usenames,dvipsnames,svgnames,table]{xcolor}
\usepackage[pagebackref,linktocpage=true,colorlinks=true,linkcolor=Blue,citecolor=BrickRed,urlcolor=RoyalBlue]{hyperref}
\usepackage[alphabetic]{amsrefs}

\usepackage[colorinlistoftodos,prependcaption,textsize=tiny]{todonotes}
\renewcommand{\ge}{\geqslant}
\renewcommand{\leq}{\leqslant}
\renewcommand{\le}{\leqslant}

\usepackage{amsmath,amssymb,amsfonts,amsthm,enumerate,textcomp}
\usepackage{url}
\usepackage{graphicx}
\usepackage{wrapfig}

\usepackage[mathscr]{euscript} 

\usepackage{esint}
\usepackage{enumitem}

\theoremstyle{plain}
\newtheorem{theorem}{Theorem}[section]
\newtheorem{definition}[theorem]{Definition}
\newtheorem{lemma}[theorem]{Lemma}
\newtheorem{prop}[theorem]{Proposition}
\newtheorem{remark}[theorem]{Remark}
\newtheorem{corollary}[theorem]{Corollary}

\usepackage{amsmath}
\usepackage{amssymb}

\numberwithin{equation}{section}

\newcommand\R{\mathbb R} 
\newcommand\N{\mathbb N} 
\renewcommand\S{\mathbb S} 

\newcommand{\diam}{{ \rm diam}}
\newcommand{\supp}{{ \rm supp}}
\newcommand{\dist}{{ \rm dist}}
\newcommand\p{\partial}

\newcommand\e{\varepsilon}
\renewcommand{\epsilon}{\e}

\renewcommand\H{\mathscr H}

\newcommand{\I}[1]{\chi_{\{#1>0\}}}
\newcommand{\po}[1]{\{#1>0\}}
\newcommand{\fb}[1]{\partial\{#1>0\}}
\newcommand{\fbs}[1]{\partial_{\rm{sing}}\{#1>0\}}

\newcommand\Om{\Omega}
\newcommand\na{\nabla}
\newcommand\HD{{\rm{HD}}} 
\usepackage{tikz}
\usetikzlibrary{arrows,shapes,positioning}
\usetikzlibrary{decorations.markings}
\tikzstyle arrowstyle=[scale=1]
\tikzstyle directed=[postaction={decorate,decoration={markings,
    mark=at position .65 with {\arrow[arrowstyle]{stealth}}}}]
\tikzstyle reverse directed=[postaction={decorate,decoration={markings,
    mark=at position .65 with {\arrowreversed[arrowstyle]{stealth};}}}]

\usetikzlibrary{shadows}

\allowdisplaybreaks


\begin{document}

\title{A free boundary problem driven by the biharmonic operator}

\author{Serena Dipierro}
\address{Serena Dipierro: Department of Mathematics
and Statistics,
University of Western Australia,
35 Stirling Hwy, Crawley WA 6009, Australia}
\email{serena.dipierro@uwa.edu.au}

\author{Aram Karakhanyan}
\address{Aram Karakhanyan:
School of Mathematics, The University of Edinburgh,
Peter Tait Guthrie Road, EH9 3FD Edinburgh, UK}
\email{aram.karakhanyan@ed.ac.uk}

\author{Enrico Valdinoci}
\address{Enrico Valdinoci:
Department of Mathematics
and Statistics,
University of Western Australia,
35 Stirling Hwy, Crawley WA 6009, Australia}
\email{enrico.valdinoci@uwa.edu.au}

\begin{abstract}
In this paper we consider the minimization of the functional 
\begin{equation*}
J[u]:=\int_\Om \Big(|\Delta u|^2+\I u\Big)
\end{equation*}
in the admissible class of functions 
\begin{equation*}
\mathcal A:=
\left\{u\in W^{2, 2}(\Om) {\mbox{ s.t. }} u-u_0\in W^{1,2}_0(\Omega) \right\}.
\end{equation*}
Here, $\Omega$ is a smooth and bounded domain of~$\R^n$
and $u_0\in W^{2,2}(\Om)$ is a given function
defining the Navier type boundary condition. 

When $n=2$, the functional~$J$ can be interpreted as a sum of
the linearized Willmore energy of the graph of~$u$ and the area of~$\po u$ on 
the~$xy$ plane. 

\medskip 

The regularity of
a minimizer~$u$ and that of the free boundary~$\fb u$ are very
complicated problems. 
The most intriguing part of this is to study the structure of~$\fb u$ near 
singular points, where $\na u=0$ (of course,
at
the nonsingular free boundary points where~$\na u\not =0$
the free boundary is locally $C^1$ smooth). 

The scale invariance of the problem
suggests that, at the singular points of the free boundary,
quadratic growth of~$u$ is expected.
We prove that $u$ is quadratically nondegenerate at the
singular free boundary points 
using a refinement of Whitney's cube decomposition, which applies,
if, for instance, the set~$\po u$ is a John domain. \medskip

The optimal growth is linked with the approximate symmetries of the free boundary.
More precisely, if at small scales the free boundary can be approximated by
zero level sets of 
a quadratic degree two homogeneous  polynomial, then we say that $\fb u$ is rank-2 flat.
\medskip 
 
Using a dichotomy method for nonlinear free boundary problems,
we also show that, at the free boundary points~$x\in \Om$ where~$\na u(x)=0$,
the free boundary is either well approximated by 
zero sets of quadratic polynomials,
i.e. $\fb u$ is rank-2 flat, or $u$ has quadratic growth.
\medskip 
 
More can be said if $n=2$, in which case 
we obtain a monotonicity formula and
show that, at the singular points of the 
free boundary 
where the free boundary is not well approximated by level sets of
quadratic polynomials,
the blow-up of the minimizer is a homogeneous function of degree two.  
 
In particular, if~$n=2$ and $\po u$  is a John domain,
then we 
get that the blow-up of the free boundary is a cone,
and in the one-phase case it follows that $\fb u$ possesses
a tangent line in the measure theoretic sense. 

Differently from the classical free boundary problems
driven by the Laplacian operator, the one-phase minimizers
present structural differences with respect to the minimizers,
and one notion is not included into the other. In addition,
one-phase minimizers arise from the combination of a volume type free boundary
problem and an obstacle type problem, hence their growth condition
is influenced in a non-standard way by these two ingredients.
\end{abstract}

\date{\hbox{\today}}
\maketitle

\setcounter{tocdepth}{2}
\tableofcontents

\section{Introduction}\label{sec-intro} 

\subsection{Mathematical framework and motivations}

In this paper we consider the problem of minimizing the functional 
\begin{equation}\label{defJ}
J[u]=J[u,\Omega]:=\int_\Om \Big(|\Delta u|^2+\I u\Big)
\end{equation}
over the admissible class of functions 
\begin{equation}\label{ADMI}
\mathcal A:=
\left\{u\in W^{2, 2}(\Om) {\mbox{ s.t. }} u-u_0\in W^{1,2}_0(\Omega) \right\}.
\end{equation}
Here, $\Omega$ is a smooth and bounded domain of~$\R^n$
and $u_0\in W^{2,2}(\Om)$ is a given function
defining the Navier type boundary condition
(see e.g. the ``hinged problem'' 
on the right hand side of Figure 1(a) and
on page~84
of~\cite{SW}, or
Figure~1.5 on page~6 of~\cite{GANGULI},
or the monograph~\cite{GAZ} for additional
information of this condition, which can be interpreted
as a weak form of two boundary conditions: $u=u_0$ along~$\partial\Omega$
and~$\Delta u=0$
along~$\partial\Omega\cap\{u\ne0\}$).  

The functional in~\eqref{defJ} is clearly related to the biharmonic
operator, which provides classical models for rigidity
problems with concrete applications, for instance,
in the construction of suspension bridges, see e.g.~\cite{MR866720}
and the references therein.
Other classical applications of the biharmonic
operator arise in the study of
steady state incompressible fluid flows at small Reynolds
numbers under the Stokes flow
approximation assumption, see e.g. formula~(1)
in~\cite{MR3512704}
and the references therein.
In our setting, we will provide a simple mechanical interpretation
of the model in Section~\ref{INTER}.\medskip

Moreover, the functional in~\eqref{defJ}
provides a linearized model for the Willmore problem
which asks to find an immersion/embedding~$M$
in~$\R^3$ that minimizes the Willmore energy
\[
W(M)=\int_M H^2 \,dA,
\]
where~$H$ denotes the mean curvature.
The linearization of this energy density gives  
$$
H^2\,dA=\frac14(\Delta u)^2 \,dx\,dy +\text{lower order terms}.
$$
In this context our problem can be regarded as a free boundary problem for the linearized
Willmore energy, where the surface $M$ has a flat part on the $xy$ plane. 

We also refer to the very recent work in~\cite{DALIO}
for a problem related to the minimization of the Willmore energy 
functional with prescribed boundary, boundary Gauss
map and area. See also
the recent contributions in~\cites{MR3456944, MR3683120}
for the one-dimensional analysis of the global properties
of the solutions of free boundary problems involving
the curvature of a curve.\medskip

In the setting of~\eqref{defJ}, an additional motivation
for us comes from the study of
the degenerate/unstable obstacle problem, see~\cite{Caff-cpde, Monneau-W}. Indeed,
we will see in Corollary \ref{lem-subham}
that~$u$ is globally almost subharmonic in $\Omega$, i.e. 
there exists~$\hat C>0$ (possibly depending also on the energy
of the minimizer)
such that $\Delta u\ge -\hat C.$
Therefore
the function $\Delta u:=f$ is bounded from below. 
Accordingly, we can relate our problem to an obstacle 
problem with unknown right hand side, namely 
determine $u$ and $f\ge -\hat C$ such that 
\begin{equation}\label{hatcdelta}
\left\{
\begin{array}{lll}
\Delta u =f &\mbox{in}\ \ \Omega,\\
u=|\nabla u|=0 & \mbox{on}\ \ \fb u,\\
f=1 &\mbox{on}\ \ \fb u.
\end{array}
\right.
\end{equation}
The principal difference from the classical obstacle problem is that 
$f$ may change sign in $\Om$ and degenerate on the
free boundary points, since the last condition in~\eqref{hatcdelta} 
is satisfied in a generalized sense: for this reason,
it does not follow from the classical obstacle problem theory that
$u$ is quadratically nondegenerate.  \medskip

Another motivation for the problem
in~\eqref{defJ} comes from
the limit as~$\e{{\to}}0$ of the singularly perturbed bi-Laplacian equation
\begin{equation}\label{eq-sing-pert}
\Delta^2 u^\e=-\frac1\e\beta\left(\frac{u^\e}\e\right),
\end{equation}
where $\beta$ is a
compactly supported nonnegative function 
with  finite total mass, see~\cite{MR4026596}. Equation~\eqref{eq-sing-pert} can be seen as the
biharmonic counterpart of classical combustion models,
see e.g.~\cite{MR1900562}.

\subsection{Comparison with the existing literature}\label{sibsec:cmp}

Free boundary problems
are of course a classical topic of investigation,
nevertheless only few results are available concerning the case
of equations of order higher than two,
and there seems to be no investigation at all for
the free boundary problem in~\eqref{defJ}.

Other types of free boundary problems for higher order operators have been
considered in~\cite{mawi}. Moreover, obstacle problems involving
biharmonic operators have been studied in~\cites{frehse,
caffa, MR620427, MR705233, adams,
pozzo, novaga1, novaga2, 2016arXiv160306819A}, but till now
we are not aware of any previous investigation
of free boundary problems dealing with higher order operators
combined with ``bulk'' volume terms as in~\eqref{defJ} here.

Of course, one of the striking differences
in our framework, as opposed to the case of
the Alt-Caffarelli functional (see~\cite{MR618549})
\[
J_{\text{AC}}[u]:=\int_{\Om}\Big(|\na u|^2+\I u\Big),
\]
is the lack of Maximum Principle and Harnack inequality for higher order operators. This,
in our setting, reflects to the fact that
the set~$\{u<0\}$ may be nonempty,
even under the boundary condition~$u_0\ge 0$.
This is one of the peculiars of the
situations involving the bi-Laplacian and
it makes the mathematical treatment of the
problem extremely difficult (and this is likely to be the reason for which
there are not many results in the direction of
free boundary regularity in the framework that we consider here). \medskip

Thus, the main difficulties
in our setting, in comparison with the existing literature, follow from the fact that major elliptic
methods based on Maximum Principle,
Harnack inequality and propagation of ellipticity cannot be applied.
Moreover,
many classical
tools, such
as domain variations, have not been fully analyzed yet
and, in any case, cannot provide consequences which are as strong as in the classical framework.
For instance, the main result that we obtain by domain variation
(given in details in Lemma~\ref{CONFF})
is that, for any~$\phi=(\phi^1, \dots, \phi^n)\in C^\infty_0(\Omega)$,
\begin{equation}\label{stat-point}
2\int_\Omega \Delta u(x)\sum_{m=1}^n\Big( 2\nabla u_m(x)\cdot \nabla\phi^m(x)
+u_m(x)\Delta\phi^m(x)\Big)\,dx=
\int_\Omega
\Big( |\Delta u(x)|^2+\chi_{\{u>0\}}(x)\Big)\mbox{\rm div}\phi(x)\,dx.
\end{equation}
As customary, we denote by~$u_m=\partial_mu=\partial_{x_m}u$
the partial derivative of~$u$ with respect to the~$m$th variable.
Then, in the classical literature,
the standard argument
leading to the monotonicity formula
for the Alt-Caffarelli problem would be to choose $\phi$ of a particular form, see~\cite{MR1620644}.
More precisely, for $\e>0$, the classical idea would be to consider
\begin{displaymath}
\eta(x):=\left\{ \begin{array}{lll}
1 &{\rm if}\ x\in B_r(x_0),\\
\displaystyle
\frac{r+\e-|x-x_0|}{\e}\ & {\rm if}\ x\in B_{r+\e}(x_0)\setminus B_r(x_0),\\
0\ & {\rm otherwise},
\end{array}\right.
\end{displaymath}
where $x_0\in \fb u$, and take $\phi(x):=x\eta(x)$ in identity \eqref{stat-point}.
Note that 
\begin{displaymath}
\na \phi(x)=\left\{ \begin{array}{lll}
\mathbb I &{\rm if}\ x\in B_r(x_0),\\
\displaystyle \mathbb I \eta-\frac1\e\frac{(x-x_0)\otimes (x-x_0)}{|x-x_0|} & 
{\rm if}\ x\in B_{r+\e}(x_0)\setminus B_r(x_0),\\
0\ & {\rm otherwise,}
\end{array}\right.
\end{displaymath}
where~$\mathbb I\in{\rm{Mat}}_{n\times n}$ is the identity matrix.
However, in our case, identity \eqref{stat-point} contains the term~$\Delta \phi$ which is not defined on the boundary of the ring 
$B_{r+\e}(x_0)\setminus B_r(x_0)$ and this creates an important conceptual
difficulty. Thus, to overcome
this issue,
one needs to perform a series of ad-hoc
integration by parts. This strategy has however to
confront with the possible 
generation of
third order derivatives of the minimizers,
which also cannot be controlled, therefore these terms need
to be suitably smoothened and simplified via appropriate cancellations.

In this setting, the lack of monotonicity formulas
can also be seen as a counterpart of a lack of Pohozhaev type inequalities,
and our approach bypasses this kind of difficulty.

As a matter of fact, we will establish a new
monotonicity formula
in dimension~$2$ which will lead to
Theorem~\ref{lemma:F}.
\medskip

In addition, differently from the harmonic case,
there are no estimates available in the literature for the biharmonic measure,
and this makes the free boundary analysis significantly more
complicated. We will overcome these difficulties by Theorem~\ref{thm-Hausdorff}.

Moreover, in terms of barrier and test functions,
an additional difficulty of the biharmonic setting is given by the fact that
the function~$\max\{u, v\}$
is not an admissible competitor, having possibly infinite energy,
so we cannot consider the maximal and minimal solutions. 

The analysis of nondegeneracy and optimal regularity of minimizers
and of their free boundary is also a novel ingredient
with respect to the classical literature, and nothing seemed to be known before about
these important questions.

\subsection{Main results}

In what follows, we will denote by~$\{u>0\}$ the positivity set of~$u$ and
by~$\partial\{u>0\}$ its free boundary.
The main results of this paper are the following:

\begin{itemize}
\item If $z\in \fb u$ and $\na u(z)=0$,
then either $\fb u$ can be approximated by the zero level sets of a
quadratic homogeneous polynomial of degree two, or~$u$
has quadratic growth at~$z$. 

\item If $n=2$, there exists a monotonicity formula and we can 
classify the homogeneous one-phase solutions of degree two.

\item We also provide various sufficient conditions for strong nondegeneracy
in terms of a suitable
refinement of Whitney's cube decomposition ($c$-covering). For instance, we
show that if 
$\po u$ is a John domain (see the definition in Subsection~\ref{WHYWTH}),
then $\fb u$ possesses a measure theoretic tangent line.
\end{itemize}


\begin{figure}
\tikzstyle{decision} = [diamond, draw, top color=red!10, bottom color=red!50, 
    text width=7.5em, text badly centered, node distance=3cm, inner sep=0pt, drop shadow]
\tikzstyle{block} = [rectangle, draw, top color =cyan!60, 
    text width=11em, text centered, rounded corners, minimum height=4em, drop shadow]
\tikzstyle{line} = [draw, -latex']
\tikzstyle{cloud} = [draw, ellipse,fill=yellow!60, node distance=3cm,
    minimum height=2em,  text width=11em, text centered, drop shadow]
\tikzstyle{cloud1} = [draw, ellipse,fill=yellow!60, node distance=3cm,
    minimum height=2em]



\scalebox{.9}{
\begin{tikzpicture}[node distance = 2cm, auto, remember picture]
    \node [decision] (init) {\small {\bf Analysis of the free boundary}};
    \node [decision, left of=init, node distance=4.6cm] (tbplus) {\small\bf Regularity of minimizers};
    \node [block, below of=tbplus, node distance=3cm] (uminlin) {\small {\bf$\Delta u$ in BMO}
    \\ (Theorem~\ref{thm-BMO})};
     \node [decision, right of=init, node distance=4.6cm] (tbplus1) {\small\bf Special features in the plane};
    \node [block, below of=tbplus1, node distance=3cm] (uminlin1) {\small {\bf Monotonicity formula} \\(Theorem~\ref{lemma:F})};
    \node [block, below of=uminlin1, node distance=3cm] (upluslin1) {\small {\bf Classification of blow-up limits} \\ (Theorem~\ref{thm-hom-blow})};
    \node [block, right of=upluslin1, node distance=4.6cm] (asym1) {\small {\bf Existence of measure theoretic tangent lines}\\ (Theorem~\ref{BYPA})};
\node [block, below of=init, node distance=-4.4cm] (step1) {\small {\bf Nondegeneracy and growth from below}\\ (Theorem~\ref{thm-nondeg})};
\node [block, right of=step1, node distance=4.6cm] (kks1) {\small {\bf
Measure theoretic properties of free boundary points} \\ (Theorem~\ref{thm-Hausdorff})};
\node [block, below of=uminlin, node distance=-7.4cm] (kks2) {\small {\bf Free boundary \\ condition}\\ (Theorem~\ref{FREE BOU COND})};
\node [decision, below of =init, node distance=7.2cm] (existk0) {\small {\bf Classification of free boundary points}};
\node [block, below of=existk0, node distance=4.6cm, 
top color=green!20,
bottom color=green!80] (s2) {\small {\bf Rank-2 flat points}};
\node [block, right of=s2, node distance=4.5cm, top color=green!20,
bottom color=green!80] (s1) {\small {\bf Singular, rank-2 non-flat points}};
\node [cloud, below of=s2, node distance=2.5cm, top color=yellow!20,
bottom color=yellow!60] (flat) {\small {\bf Stratification \\
of the free boundary} \\ (Theorem~\ref{thm-strata})};
\node [block, left of=s2, node distance=4.5cm, top color=green!20,
bottom color=green!80] (scale) {\small{\bf Nonsingular points} \\($\nabla u\ne0$)};
\node [block, right of=s1, node distance=5cm] (s3) {\small {\bf Quadratic growth from above}
\\ (Theorem~\ref{growth})};
    %
    %
    \path [line] (init) -- (tbplus);
    \path [line] (tbplus) -- (uminlin);
    \path [line] (init) -- (step1);
   \path [line] (init) -- (tbplus1);
   \path [line] (init) -- (kks1);
   \path [line] (init) -- (kks2);
    \path [line] (tbplus1) -- (uminlin1);
    \path [line] (uminlin1) -- (upluslin1);
    \path [line] (uminlin1) -- (asym1);
    \path [line] (existk0) -- (scale);
    \path [line] (scale) -- (flat);
    \path [line] (s1) -- (flat);
    \path [line] (s2) -- (flat);
        \path [line] (s1) -- (s3);
    \path [line] (init) -- (existk0);  
    \path [line] (existk0) -- (s1);  
    \path [line] (existk0) -- (s2); 
\end{tikzpicture}
}
\caption{A road map of this article}
\label{x345-00-r3jrfmn1asdf1}
\end{figure}
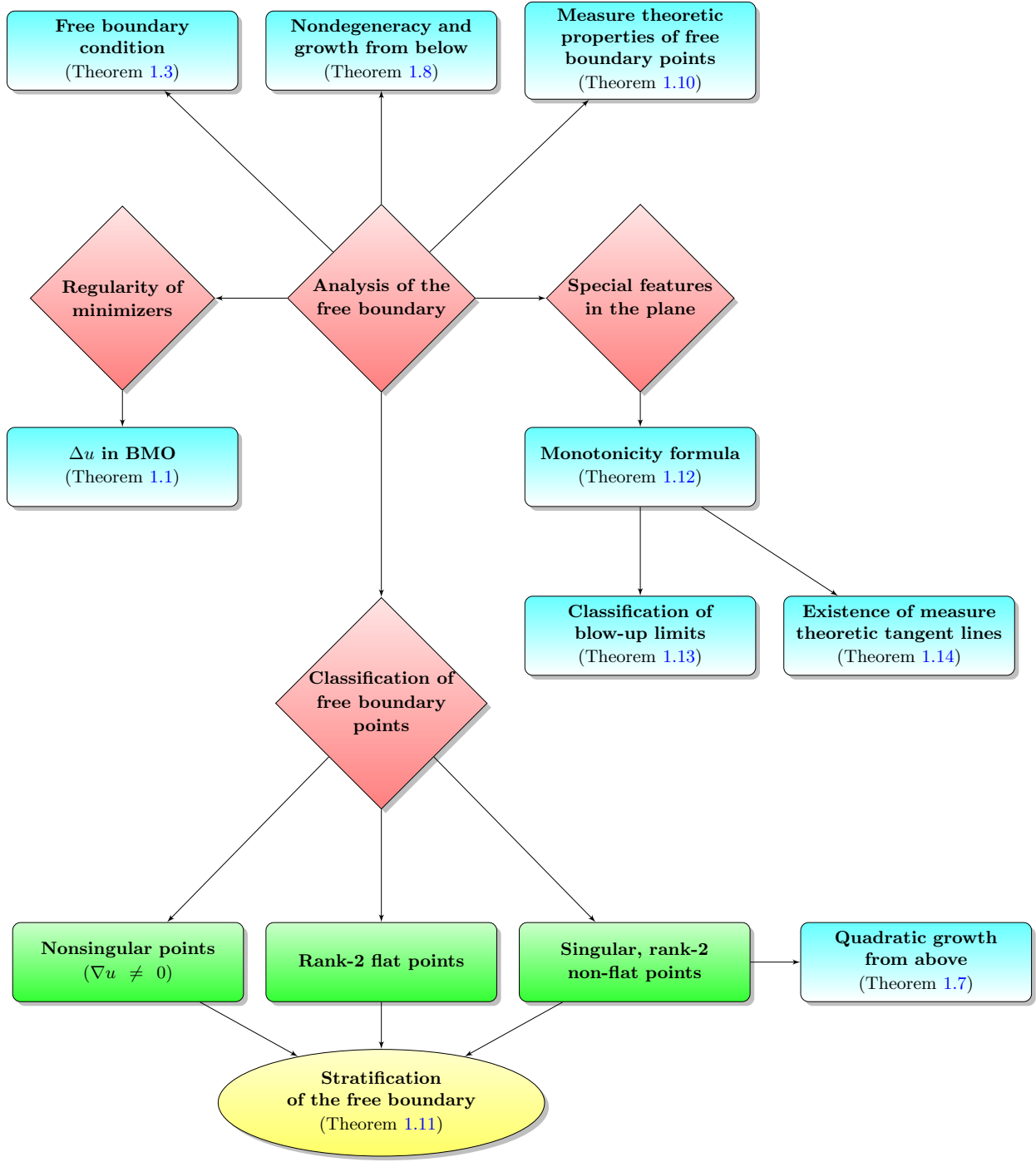

A road map of this article is displayed in Figure~\ref{x345-00-r3jrfmn1asdf1}.

\subsubsection{BMO estimates for the Laplacian of the minimizers,
and free boundary conditions}

In further details, the first regularity result that we establish is a BMO estimate
on the Laplacian of the minimizers. Namely, we prove that:

\begin{theorem}\label{thm-BMO}
Let~$u$ be a minimizer of the functional~$J$ defined in~\eqref{defJ}.
Then, we have that~$\Delta u\in BMO_{{\rm{loc}}}(\Om)$.
\end{theorem}

We also introduce a notion of one-phase minimizer, in the following setting:

\begin{definition}\label{LA:HDEF1ph}
We say that~$u$ is a one-phase minimizer of~$J$
if it minimizes the functional~$J$ in~\eqref{defJ}
among the nonnegative admissible functions~$\{u\in
\mathcal A {\mbox{ s.t. }}u\ge 0{\mbox{ in }}\Omega\}$,
$\mathcal A $ being~as in~\eqref{ADMI}.
\end{definition}

Interestingly, one-phase minimizers, as given in Definition~\ref{LA:HDEF1ph},
arise from a combination of a biharmonic free boundary problem
and an obstacle problem. We also observe that, in general,
minimizers of~$J$ which happen to be nonnegative do not naturally develop
open regions in which the minimizer vanishes (see Proposition~\ref{NOONE} for a concrete result),
while one-phase minimizers do (hence, the notion of
minimizers that are nonnegative and the notions of one-phase minimizers
are structurally very different in this framework, due to the lack of maximum principle).\medskip

We stress that one-phase minimizers, as given
in Definition~\ref{LA:HDEF1ph}, are not necessarily minimizers over $\mathcal A$.
This fact produces significant differences with respect to the classical
case of free boundary problems driven by the Laplacian,
and requires some non-standard techniques to
overcome the lack of structure provided, in the classical case,
by super-harmonic functions.\medskip

We also observe that, in the classical 
Alt-Caffarelli problem~\cite{MR618549} a nonnegative
boundary datum produces, in general, considerable portions
of the domain in which the minimizer vanishes, but in our case
minimizers with nonnegative (and even strictly positive)
boundary data may produce region with considerable negative
phases. This difference between zero and strictly negative phases
is indeed one of the typical features of our problem
and it is also due to the characteristic function in~\eqref{defJ}.
Specifically, the Alt-Caffarelli problem~\cite{MR618549}
with nonnegative datum typically produces large zero phases,
while in most of the situations that one can imagine
our minimizers with nonnegative data have negligible zero sets
(but non-negligible negative sets): the role of
one-phase minimizers in our setting is precisely to
create natural conditions to produce non-negligible
zero sets
(the reader may also consider looking immediately
at the examples in Section~\ref{sec:ex2}
to see these phenomena of zero and negative phases
in simple, but concrete, cases).
\medskip

Given the higher order structure of the biharmonic functional,
the minimizers satisfy a free boundary condition which is richer, and more complicated,
than in the harmonic case. 
To express it in a general form, suppose that
the free boundary (locally) separates two regions, say~$\Omega^{(1)}$
and~$\Omega^{(2)}$, of the domain~$\Omega$, with~$\partial\Omega^{(1)}=
\partial\Omega^{(2)}=\partial\{u>0\}$:
in this case, the minimizer~$u$ can be seen as the result
of the junction of two functions, say~$u^{(1)}$ and~$u^{(2)}$, from each side of
the free boundary, 
with~$u^{(1)}$ and~$u^{(2)}$ not changing
sign.
In this notation, for~$i\in\{1,2\}$, we set
\begin{equation}\label{LAMDEFI}\lambda^{(i)}:=\begin{cases}
1 & {\mbox{ if $u^{(i)}>0$ in $\Omega^{(i)}$}},\\
0 & {\mbox{ if $u^{(i)}\le0$ in $\Omega^{(i)}$}}.
\end{cases}\end{equation}
Then, we have the following result describing the free boundary condition in this framework:

\begin{theorem} \label{FREE BOU COND}
Let~$u$ be either a minimizer 
or a continuous one-phase minimizer of the functional~$J$ defined in~\eqref{defJ}.
Assume that 
\begin{equation}\label{ZEROM}
{\mbox{~$\partial\{|u|>\e\}$ is of class~$C^1$,}}\end{equation} 
for all~$\e\in(0,\e_0)$, for some~$\e_0>0$.
Then,
for any~$\phi=(\phi^1, \dots, \phi^n)\in C^\infty_0(\Omega)$,
\begin{equation}\label{PALA}\begin{split}&
\lim_{\e{{\to}}0}
\int_{\partial(\Omega \cap\{ |u|>\e\})}
\Bigg(
\big(|\Delta u^{(1)}|^2+\lambda^{(1)}\big)\phi\cdot\nu-
2\sum_{m=1}^n \Big(
\phi^m\big(\Delta u^{(1)}\nabla u_m^{(1)}-u_m^{(1)}\nabla\Delta u^{(1)}\big)\cdot\nu
+ \Delta u^{(1)} u_m^{(1)}
\nabla\phi^m\cdot\nu\Big)\Bigg)
\\ &\qquad=
\lim_{\e{{\to}}0}
\int_{\partial(\Omega \cap\{ |u|>\e\})}\Bigg(\big(
|\Delta u^{(2)}|^2+\lambda^{(2)}\big)\phi\cdot\nu-
2\sum_{m=1}^n \Big(
\phi^m\big(\Delta u^{(2)}\nabla u_m^{(2)}-u_m^{(2)}\nabla\Delta u^{(2)}\big)\cdot\nu
+ \Delta u^{(2)} u_m^{(2)}
\nabla\phi^m\cdot\nu\Big)\Bigg)
,\end{split}\end{equation}
where~$\nu$ is the exterior normal to~$\Omega^{(1)}$.

Furthermore, 
if~$u\in C^1(\Omega)\cap C^3\big(\overline{\Omega^{(1)}}\big)\cap C^3\big(\overline{\Omega^{(2)}}\big)$
and~$\partial\{ |u|>\e\}$ approaches $\partial\{ |u|>0\}=\partial\{ u>0\}=\partial\{ u<0\}=\{u=0\}$
in the $C^1$-sense, we have that
\begin{equation}\label{Nu1}
\left\{\quad
\begin{split}
& \Delta u^{(1)} u_m^{(1)}=
\Delta u^{(2)} u_m^{(2)}\\
\\ {\mbox{and }}\qquad&
\big(|\Delta u^{(1)}|^2+\lambda^{(1)}\big)\nu_m-
2\big(\Delta u^{(1)}\nabla u_m^{(1)}-u_m^{(1)}\nabla\Delta u^{(1)}\big)\cdot\nu
\\&\qquad=
\big(|\Delta u^{(2)}|^2+\lambda^{(2)}\big)\nu_m-
2\big(\Delta u^{(2)}\nabla u_m^{(2)}-u_m^{(2)}\nabla\Delta u^{(2)}\big)\cdot\nu
,\end{split}\right.
\end{equation}
for any~$m\in\{1,\dots,n\}$, on~$\partial\{u>0\}$.
\end{theorem}

Concrete examples of this free boundary condition, together with some applications
from mechanics, will be also presented in Sections~\ref{sec:ex2} and~\ref{INTER}
(of course, the reader is welcome
to {\em jump to these examples right away}, before
diving into all the rather technical details of this paper,
if she or he wants to have
immediately
a close-to-intuition approach to
the model and the problems discussed in this paper,
and to develop some feeling on how minimizers may be
expected to look like).\medskip

As already discussed in Subsection~\ref{sibsec:cmp},
one of the principal features of the problem that
we consider in the present work is that it does not share the standard 
properties of its ``sibling'' Alt-Caffarelli problem~\cite{MR618549},
such as non degeneracy, linear growth, etc. 
Moreover, the existing techniques fail because
of the involvement of higher order derivatives.

However, the scale invariance of the functional suggests
that the optimal regularity of $u$ must be~$C^{1,1}$.
This is also supported by the computations
that we have for the one-dimensional case (see Remark~\ref{REMBa}
and the explicit examples in Section~\ref{sec:ex2}).

\subsubsection{Notion of rank-2 flatness, the role played by quadratic
polynomials, and dichotomy arguments}

To study the free boundary points
of the minimizers, it is useful to distinguish
between regular and singular points.
Related to this, suppose that~$x\in \fb u$, then there are two possible cases:
\begin{itemize}
\item $\na u(x) \not =0$, then $\fb u$ is $C^1$ near $x$.
\item $\na u(x)=0$, then we expect $u$ to grow quadratically and the free boundary may
have self-intersections.
\end{itemize}
To analyze these situations,
we introduce the following setting:

\begin{definition}\label{def:sing}
If $x\in \fb u$ and $\na u(x)=0$, then we say that~$x$ is a
singular free boundary point. 
The set of singular points is denoted by~$\fbs u$.
\end{definition}

Clearly the singular points are the most interesting points of the free boundary to 
study.
In order to overcome all the difficulties mentioned in Subsection~\ref{sibsec:cmp}
and
study the regularity of~$u$ and that of the free boundary~$\fb u$, we 
employ a dichotomy argument which was introduced in~\cite{DK}.
The idea is to exploit a suitable notion of ``flatness'' and distinguish
between points where the free boundary is flat and points where it is non-flat,
according to this {\em new notion}.

To this aim, we let 
\begin{equation}\label{defHD}
\HD(A, B):=\max\left\{\sup\limits_{a\in A}\dist(a, B),\; 
\sup\limits_{b\in B}\dist(b, A)\right\}
\end{equation}
be the 
Hausdorff distance of two sets $A$, $B\subset \R^n$. 

We also let $P_2$ be the set of all homogeneous polynomials of degree two, i.e.
\begin{equation}\label{poliu7u65}
P_2:=\left\{p(x)=\sum_{i,j=1}^n a_{ij}x_ix_j, {\mbox{ for any }}
x\in \R^n, {\mbox{ with }} \|p\|_{L^\infty(B_1)}=1 \right\},
\end{equation}
where $a_{ij}$ is a symmetric $n\times n$ matrix. Moreover, given~$p\in P_2$
and~$x_0\in\R^n$, we set~$p_{x_0}(x):=p(x-x_0)$ and 
\begin{equation}\label{spx}
S(p, x_0):=\{x\in \R^n : p_{x_0}(x)=0\}.\end{equation}
We observe that the set $S(p, x_0)$ defined in~\eqref{spx}
is a cone with vertex at~$x_0$, i.e. if~$x\in S(p, x_0)$
then, for every~$t>0$, it holds that~$x_0+t(x-x_0)\in S(p, x_0)$. 

With this notation, we set:

\begin{definition}\label{def1}
Let~$\delta>0$,~$R>0$ and~$x_0\in \fb u$.
We say that~$\fb u$ is $(\delta, R)$-rank-2 flat at~$x_0$
if, for every~$r\in (0, R]$, 
there exists~$p\in P_2$ such that
\[
\HD\Big(\fb u\cap B_r(x_0), S (p, {x_0})\cap B_r(x_0)\Big)< \delta\, r.
\]
\end{definition}

Now, given~$r>0$, $x_0\in\partial\{u>0\}$ and~$p\in P_2$, we let
\begin{equation}\label{defflat:BIS} h_{\rm{min}}(r,x_0,p):=
\HD\Big(\fb u\cap B_r(x_0), S(p, x_0)\cap B_r(x_0)\Big).
\end{equation}
Then, we define the rank-2
flatness at level $r>0$ of $\fb u$ at~$x_0$ as follows.
We set
\begin{equation}\label{flatdef}
h(r, x_0):=\inf_{p\in P_2}h_{\rm{min}}(r,x_0, p)
\end{equation}
and we introduce the following notation:

\begin{definition}\label{def:flat}
Let~$\delta>0$,~$r>0$ and~$x_0\in\partial\{u>0\}$.
We say that $\fb u$ is $\delta$-rank-2 flat at level $r$ at~$x_0$
if~$h(r, x_0)<\delta r$.
\end{definition}

In view of Definitions~\ref{def1} and~\ref{def:flat}, we can say that~$
\fb u$ is $(\delta, R)$-rank-2 flat at~$x_0\in\partial\{u>0\}$
if and only if, for every~$r\in (0, R]$, it is 
$\delta$-rank-2 flat at level $r$ at~$x_0$.\medskip

We stress that the notion of ``flatness'' introduced
in Definitions~\ref{def1} and~\ref{def:flat} do not refer to a geometric
property of being ``close to linear'',
but rather to a proximity to level sets of quadratic polynomial
(that is, from the linguistic perspective, one should not
separate the adjective ``flat'' from its own specification ``rank-2'').
Roughly speaking, our objective is to exploit quadratic
objects to describe the minimizers, and our typical strategy would
be to distinguish between points of the free boundary where
the free boundary itself ``looks like the level set of a quadratic polynomial''
(i.e., it is in some sense rank-2 flat), and the ``other points'' of the free
boundary, proving in the latter case that then it is the minimizer
itself to possess some similarities, in terms of growth, with
``quadratic objects''.
The reason for which we used the terminology
of ``flatness'' to describe these ``quadratic''
(rather than ``linear'') scenarios is to maintain
some jargon coming from the classical case in~\cite{MR618549},
and to interpret the notion of flatness as the one describing
the ``deviation'' from a well-understood case (that is, the linear case
in~\cite{MR618549}, and the quadratic case here).
\medskip

Of course, making precise these results in our setting requires
the development of a rather technical terminology,
and detailed formulations of these ideas will be provided
in Theorems~\ref{growth}, \ref{thm-nondeg},
\ref{thm-Hausdorff} and~\ref{thm-strata}.
\medskip

In this framework, we now state the following result
concerning the quadratic growth of~$u$ at $\delta$-rank-2 non-flat
points of the
free boundary.

\begin{theorem}\label{growth}
Let $n\ge 2$ and $u$ be a minimizer of the functional~$J$ defined in~\eqref{defJ}.
Let~$D\subset\subset\Omega$, $\delta>0$ and let~$x_0\in \fb u\cap D$
such that~$|\na u(x_0)|=0$ and~$\fb u$
is not~$\delta$-rank-2 flat at~$x_0$ at any level~$r>0$.

Then, $u$ has at most quadratic growth at~$x_0$, bounded from above
in dependence on~$\delta$.
\end{theorem}

\subsubsection{Further results on the quadratic growth of
the minimizers}

Now we turn our attention to the nondegeneracy properties of the minimizers.
First of all, setting as usual~$u^+(x):=
\max\{u(x),0\}$, we provide a weak form of nondegeneracy, investigating the validity
of statements of this form:
\begin{equation}\label{ureebj}
{\mbox{if~$B\subset \po u$ is a ball touching~$\fb u$, then
$\sup_B u^+\ge C[\diam(B)]^2$}}
\end{equation}  for some~$C>0$ (possibly depending on dimension, on the domain and
on the datum~$u_0$).

We consider this as a 
weak form of nondegeneracy as opposed to
the one in which $B$ is
centered at free boundary points, which we call strong nondegeneracy.

We establish that \eqref{ureebj} is satisfied, and, more generally, that the
positive density of the positivity set is sufficient to ensure at least quadratic growth
from the free boundary. The precise result that we obtain is the following:

\begin{theorem}\label{thm-nondeg}
Let $u$ be a minimizer of the functional~$J$ defined in~\eqref{defJ}. Then:
\begin{itemize} 
\item[$\bf 1^\circ$] If~$x_0\in \fb  u$ and
\begin{equation}\label{iehnfnb0094}
\liminf_{\rho\to 0}\frac{\big|B_\rho(x_0)\cap \po u\big|}{|B_\rho|}\ge \theta_*
\end{equation}
for some $\theta_*>0$, then 
$$\sup_{B_r(x_0)} |u| \ge \bar c r^2$$
as long as~$B_r(x_0)\subset\subset\Om$,
for some~$\bar c>0$ depending on $\theta_*$,
$n$, $\dist(B_r(x_0),\Om)$ and~$ \|u_0\|_{W^{2, 2}(\Om)}$.
\item[$\bf 2^\circ$] If~$x_0\in\{u>0\}$ and~$r:=\dist(x_0, \fb u)$, then 
$$\sup_{B_r(x_0)}u^+\ge \bar cr^2,$$
as long as~$B_r(x_0)\subset\subset\Om$,
for some~$\bar c>0$ depending on
$n$, $\dist(B_r(x_0),\Om)$ and~$ \|u_0\|_{W^{2, 2}(\Om)}$.
\end{itemize}
\end{theorem}

We observe that the claim in $\bf 2^\circ$ is exactly the statement
in~\eqref{ureebj}. 

Sufficient conditions for the density estimate in~\eqref{iehnfnb0094} to hold will
be discussed in Subsection~\ref{WHYWTH}, where we also recall and compare
the notions of weak $c$-covering condition and Whitney's covering.
In addition, in Subsection~\ref{BIHAMOE}
we will relate the nondegeneracy properties 
with a fine analysis of the biharmonic measure, which in turn produces
some regularity results on the free boundary.
\medskip

It is also convenient to consider ``vanishing'' free boundary points,
in the following sense:

\begin{definition}\label{def-vanish}
Let $u$ be a minimizer of the functional~$J$ defined in~\eqref{defJ}
and let~$x_0\in \fb u \cap B_1$.
We say that $\fb u$ is vanishing rank-2
flat at~$x_0$ if there exist 
sequences~$\delta_k\to 0$ and~$r_k\to 0$ such that
\begin{equation}\label{gerghgh068} h(r_k,x_0)\le \delta_k r_k,\end{equation}
where~$h$ is defined in~\eqref{flatdef}.
\end{definition}

Notice, in particular, that condition~\eqref{gerghgh068} is equivalent to
$$ \lim_{k\to+ \infty}\frac{ h(r_{k}, x_0)}{r_k}=0,$$
and this justifies the name of ``vanishing'' in Definition~\ref{def-vanish}.

Then, we have:

\begin{theorem}\label{thm-Hausdorff} 
Let~$u$ be a minimizer of the functional~$J$ defined
in~\eqref{defJ}. Then:
\begin{itemize}
\item[$\bf 1^\circ$] 
The set
of vanishing rank-$2$ flat points of the free boundary
has zero measure in $\Omega$.

\item[$\bf 2^\circ$] If $D\subset\subset \Om$ and there exists~$\bar c>0$
such that 
\begin{equation}\label{9eybn-sdfsn}\liminf_{r\to 0}\frac{\sup_{B_r(x)}|u|}{r^2}\ge \bar c
\end{equation}
for every $x\in \fb u\cap \overline D$,
then $\fb u$ has zero measure, and for any~$\delta>0$,
the set of free boundary points
that are not
$\delta$-rank-2 flat
has finite $(n-2)$-dimensional Hausdorff measure.
\end{itemize}
\end{theorem}

In general, we can restate the previous results
in a dichotomy form: roughly speaking, the free boundary
in the vicinity of singular points
is either ``flat'' with respect to the level sets
of homogeneous polynomial of degree two, being ``close'' to the level sets of quadratic polynomials,
or ``non-flat'' and in this case the growth from the free boundary is quadratic.
To formalize these notions,
we decompose the class $P_2$ introduced in~\eqref{poliu7u65}
as $$ P_2=\bigcup_{i=1}^n P_2^i,$$
where 
$$P_2^i:=\left\{p\in P_2 : \text{Rank} (D^2 p)=i\right\} .$$ 
As we will see, in our setting, the above notion will play
a useful role since if $x_0\in \fb u$, with
$|\na u(x_0)|=0$, and $\partial\{ u>0\}$ is rank-2 flat at $x_0$,
then there exists $p\in P_2$ 
such that the blow-up of 
$\fb u$ at $x_0$ is the zero set of $p$. 
We separate out some interesting cases:
\begin{itemize}
\item If $\text{Rank}(D^2 p)=n$ and $D^2p\ge 0$ then the free boundary is a singleton.
\item If $\text{Rank}(D^2p )=1$ then the free boundary is a hyperplane in $\R^n$, i.e. 
a codimension 1 plane in $\R^n$ and after some rotation of coordinates we can write 
$p(x)=\alpha(x_1^+)^2$, where $\alpha\in \R$ is a normalizing constant.
\item If $\text{Rank}(D^2 p)=n$ and $D^2p$ has eigenvalues of 
opposite signs   then the free boundary has self intersection. For instance, 
if $n=2$ then $p(x)=\alpha(x_1^2-x_2^2)$, where~$\alpha\in \R$ is a normalizing constant.  
\end{itemize}
Roughly speaking, in this setting
the classes $P_2^i$ detect the approximate symmetries of the free boundary at small scales.
\medskip

Now, we let $\mathcal F\subseteq\fbs u$ be the set of singular
free boundary points that are vanishing rank-2 flat.
Let also
\begin{eqnarray*} \mathcal N&:=&\big(\fb u\setminus \mathcal F\big)\cap\left\{|\nabla u|=0 \right\}\\&=&
\fbs u\setminus\mathcal{F}.\end{eqnarray*}
In this framework, the main result in the stratification setting reads as follows:

\begin{theorem}\label{thm-strata}Let~$u$ be a minimizer of~$J$.
We have that
\begin{itemize}
\item for any~$z\in \mathcal F$, there exist $r_k{{\to}} 0$ and $p\in P_2^i$, for some~$i\in
\{1,\dots, n\}$,
such that \begin{equation}\label{FOR1}
\lim_{k\to+ \infty}\HD\big((\partial E_k) \cap B_R, \{p=0\}\cap B_R\big)=0\end{equation}
for every fixed $R>0$,
where $$E_k:=\left\{x\in \R^n : z+r_kx\in \po u\right\}.$$
Furthermore, $u^+$ is strongly nondegenerate at $z$, namely
$$ \sup_{B_r(z)} u^+\ge cr^2,$$
for some~$c>0$, as long as~$B_{r}(z)\subset\subset\Omega$, with~$c$ possibly depending on~$n$,
$\dist(z,\partial\Omega)$ and~$u$;
\item for any~$z\in \mathcal N,$ there exists~$C_z>0$,
possibly depending on~$n$,
$\dist(z,\partial\Omega)$ and~$\|u\|_{W^{2,2}(\Omega)}$,
such that 
\begin{equation}\label{FOR2} |u(x)|\le C_z|x-z|^2\end{equation} near $z$. 
\end{itemize}
\end{theorem}

\subsubsection{Monotonicity formula, and classification of blow-up limits}

To analyze and classify the free boundary properties of the minimizers
of~$J$ and their blow-up limits, it would be extremely desirable to have suitable
monotonicity formulas. Differently from the classical case,
in our setting no general result of this type is available in the literature.
To overcome this difficulty, we focus on the two-dimensional case,
for which we prove that:

\begin{theorem}\label{lemma:F}
Let~$n=2$ and~$\tau>0$ such that~$B_\tau\subset\subset\Omega$.
Let~$u:\Omega\to\R$, with~$0\in\partial\{ u>0\}$ and~$\nabla u(0)=0$,
be 
\begin{itemize}
\item either: a minimizer of the functional~$J$,
with~$0$ not $(\delta, \tau)$-rank-$2$ flat in
the sense of Definition~\ref{def1},
\item or: a one-phase minimizer of the functional~$J$
with~$u\in C^{1,1}(\Omega)$,
and such that~$\partial\{u>0\}$ has null Lebesgue measure. 
\end{itemize}
Then, there exists a function~$E:(0,\tau)\to\R$, which is bounded,
nondecreasing and
such that, for any~$\tau_2>\tau_1>0$,
\begin{equation}\label{MONOFORMULA}
E(\tau_2)-E(\tau_1)=
\int_{\tau_1}^{\tau_2}\left\{
\frac1{r^2}\int_{\partial B_r}
\left[\left( \frac{u_{\theta r}}{r}-\frac{2u_\theta}{r^2}\right)^2
+\left( u_{rr}-\frac{3u_r}{r}+\frac{4u}{r^2}\right)^2\right]
\right\}\,dr.
\end{equation}
The explicit value of the function~$E$ is given by
\begin{equation}\label{EMMEDE}
E(r) = 
\int_{\partial B_r}
\left(\frac{\Delta u\,u_r}{2r^2}-
\frac{5u_r^2}{2r^3}
-\frac{\Delta u u}{r^3}+\frac{6u u_r}{r^4}+
\frac{u_{\theta}u_{\theta r}}{r^4}
-\frac{4u^2}{r^5}
-\frac{3u_\theta^2}{2r^5}
\right)
+
\frac1{4r^2}\int_{B_r} \big( |\Delta u|^2+\chi_{\{u>0\}}\big)
.\end{equation}
Furthermore, if~$E$ is constant in~$(0,\tau)$,
then~$u$
is a homogeneous function of degree two in~$B_\tau$. 
\end{theorem}


We stress that the $C^1$
assumption on~$u$ in Theorem~\ref{lemma:F}
is taken only in the case of one-phase minimizers,
while for minimizers no additional regularity assumption is required
in Theorem~\ref{lemma:F}.\medskip

Given~$x_0\in\fb u$ 
we consider the blow-up
sequence of~$u$ at~$x_0$, defined as 
\begin{equation}\label{blow}
u_k(x):=\frac{u(x_0+\rho_k x)}{\rho_k^2},\end{equation}
where~$\rho_k\to0$ as~$k\to+\infty$.

In this setting, we can classify blow-up limits of minimizers in the plane,
according to the following result:

\begin{theorem}\label{thm-hom-blow}
Let~$n=2$.
Let $B_r\subset\subset\Omega$. Let~$x_0\in\Omega$
and~$u:\Omega\to\R$, with~$x_0\in \fbs u$.

Assume that either~$u$
is a minimizer of the functional~$J$, with
\begin{equation}\label{34934jsq92358858678}
{\mbox{$\fb u$ not $\delta$-rank-2 flat at $x_0$
at any level,}}\end{equation} for some~$\delta>0$, or that
$u$ is a one-phase minimizer of the functional~$J$
with~$u\in C^{1,1}(\Omega)$,
and such that~$\partial\{u>0\}$ has null Lebesgue measure.

Then every blow-up limit of $u$ at $x_0$ is
either a homogeneous function of degree two, or it is identically zero.
\end{theorem}

One of the main issues in the free boundary analysis is
that, even in the one-phase problem, the topological and measure theoretic boundaries of 
$\po u$ may not coincide. 
On the other hand, following is a regularity result for the
one-phase free boundary in the plane:

\begin{theorem}\label{BYPA} Let $n=2$.
Suppose that~$B_1\subset\subset\Omega$.
Assume that~$u$ is a one-phase
minimizer for~$J$, that
\begin{equation}\label{7uhn7yhnb7yhb02394}
u\in C^{1,1}(B_1),\end{equation}
and that~$\partial\{u>0\}$ has null Lebesgue measure.

Suppose that $0\in \fbs u$.
Assume also
that, for every~$\bar x\in\fb u\cap B_1$,
\begin{equation}\label{SUPu} 
\liminf_{\rho\to0^+}\frac{ \sup_{B_\rho(\bar x)} u}{\rho^2}\ge c,
\end{equation}
for some~$c>0$, for all~$\rho\in(0,1)$,
and that
\begin{equation}\label{LIMSUP-0}
\limsup_{\rho\to 0}\frac{|B_\rho\cap \po u|}{|B_\rho|}<1.\end{equation}
Then there exists $r_0>0$ such that at  every
point $\bar x$ of~$ \fb u\cap B_{r_0}$ the free boundary possesses
a unique approximate tangent line in measure theoretic sense, namely
if~$D$ is the symmetric difference of the sets~$\{u>0\}$ and
a suitable rotation of~$\{ (x-\bar x)\cdot e_1>0\}$, we have that
$$ \lim_{\rho\to0^+} \frac{|B_\rho(\bar x)\cap D|}{|B_\rho(\bar x)|}=0.$$
\end{theorem}

We think that it is an interesting open problem to detect
suitable conditions guaranteeing that
the $C^{1,1}$-assumptions
taken in Theorems~\ref{lemma:F}, \ref{thm-hom-blow} and~\ref{BYPA}
are fulfilled.\medskip

Moreover, in our setting, Theorems~\ref{thm-BMO},
\ref{growth}, \ref{thm-nondeg}, \ref{thm-Hausdorff}
and~\ref{thm-strata} are obtained specifically for the minimizers,
and Theorem~\ref{BYPA} specifically for the one -phase
minimizers, while Theorems~\ref{FREE BOU COND},
\ref{lemma:F} and~\ref{thm-hom-blow} are valid for
both minimizers and one -phase
minimizers. Though the minimization setting is, in our case,
structurally different from that of one-phase
minimization, due to the lack of Maximum Principle,
we think that it is an interesting open problem
to unify as much as possible the theory of
minimizers with that of one-phase minimizers.

It is also an interesting problem to detect the optimal regularity
of the solutions and of their free boundaries.

\subsection{Organization of the paper}

The rest of the paper is organized as follows.
Section~\ref{sec:ex} contains the main existence result. 
In Section~\ref{sec:BMO} we provide the proof of the local BMO
estimate for the Laplacian of the minimizers, as given by Theorem~\ref{thm-BMO}.

In Section~\ref{sec:first} we present some structural properties
of the minimizers which are based on the first variation of
the functional~$J$.
As a consequence, we also obtain the free boundary condition
and we prove Theorem~\ref{FREE BOU COND}.

In Section~\ref{sec:ex2}, we discuss some one-dimensional examples, 
and in Section~\ref{INTER} we provide a mechanical interpretation
of the free boundary condition.

Section~\ref{sec:dic} contains a dichotomy argument
which leads to the proof of Theorem~\ref{growth}.

Section~\ref{sec-nondeg} is devoted to 
nondegeneracy considerations and to the proof of
Theorems~\ref{thm-nondeg} and \ref{thm-Hausdorff}.

In Section~\ref{sec:stra} we consider the stratification
of the free boundary, reformulating some results obtained in
Section~\ref{sec:dic},
and, in particular, we prove Theorem~\ref{thm-strata}.

Section~\ref{PF:MO}
focuses on the monotonicity formula
and contains the proof of
Theorem~\ref{lemma:F}.

In Section~\ref{sec:class}
we present an application of such a monotonicity
formula, proving the homogeneity of the blow-up limits, and
establishing Theorem~\ref{thm-hom-blow}.

Then, Section~\ref{sec:reg} focuses on explicit
two-dimensional regularity and classification results
and contains the proof of Theorem~\ref{BYPA}.

The paper ends with two appendices which collect some ancillary observations.

\section{Existence of minimizers}\label{sec:ex}

The following result exploits the direct method
of the calculus of variations to obtain the existence
of the minimizers for our problem. Due to the presence
of several technical aspects in the proof, we provide the argument in 
full details:

\begin{lemma}\label{existence}
The functional in~\eqref{defJ} attains a minimum over~${\mathcal{A}}$.
\end{lemma}

\begin{proof}
Let $u_k\in{\mathcal{A}}$ be a minimizing sequence, namely
\begin{equation}\label{ElX}
\lim_{k\to+\infty}J[u_k]=\inf_{ v\in{\mathcal{A}}}J[v].\end{equation}
For large~$k$, we can suppose that
\begin{equation}\label{El0} J[u_k]\le J[u_0]+1\le \int_\Omega (|\Delta u_0|^2+1)\le C,\end{equation}
for some~$C>0$. Also,
since~$u_k\in{\mathcal{A}}$, we know from~\eqref{ADMI}
that~$u^*_k := u_k-u_0\in W^{2,2}(\Omega)\cap W^{1,2}_0(\Omega)$.
Let also~$v^*_{k}:=\Delta u^*_{k}\in L^2(\Omega)$.
In this way, we have that
$$ \left\{\begin{matrix}
\Delta u^*_{k}=v^*_{k} & {\mbox{ in }}\Omega,
\\ u^*_k=0 & {\mbox{ on }}\partial\Omega.
\end{matrix}
\right. $$
Consequently,
by elliptic regularity (see Theorem~4 on page~317 of~\cite{EVANS98})
we know that
\begin{equation}\label{El1} \| u^*_{k}\|_{W^{2,2}(\Omega)}\le C'\,\big(
 \| v^*_{k}\|_{L^{2}(\Omega)}+
 \| u^*_{k}\|_{L^{2}(\Omega)}
\big),\end{equation}
for some~$C'>0$.
Also (see Theorem~6 on page~306 of~\cite{EVANS98}),
one has that
\begin{equation}\label{El2} \| u^*_{k}\|_{L^2(\Omega)}
\le C''\, \| v^*_{k}\|_{L^{2}(\Omega)},\end{equation}
for some~$C''>0$.
Therefore, in light of~\eqref{El1} and~\eqref{El2} we conclude that
$$ \| u^*_{k}\|_{W^{2,2}(\Omega)}\le C'''\,\| v^*_{k}\|_{L^{2}(\Omega)}=
C'''\,\| \Delta u^*_{k}\|_{L^{2}(\Omega)}$$
for some~$C'''>0$. This and~\eqref{El0} imply that
$$ \| u^*_{k}\|_{W^{2,2}(\Omega)}\le C''''$$
for some~$C''''>0$. Therefore, we can suppose, up to a subsequence,
that
\begin{equation}\label{El4}
{\mbox{$u_k^*$ converges to some~$u^*$ weakly in~$W^{2,2}(\Omega)$}}\end{equation}
and then, by compact embedding, 
\begin{equation}\label{El4BIS}
{\mbox{$u_k^*$ converges
strongly to~$u^*$ in~$W^{1,2}(\Omega)$.}}\end{equation}
Since~$u_k^*\in W^{1,2}_0(\Omega)$, this implies that
also~$u^*\in W^{1,2}_0(\Omega)$. As a consequence, recalling~\eqref{ADMI},
we know that
\begin{equation}\label{qweurfcbxzq1idhfer}
{\mbox{$u:=u^*+u_0$ belongs to~${\mathcal{A}}$}}.
\end{equation}
Furthermore, by~\eqref{El4}, it holds that
$u_k$ converges to~$u$ weakly in~$W^{2,2}(\Omega)$.
In particular, $u_k$ is bounded in~$W^{2,2}(\Omega)$
and therefore, for any~$i\in\{1,\dots,n\}$,
it holds that~$\partial^2_i u_k$ is bounded in~$L^2(\Omega)$.
This yields that~$\partial^2_i u_k$ converges to some~$w_i$
weakly in~$L^2(\Omega)$.
This and 
\begin{equation}\label{El4TRIS}
{\mbox{the strong convergence of~$u_k$ to~$u$ in~$
W^{1,2}_0(\Omega)\subset L^2(\Omega)$}}\end{equation}
(recall~\eqref{El4BIS})
imply that, for any~$\varphi\in C^\infty_0(\Omega)$,
$$
\int_\Omega w_i \, \varphi=
\lim_{k\to+\infty}\int_\Omega \partial^2_i u_k\, \varphi
=
\lim_{k\to+\infty}\int_\Omega u_k\, \partial^2_i\varphi
=\int_\Omega u\, \partial^2_i\varphi,$$
which shows that~$w_i=\partial^2_i u$.

Accordingly, we have that~$
\partial^2_i u_k$ converges to~$\partial^2_i u$
weakly in~$L^2(\Omega)$.
Therefore, we have that
\begin{equation}\label{0912eyru}
\begin{split}&
0\le\lim_{k\to+\infty}\int_\Omega |\Delta(u_k-u)|^2
=
\lim_{k\to+\infty}\int_\Omega |\Delta u_k|^2
+\int_\Omega |\Delta u|^2-2
\int_\Omega \Delta u_k\Delta u\\&\qquad\qquad=
\lim_{k\to+\infty}\int_\Omega |\Delta u_k|^2-\int_\Omega |\Delta u|^2.\end{split}
\end{equation}
Now, up to a subsequence, recalling~\eqref{El4TRIS},
we can suppose that~$u_k$ converges to~$u$ a.e. in~$\Omega$
and therefore
$$ \liminf_{k\to+\infty} \chi_{\{u_k>0\}}\ge
\chi_{\{u>0\}}$$
a.e. in~$\Omega$. Consequently, by Fatou Lemma,
$$ \liminf_{k\to+\infty} \int_\Omega\chi_{\{u_k>0\}}\ge\int_\Omega
\chi_{\{u>0\}}.$$
Combining this with~\eqref{0912eyru}, we see that~\eqref{ElX} gives that
$$ J[u]\le \liminf_{k\to+\infty}J[u_k]=\inf_{ v\in{\mathcal{A}}}J[v].$$
This and~\eqref{qweurfcbxzq1idhfer} imply that~$u$
is the desired minimizer.
\end{proof}

By taking into account a nonnegative constraint in the minimizing sequence
in the proof of Lemma~\ref{existence}, one also obtains an existence result
for the one-phase problem.

\section{BMO estimates and proof of Theorem~\ref{thm-BMO}}\label{sec:BMO}

The goal of this section is to show that the minimizers
of~\eqref{defJ} have a Laplacian which is a function of locally bounded
mean oscillation, and thus prove Theorem~\ref{thm-BMO}. 

\begin{proof}[Proof of Theorem~\ref{thm-BMO}]
We fix~$R_0> R>  r>0$ and~$x_0\in\Omega$ such that the ball~$
B_{{2R_0}}(x_0)$ is contained in~$\Om$, and we
consider the function~$h$ that solves
\[
\left\{
\begin{array}{lll}
\Delta^2 h=0 &\mbox{in}\ \ B_{{2R}}(x_0),\\
h=u &  {\mbox{ on }}\partial B_{{2R}}(x_0),\\
\nabla h=\nabla u &  {\mbox{ on }}\partial B_{{2R}}(x_0).
\end{array}
\right.
\]
The existence of $h$ follows from the Green's formula for biharmonic functions,
see page 48 in~\cite{GAZ}, or by minimizing energy with
\begin{equation}\label{BOBB}
h-u\in W^{2,2}_0(B_{{2R}}(x_0)). 
\end{equation}
We also extend~$h$ outside~$B_{{2R}}(x_0)$ to be equal to~$u$ in~$\Omega\setminus B_{2R}(x_0)$.
We observe that the function~$h$ is an admissible competitor for~$u$, since
\begin{equation}\label{BRAV}
h\in W^{2,2}(\Omega).
\end{equation}
Indeed, if~$v:=h-u$, we see from~\eqref{BOBB}
and the extension results in classical Sobolev spaces (see e.g. Proposition~IX.18
in~\cite{MR697382}) that~$v\in W^{2,2}(\Omega)$. Since~$u\in W^{2,2}(\Omega)$,
the claim in~\eqref{BRAV} plainly follows.

Then, by the minimality of~$u$, we have that~$J[u]\le J[h]$, that is
\[
\int _{B_{{2R}}(x_0)}\left| \Delta u\right| ^{2}+\I u\leq \int _{B_{{2R}}(x_0)}\left|
\Delta h\right| ^{2}+\I h,
\]
which in turn yields
\begin{equation}\label{eiy584gb}
\int _{B_{{2R}}(x_0)}\left| \Delta u\right| ^{2}-\left|
\Delta h\right| ^{2}\leq C{{R}}^{n},
\end{equation}
for some~$C>0$.
Also, by \eqref{BOBB}, and since~$\Delta^2 h=0$ in~$B_{{2R}}(x_0)$,
we get 
\begin{eqnarray*}
\int _{B_{{2R}}(x_0)}\left| \Delta u\right| ^{2}-\left| \Delta h\right| ^{2}
&=&\int_{B_{{2R}}(x_0)}(\Delta u-\Delta h)(\Delta u+\Delta h)\\
&=&\int_{B_{{2R}}(x_0)}(\Delta u -\Delta  h)\Delta u \\
&=&\int_{B_{{{2R}}}(x_0)}|\Delta u -\Delta h|^2.
\end{eqnarray*}
{F}rom this and~\eqref{eiy584gb}, we obtain that
\begin{equation}\label{eiy584gb:2}
\int _{B_{{2R}}(x_0)}\left| \Delta u-\Delta h\right| ^{2}\leq CR^{n}.
\end{equation}
Now we introduce the notation
$$ (\Delta u)_{x_0,r}:=\fint_{B_r(x_0)} \Delta u(x)\,dx,$$
and we observe that, by H\"older's inequality,
$$
\left| (\Delta u)_{x_0,r}- (\Delta h)_{x_0,r}\right|^2
\le \left( \fint_{B_r(x_0)}|\Delta u-\Delta h|\right)^2
\le \fint_{B_r(x_0)}|\Delta u-\Delta h|^2
$$
which implies that
\begin{equation}\label{eog8yu437}
\int_{B_r(x_0)}|(\Delta h)_{x_0, r}
-(\Delta u)_{x_0, r}|^2\le \int_{B_r(x_0)}|\Delta u -\Delta h|^2.
\end{equation}
Moreover, since the function~$H:=\Delta h $ is harmonic
in~$B_{{2R}}(x_0)$, we have the following Campanato type estimate: 
there exists~$\alpha>0$ and a universal constant~$C>0$ such that
\[
\fint_{B_r(x_0)}|\Delta h -(\Delta h)_{x_0, r}|^2
\le C\left(\frac{r}R\right)^\alpha\fint_{B_R}|\Delta h
-(\Delta h)_{x_0, R}|^2,
\] 
see e.g. 
formula~(1.13) on page~96 in~\cite{MR717034}
(see also the notation on page~92 there).

Hence, using also the triangle inequality and
recalling~\eqref{eiy584gb:2} and~\eqref{eog8yu437}, 
\begin{equation}\label{ALS12678ghj88634oqpe}
\begin{split}
&\int_{B_r(x_0)}|\Delta u -(\Delta u)_{x_0, r}|^2
\\=\;&
\int_{B_r(x_0)}|\Delta u -\Delta h+\Delta h-(\Delta h)_{x_0, r}+
(\Delta h)_{x_0, r}-(\Delta u)_{x_0, r}|^2\\
\le\;& C\left(
\int_{B_r(x_0)}|\Delta u -\Delta h|^2+\int_{B_r(x_0)}|
\Delta h-(\Delta h)_{x_0, r}|^2+\int_{B_r(x_0)}|(\Delta h)_{x_0, r}
-(\Delta u)_{x_0, r}|^2\right)\\
\le\;& 
C\left({{R}}^n+\left(\frac{r}R\right)^{\alpha+n}\int_{B_R(x_0)}
|\Delta h -(\Delta h)_{x_0, R}|^2\right)\\
=\;&
C\left({{R}}^n+\left(\frac{r}R\right)^{\alpha+n} \int_{B_R(x_0)}
|\Delta h -\Delta u+\Delta u-(\Delta u)_{x_0,R}
+(\Delta u)_{x_0, R}-(\Delta h)_{x_0, R}|^2\right)\\
\le\;&
C\left[{{R}}^n+\left(\frac{r}R\right)^{\alpha+n} \left(\int_{B_R(x_0)}
|\Delta h -\Delta u|^2+\int_{B_R(x_0)}
|\Delta u-(\Delta u)_{x_0, R}|^2
+\int_{B_R(x_0)}|(\Delta u)_{x_0, R}-(\Delta h)_{x_0, R}|^2\right)\right]\\
\le\;&
C\left[{{R}}^n+\left(\frac{r}R\right)^{\alpha+n}\left(\int_{B_R(x_0)}
|\Delta h -\Delta u|^2
+\int_{B_R(x_0)}|\Delta u-(\Delta u)_{x_0, R}|^2\right)\right]\\
\le\;&
C\left[{{R}}^n+\left(\frac{r}R\right)^{\alpha+n}\left(R^n
+\int_{B_R(x_0)}|\Delta u-(\Delta u)_{x_0, R}|^2\right)\right]\\ \le\;&
C\left[{{R}}^n+\left(\frac{r}R\right)^{\alpha+n}
\int_{B_R(x_0)}|\Delta u-(\Delta u)_{x_0, R}|^2\right].
\end{split}\end{equation}
We can therefore exploit Lemma~2.1
in Chapter 3 on page~86 of~\cite{MR717034}
(see also Lemma~3.1 in~\cite{DK} and
Theorem~1.1 in~\cite{selfdriven}), 
used here with
\begin{eqnarray*}
&&\phi(\rho):=\int_{B_\rho(x_0)}|\Delta u -(\Delta u)_{x_0, \rho}|^2,
\\&&\beta:=n,\\&&a:=\alpha+n,\\&&\e:=0,\\&&A:=C\\ {\mbox{and }}
&&B:=C
\end{eqnarray*}
thus writing~\eqref{ALS12678ghj88634oqpe} in the form
$$ \phi(r)\le
C\left[{{R}}^\beta+\left(\frac{r}R\right)^{a}\phi(R)\right]=
A\left[\left(\frac{r}R\right)^{a}+\e\right]\phi(R)+BR^\beta
, $$
and hence deducing that
\begin{equation*} \phi(r)\le C\left[\left(\frac{r}R\right)^{\beta}\phi(R)+r^\beta\right],\end{equation*}
up to renaming constants, that gives that
\begin{equation}\label{90-90-9024hchi457}
\int_{B_{{r}}(x_0)}|\Delta u-(\Delta u)_{x_0, {{r}}}|^2\le C{{r}}^n,
\end{equation}
for a suitable~$C>0$, possibly depending on~$u$, $x_0$, $R_0$,
which gives the desired result and finishes
the proof of Theorem~\ref{thm-BMO}.
\end{proof}

\section{First variation of $J$, free boundary condition,
and proof of Theorem~\ref{FREE BOU COND}}\label{sec:first}

In this section,
we consider the first variation of the functional in~\eqref{defJ}.
Of course, the main problem is to take into account variations
performed by a test function whose support intersects the free
boundary of~$u$, since in this case the lack of regularity
of the characteristic function plays an important role.
Therefore, it is useful to know that 
the set $\po u$ is an open subset of $\Om$,
which, in the case of minimizers, follows from the fact
that
\begin{equation}\label{3287uUUSp}
{\mbox{$u\in C^{1,\alpha}_{\rm loc}(\Omega)$ for any~$\alpha\in(0,1)$, }}\end{equation} 
which, in turn, follows from the fact that
\begin{equation}\label{wduep}
{\mbox{$u\in W^{2,p}_{\rm loc}(\Omega)$
for any~$p\in(1,+\infty)$,}}\end{equation} 
in virtue of Theorem~\ref{thm-BMO} and the 
Calder\'on-Zygmund
regularity theory (we think that
it is an interesting open problem to establish whether~\eqref{3287uUUSp} and~\eqref{wduep} are also fulfilled by
one-phase minimizers).\medskip

The main structural properties of the minimizers which are based
on the first variation of the functional are given by the following result:

\begin{lemma}\label{POBIA} Let~$u$ be a minimizer of~$J$. Then
$u$ is weakly super-biharmonic in $\Om$ (i.e.~$\Delta^2u\le0$ in the sense of distributions)
and
biharmonic in $\po u\cup\{u< 0\}^\circ$, where 
$E^\circ$ denotes the interior of $E$.

Similarly, if~$u$ is a one-phase minimizer of~$J$
and~$B$ is an open ball contained in~$\{ u\ge a\}$, with~$a>0$,
then~$u$ is
biharmonic in $B$.
\end{lemma}

\begin{proof} We prove the claims assuming that~$u$ is a minimizer
(the one-phase problem can be treated similarly).
Define $u_\e:=u-\e \phi$, where $0\le \phi\in W^{2,2}(\Om)\cap W^{1,2}_0
(\Omega)$
and~$\e$ is a small parameter to be fixed below. Using the comparison  
of the energies of~$u$ and~$u_\e$, and recalling \eqref{defJ},
we get 
\[
\int_\Om\Big(|\Delta u|^2-|\Delta u-\e \Delta \phi|^2\Big)
\le \int_\Om\Big(\I{u-\e\phi}-\I u\Big).
\]
Note that $\{u-\e\phi>0\}\subset \po u$, provided that~$\e>0$.
Consequently, we have that 
\begin{equation}\label{58hgspgm}
0\ge
\int_\Om\Big(|\Delta u|^2-|\Delta u-\e \phi|^2\Big)=
2\e \int_\Om \Delta u\,\Delta \phi-\e^2\int_\Om(\Delta u)^2.
\end{equation}
Dividing both sides of the last inequality by~$\e>0$ and then letting~$\e\to 0$,
we get that 
\begin{equation*}
\int_\Om\Delta u\,\Delta \phi\le 0.
\end{equation*}
If we take~$\phi\in C^\infty_0(\Omega)$, this gives that~$u$ is
super-biharmonic.
In addition,
if we suppose that~$\supp\phi\subset \po u$, then from~\eqref{58hgspgm}
we deduce, without any sign assumption on~$\e$, that
$$ \int_\Om\Delta u\,\Delta \phi=0,$$
which completes the proof of Lemma~\ref{POBIA}.
\end{proof}

Concerning the statement of
Lemma~\ref{POBIA},
it is interesting to remark that one-phase minimizers
are not necessarily super-biharmonic (an explicit counterexample to this fact is discussed
on page~\pageref{PAGP}).

The basic analytic
structure of the minimizers is then completed by the following result:

\begin{corollary}\label{lem-subham} Let~$u$ be a minimizer of~$J$. 
For every bounded subdomain~$\Omega'\subset \subset \Om$,
there exists~$C>0$, depending only on~$n$,
such that
$$\Delta u\ge -\frac{C\,\|\Delta u\|_{L^1(\Omega)}}{
\big(\dist(\Om', \p \Om)\big)^n} \quad {\mbox{ in }}\Om'.$$
\end{corollary}

\begin{proof} 
Let $r:=\frac12\dist(\Om', \p \Om)$ and, for all~$y\in\Om'$, define the function
\[
\phi(y):=\fint_{B_r(y)}\Delta u(x)\,dx.
\]
Thanks to~\eqref{wduep}, we see that~$\phi$
is continuous on the compact set~$\overline{\Om'}$.
Therefore, there exists~$y_0\in \overline{\Om'}$
such that~$\min_{\overline{\Om'}}\phi(y)=\phi(y_0)$.
Then, for any~$y\in\Omega'$,
\begin{equation}\label{J123e2ef4Ah345d}
\phi(y)\ge\phi(y_0)\ge-\fint_{B_r(y_0)}|\Delta u(x)|\,dx\ge
-\frac{\| \Delta u\|_{L^1(\Om)}}{|B_r|}.\end{equation}
As a consequence, since~$u$ is super-biharmonic, thanks to Lemma~\ref{POBIA},
we obtain the desired estimate by the mean value inequality for weak subsolutions
of the Laplace equation (see e.g.~\cite{MR2906766}
and~\cite{MR177186}). More precisely, if~$v$ is weakly super-harmonic in~$\Om$,
we know from 
Theorem~A in~\cite{MR177186} that there exists a sequence of smooth super-harmonic
functions~$v_h$ in~$\Om'$ that converge to~$v$ a.e. in~$\Om'$ and in~$L^1(\Om')$.
Consequently, a.e. $y\in\Om'$,
\begin{equation}\label{SUODFer} v(y)=\lim_{h\to 0}v_h(y)\ge\lim_{h\to 0}
\fint_{B_r(y_0)} v_h(x)\,dx=\fint_{B_r(y_0)} v(x)\,dx.\end{equation}
Then, choosing~$v:=\Delta u$ and applying~\eqref{J123e2ef4Ah345d}, we find that
$$\Delta u(y)\ge
\fint_{B_r(y)}\Delta u(x)\,dx=
\phi(y)\ge
-\frac{\| \Delta u\|_{L^1(\Om)}}{|B_r|}
,$$
as desired.
\end{proof}

For the sake of completeness, we observe that the statement of
Corollary~\ref{lem-subham} can be strengthen by showing,
under additional regularity assumptions, that minimizers
are super-harmonic, according to next result:

\begin{prop}\label{PALLS3u545}
Let~$u$ be a minimizer of~$J$. Assume that
\begin{equation}\label{GA72GA0234}
u\in C(\overline{\Omega}).
\end{equation}
Assume also that
\begin{equation}\label{dodcvdfgm7mse}
{\mbox{$\Delta u$ is $C^1$
in a neighborhood of~$\partial\Omega$,}}\end{equation}
and that
\begin{equation}\label{dodcvdfgm7mse-BIS}
{\mbox{
$\partial\Omega\cap\{|u|>0\}$ is dense in~$\partial\Omega$.}}\end{equation}
Then,
\begin{equation}\label{APS566dfkg2345r}
\Delta u\ge0\qquad{\mbox{a.e. in }}\Omega.\end{equation}
\end{prop}

We think that the result of Proposition~\ref{PALLS3u545} is helpful
to understand the geometric structure of the minimizers:
nevertheless, since it is not used in the rest of this paper,
we deferred its proof to Appendix~\ref{PALLS3u545-APP}.

In Example~4 of Section~\ref{sec:ex2}
(see page~\pageref{HAskdkjdjfjfagbqew4445tf}), we will further discuss the result
of Proposition~\ref{PALLS3u545},
also in view of the free boundary conditions provided by
Theorem~\ref{FREE BOU COND} and of the bi-harmonicity 
properties outside the free boundary discussed in Lemma~\ref{POBIA}.
\medskip

Next we compute the first domain variation (for this, we use the notation
in which subscripts denote differentiation and superscripts denote coordinates).

\begin{lemma}\label{CONFF}
Let~$u$ be either a minimizer or a one-phase minimizer of~$J$. 
For any~$\phi=(\phi^1, \dots, \phi^n)\in C^\infty_0(\Omega)$
it holds that
\begin{equation}\label{AS D}
2\int_\Omega \Delta u(x)\sum_{m=1}^n\Big( 2\nabla u_m(x)\cdot \nabla\phi^m(x)
+u_m(x)\Delta\phi^m(x)\Big)\,dx=
\int_\Omega
\Big( |\Delta u(x)|^2+\chi_{\{u>0\}}(x)\Big)\mbox{\rm div}\phi(x)\,dx.
\end{equation}
\end{lemma}

\begin{proof} Fix~$\e\in\R$ (to be taken with~$|\e|$ small in the sequel).
Let \begin{equation}\label{Pert}
u_\e(x) :=u(x+\e \phi(x)).\end{equation} 
Notice that~$u_\e$ is an admissible competitor for~$u$
(in case we are dealing with the one-phase problem,
observe that~$u_\e\ge0$ if~$u\ge0$).

For any~$i\in\{1,\dots,n\}$, we have that
\begin{eqnarray*}
\p_i u_\e&=&
\sum_{m=1}^n
u_m(\delta_{mi}+\e \phi^m_i)\\{\mbox{and }}\;
\p_{ii} u_\e&=&\sum_{m,l=1}^nu_{ml}(\delta_{li}+\e\phi^l_i)(\delta_{mi}+\e\phi^m_i)
+\sum_{m=1}^n u_m\e\phi_{ii}^m\\
&=&
u_{ii}+\e\left[\sum_{m,l=1}^n\Big(
u_{ml}\phi^l_i\delta_{mi}+u_{ml}\phi^m_i\delta_{li}\Big)
+\sum_{m=1}^n u_m\phi^m_{ii}\right]+\e^2\sum_{m,l=1}^n u_{ml}\phi^l_i\phi^m_i\\
&=&
u_{ii}+\e\sum_{m=1}^n\Big(
2u_{mi}\phi^m_i+u_m\phi^m_{ii}\Big)+\e^2\sum_{m,l=1}^n u_{ml}\phi^l_i\phi^m_i
.\end{eqnarray*}
We use the change of variable~$y:=x+\e\phi(x)$.
In this way, noticing that
$$ \phi(x)=\phi(y-\e\phi(x))=\phi(y)+O(\e)
,$$ we get
\begin{eqnarray*}
&&J[u_\e]\\&=&\int_\Om \Bigg\{ \Bigg|
\sum_{i=1}^n\Big[
u_{ii}(x+\e\phi(x))+\e
\sum_{m=1}^n
\Big( 2u_{mi}
(x+\e\phi(x))\phi^m_i(x)+u_m(x+\e\phi(x))\phi^m_{ii}(x)\Big)\Big]
+o(\e)\Bigg|^2\\&&\qquad +\chi_{\{u>0\}}(x+\e\phi(x))\Bigg\}\,dx\\
&=&\int_\Om
\Bigg\{ \Bigg|
\sum_{i=1}^n\Big[
u_{ii}(y)+\e
\sum_{m=1}^n
\Big( 2u_{mi}
(y)\phi^m_i(y)+u_m(y)\phi^m_{ii}(y)\Big)\Big]
+o(\e)\Bigg|^2+\chi_{\{u>0\}}(y)\Bigg\}\,
\Big(1-\e\mbox{\rm div}\phi(y)+o(\e)\Big)\,dy\\
&=&\int_\Om
\Bigg\{ \Bigg|
\sum_{i=1}^n
u_{ii}(y)+\e
\sum_{i,m=1}^n
\Big( 2u_{mi}
(y)\phi^m_i(y)+u_m(y)\phi^m_{ii}(y)\Big)
\Bigg|^2+\chi_{\{u>0\}}(y)\Bigg\}\,
\Big(1-\e\mbox{\rm div}\phi(y)\Big)\,dy+o(\e)\\&=&
\int_\Om
\Bigg\{ 
\sum_{i,j=1}^n
u_{ii}(y)u_{jj}(y)+2\e
\sum_{i,j,m=1}^n
\Big( 2u_{jj}(y)u_{mi}
(y)\phi^m_i(y)+u_{jj}(y)u_m(y)\phi^m_{ii}(y)\Big)
+\chi_{\{u>0\}}(y)\Bigg\}\,\\&&\qquad\cdot
\Big(1-\e\mbox{\rm div}\phi(y)\Big)\,dy+o(\e)\\
&=& J[u]-\e\int_\Omega\Bigg\{
\Big( |\Delta u(y)|^2+\chi_{\{u>0\}}(y)\Big)\mbox{\rm div}\phi(y)
-2\Delta u(y)\sum_{m=1}^n\Big( 2\nabla u_m(y)\cdot \nabla\phi^m(y)
+u_m(y)\Delta\phi^m(y)
\Big)
\Bigg\}\,dy+o(\e).
\end{eqnarray*}
Thus taking the derivative in $\e$ and evaluating it at $\e=0$ 
we obtain~\eqref{AS D},
as desired.
\end{proof}

As a consequence of Lemma~\ref{CONFF}, we obtain
the free boundary condition of Theorem~\ref{FREE BOU COND}:

\begin{proof}[Proof of Theorem~\ref{FREE BOU COND}] We use the notation
\begin{eqnarray*}
&& g(x):= |\Delta u(x)|^2+\chi_{\{u>0\}}(x),\\
&& G^m(x):=
\Delta u(x)\nabla u_{m}(x)\\
{\mbox{and }}&& H^m(x):=\Delta u(x) u_m(x)
\end{eqnarray*}
for each~$m\in\{1,\dots,n\}$.

We let~$\phi\in C^\infty_0(\Omega)$ and we claim that
\begin{equation}\label{STAMPA}
g\,\mbox{\rm div}\phi-
4\sum_{m=1}^n G^m\cdot\nabla\phi^m-2\sum_{m=1}^n H^m\Delta\phi^m=0\qquad{\mbox{a.e. in }}\Omega.
\end{equation}
To check this, we recall that
\begin{equation}\label{STAMPA2}
{\mbox{if $f\in W^{1,1}_{\rm{loc}}(\Omega)$ then~$\nabla f=0$
a.e. in~$\{x\in\Omega {\mbox{ s.t. }}f=0 \}$,}}
\end{equation}
see e.g. Theorem~6.19 in~\cite{MR1817225}
(used here with~$A:=\{0\}$).
Then, first of all, since~$u\in W^{2,2}(\Omega)$,
we deduce from~\eqref{STAMPA2} that
\begin{equation}\label{STAMPA3}
{\mbox{$\nabla u(x)=0$
for all~$x\in\{u=0\}\setminus Z$,}}\end{equation}
for a suitable~$Z$ of null measure.
Furthermore, for every~$j\in\{1,\dots,n\}$, we have that~$\partial_j u\in
W^{1,2}(\Omega)$.
Accordingly, using~\eqref{STAMPA2} once again,
we find that
\begin{equation}\label{STAMPA4}
{\mbox{$\nabla\partial_j u(x)=0$ for all~$x\in\{\partial_ju=0\}\setminus Z_j$,}}\end{equation}
with~$Z_j$ of null measure.

We also remark that
$$\{\partial_ju=0\} \supseteq \{u=0\}\setminus Z,$$ thanks
to~\eqref{STAMPA3}, and therefore~\eqref{STAMPA4} yields that
\begin{equation}\label{STAMPA5}
{\mbox{$\nabla\partial_j u(x)=0$ for all~$x\in\{ u=0\}\setminus(Z\cup Z_j)$.}}\end{equation}
Hence, defining~$Z^\star:=Z\cup Z_1\cup\dots Z_n$,
we have that~$Z^\star$ has null measure and, by~\eqref{STAMPA3}
and~\eqref{STAMPA5},
\begin{equation}\label{STAMPA7}
D^2u(x)=0\qquad{\mbox{for every }}x\in \{u=0\}\setminus Z^\star. \end{equation}
Moreover, if~$x\in\{u=0\}$, then~$\chi_{\{u>0\}}(x)=0$. This and~\eqref{STAMPA7}
give that~$g=G^m=H^m=0$ in~$\{u=0\}\setminus Z^\star$,
which in turn yields~\eqref{STAMPA}, as desired.

As a consequence of~\eqref{STAMPA} and of the Monotone Convergence
Theorem, we deduce that
$$ \int_\Omega \Bigg( g\,\mbox{\rm div}\phi-
4\sum_{m=1}^n G^m\cdot\nabla\phi^m-2\sum_{m=1}^n H^m\Delta\phi^m\Bigg)\\
=\lim_{\e{{\to}}0}
\int_{\Omega \cap\{ |u|>\e\}}\Bigg(
g\,\mbox{\rm div}\phi-
4\sum_{m=1}^n G^m\cdot\nabla\phi^m-2\sum_{m=1}^n H^m\Delta\phi^m\Bigg).$$
Therefore, recalling~\eqref{AS D} and~\eqref{ZEROM}, we find that
\begin{equation}\label{AS D2}
\begin{split}
0 \;&= \int_\Omega \Bigg( g\,\mbox{\rm div}\phi-
4\sum_{m=1}^n G^m\cdot\nabla\phi^m-2\sum_{m=1}^n H^m\Delta\phi^m\Bigg)\\
&=\lim_{\e{{\to}}0}
\int_{\Omega \cap\{ |u|>\e\}}\Bigg(
g\,\mbox{\rm div}\phi-
4\sum_{m=1}^n G^m\cdot\nabla\phi^m-2\sum_{m=1}^n H^m\Delta\phi^m\Bigg)
\\&=\lim_{\e{{\to}}0}
\int_{\Omega \cap\{ |u|>\e\}}\Bigg(
\mbox{\rm div}(g\,\phi)-
4\sum_{m=1}^n \mbox{\rm div}(\phi^m G^m)-2\sum_{m=1}^n \mbox{\rm div}(H^m\nabla\phi^m)
\\&\qquad\qquad+4\sum_{m=1}^n \phi^m\mbox{\rm div} G^m
+2\sum_{m=1}^n \nabla H^m\cdot\nabla\phi^m
-\nabla g\cdot\phi\Bigg).
\end{split}
\end{equation}
We remark that, in~$\{ |u|>\e\}$,
\begin{eqnarray*}
&&4\sum_{m=1}^n \phi^m\mbox{\rm div} G^m
+2\sum_{m=1}^n \nabla H^m\cdot\nabla\phi^m
-\nabla g\cdot\phi\\&=&
\sum_{m=1}^n \Bigg(4\phi^m\big( \nabla\Delta u\cdot\nabla u_m+\Delta u\Delta u_m\big)
+ 2\big(u_m\nabla\Delta u+\Delta u\nabla u_m\big)\cdot\nabla\phi^m-2\Delta u
\Delta u_m\phi^m\Bigg)
\\&=&
\sum_{m=1}^n \Bigg(4\nabla\Delta u\cdot\nabla u_m\phi^m+
2\Delta u\Delta u_m\phi^m
+ 2\big(u_m\nabla\Delta u+\Delta u\nabla u_m\big)\cdot\nabla\phi^m\Bigg)\\&=&
\sum_{m=1}^n \Bigg(4\nabla\Delta u\cdot\nabla u_m\phi^m+
2\Delta u\Delta u_m\phi^m
+ 2\mbox{\rm div}\Big(\phi^m
\big(u_m\nabla\Delta u+\Delta u\nabla u_m\big)\Big)-
2\mbox{\rm div}\big(u_m\nabla\Delta u+\Delta u\nabla u_m\big)\phi^m
\Bigg)\\
&=& 2\sum_{m=1}^n\Bigg(
\mbox{\rm div}\Big(\phi^m
\big(u_m\nabla\Delta u+\Delta u\nabla u_m\big)\Big)-
u_m\Delta^2 u\phi^m
\Bigg)\\
&=& 2\sum_{m=1}^n
\mbox{\rm div}\Big(\phi^m
\big(u_m\nabla\Delta u+\Delta u\nabla u_m\big)\Big)
,\end{eqnarray*}
by virtue of Lemma~\ref{POBIA}. 
As a consequence, we see that
\begin{eqnarray*}
&&\int_{\Omega \cap\{ |u|>\e\}}\Bigg(4\sum_{m=1}^n \phi^m\mbox{\rm div} G^m
+2\sum_{m=1}^n \nabla H^m\cdot\nabla\phi^m
-\nabla g\cdot\phi\Bigg)\\
&=& 2\sum_{m=1}^n \int_{\Omega \cap\{ |u|>\e\}}
\mbox{\rm div}\Big(\phi^m
\big(u_m\nabla\Delta u+\Delta u\nabla u_m\big)\Big)\\&=&
2\sum_{m=1}^n \int_{\partial(\Omega \cap\{ |u|>\e\})}
\phi^m\big(u_m\nabla\Delta u+\Delta u\nabla u_m\big)\cdot\nu,
\end{eqnarray*}
where~$\nu$ is the exterior normal to~$\Omega \cap\{ |u|>\e\}$.
Hence, using this information in~\eqref{AS D2},
we obtain that
\begin{equation*}
\begin{split}
0 \;&= \lim_{\e{{\to}}0}
\int_{\partial(\Omega \cap\{ |u|>\e\})}\Bigg(
g\,\phi\cdot\nu-\sum_{m=1}^n\Big(
4 \phi^m G^m\cdot\nu
+2 H^m\nabla\phi^m\cdot\nu
-2\phi^m\big(u_m\nabla\Delta u+\Delta u\nabla u_m\big)\cdot\nu\Big)\Bigg)\\&=
\lim_{\e{{\to}}0}
\int_{\partial(\Omega \cap\{ |u|>\e\})}\Bigg(
\big(|\Delta u|^2+\chi_{\{u>0\}}\big)\phi\cdot\nu-
2\sum_{m=1}^n \Big(
\phi^m\big(\Delta u\nabla u_m-u_m\nabla\Delta u\big)\cdot\nu
+ \Delta u u_m
\nabla\phi^m\cdot\nu\Big)\Bigg).
\end{split}
\end{equation*}
This gives~\eqref{PALA}.
Then, to obtain~\eqref{Nu1},
one uses the two different scales of the test function~$\phi^m$
and of its derivative.
\end{proof}

\begin{remark} \label{REMBa}{\rm We point out that if~$n=1$,
when the free boundary divides regions of positivity and nonpositivity of~$u$
(say,~$u^{(1)}>0$ and the interior of~$u^{(2)}\le0$),
formula~\eqref{Nu1}
gives the free boundary conditions
\begin{eqnarray}\label{FB1dim1}
&& \ddot u^{(1)} \dot u^{(1)}=\ddot u^{(2)}\dot u^{(2)}\\{\mbox{and }}\;&&\label{FB1dim2}
2\dot u^{(1)}\dddot u^{(1)}-|\ddot u^{(1)}|^2+1
=2\dot u^{(2)}\dddot u^{(2)} -|\ddot u^{(2)}|^2.
\end{eqnarray}
Also, since~$u\in W^{2,2}(\Omega)$
and~$n=1$, by standard embedding results we already know that~$u\in C^1(\Omega)$.
This, in view of~\eqref{FB1dim1}, implies that
either~$\dot u=0$ at a free boundary point, or~$\ddot u^{(1)} =\ddot u^{(2)}$.
That is, either~$u$ has horizontal tangent at a free boundary point,
or it is~$C^2$ across the free boundary point. Hence, from~\eqref{FB1dim2},
we have the following one-dimensional dichotomy for the free boundary points:
\begin{eqnarray}\label{FB:X1}
&& {\mbox{either: $\dot u=0$ and~$|\ddot u^{(1)}|^2-|\ddot u^{(2)}|^2=1$,}}\\
\label{FB:X2}
&&{\mbox{or: $\dot u\ne0$, $u$ is $C^2$ across and~$
\dddot u^{(1)} =\dddot u^{(2)}-\displaystyle\frac{1}{2\dot u}$.}}\end{eqnarray}
}\end{remark}

\section{Some examples in dimension~$1$}\label{sec:ex2}

\noindent{\bf Example 1.}
To better understand Remark~\ref{REMBa}, we can sketch
some one-dimensional computations.
Namely, we let $n=1$, consider
an interval~$\Omega:=(0,A)$, with~$A>0$,
and prescribe
the Navier conditions~$u(0)=\ddot u(0)=0$,
$u(A)=1$ and~$\ddot u(A)=0$. We look for one-phase
minimizers of~$J$
with such boundary conditions.

In this case, by the finiteness of the energy and Sobolev embedding, we know
that the one-phase minimizer is~$C^1(0,A)$; also
the free boundary points are
minimal point for~$u$, and therefore
\begin{equation}\label{Joi}
{\mbox{$\dot u=0$ at any free boundary point.}}\end{equation} Accordingly,
condition~\eqref{FB:X1} prescribes that
\begin{equation}\label{TBCHE}
\ddot u^+=1
.\end{equation}
Let us see how such condition emerges from energy considerations.
We suppose that the problem develops a free boundary
and we denote by~$a\in(0,A)$ the largest free boundary point,
i.e. $u(a)=0$ and~$u>0$ in~$(a,A)$.
{F}rom Lemma~\ref{POBIA}, we know that~$\ddddot u=0$ in~$(a,A)$
and so~$u$ is a polynomial of degree~$3$ in~$(a,A)$.
Consequently, we can write, for any~$x\in(a,A)$,
$$ u(x)=\alpha (x-a)+\beta(x-a)^2+\gamma(x-a)^3.$$
Then, recalling~\eqref{Joi}, we conclude that~$\alpha=0$.
Imposing the boundary conditions at the point~$x=A$, we find that
$$ \beta=\frac{3}{2(A-a)^2}\quad{\mbox{ and }}\quad
\gamma=-\frac{1}{2(A-a)^3},$$
and therefore
\begin{equation}\label{m3a}
u(x)=\frac{3(x-a)^2}{2(A-a)^2}-\frac{(x-a)^3}{2(A-a)^3}.\end{equation}
The goal is then to choose~$a\in(0,A)$ in order to minimize
the energy contribution of~$u$ in~$(a,A)$, namely we want
to minimize the function
\begin{eqnarray*}
\Phi(a)&:=&\int_a^A |\ddot u(x)|^2\,dx+(A-a)\\&=&
\int_a^A \left|
\frac{3}{(A-a)^2}-\frac{3(x-a)}{(A-a)^3}
\right|^2\,dx+(A-a)\\&=&9
\int_a^A \left|
\frac{(A-a)-(x-a)}{(A-a)^3}
\right|^2\,dx+(A-a)
\\&=&\frac{9}{(A-a)^6}
\int_a^A \left|
A-x
\right|^2\,dx+(A-a)\\&=&\frac{3}{(A-a)^3}+(A-a),
\end{eqnarray*}
which attains its minimum for
\begin{equation}\label{89k934-601-139498-1}
a=A-\sqrt{3}.\end{equation}
That is, comparing with the linear function~$\ell(x):=\frac{x}{A}$,
we have that
$$ A=J[\ell]\ge J[u]\ge \Phi(a)\ge\Phi(A-\sqrt{3})=\frac{1}{\sqrt{3}}+\sqrt{3}.$$
This means that when~$A<\frac{1}{\sqrt{3}}+\sqrt{3}=:B$,
the problem does not develop any free boundary;
when~$A=B$ the problem
has two minimizers, and when~$A>B$
the minimizer in~\eqref{m3a} becomes
\begin{equation}\label{U:CA} u(x)=\frac{(x-a)^2}{2}-\frac{(x-a)^3}{2\cdot 3^{3/2}},\end{equation}
for which~$\ddot u(a^+ )=1$.
This checks \eqref{TBCHE} in this case.

\begin{figure}
\begin{center}
 \begin{tikzpicture}

    \coordinate (O) at (0,0) ;
        \coordinate (X) at (0,7.5) ;
            \coordinate (Y) at (4,0) ;
    
    \coordinate (A) at (0,4) ;
    \coordinate (B) at (0,-4) ;

    \draw[->,black,ultra thick] (O) node[below left] {$O$}-- (90:3.0); 
       \draw[->, black,ultra thick] (O) -- (0:8.5); 

               \draw[blue, thick, dash pattern=on5pt off3pt] (O) -- (2,2) ;
               \draw[blue, thick, dash pattern=on5pt off3pt] (O) -- (3,2) ;
               \draw[blue, thick, dash pattern=on5pt off3pt] (O) -- (4,2) ;
               \draw[blue, thick, dash pattern=on5pt off3pt] (O) -- (5,2) ;    
                              \draw[red, thick, dash pattern=on5pt off3pt] (O) -- (6,2) ;

               \draw[thick, dash pattern=on5pt off3pt] (0,2) node[below left] {$1$}-- (8.5,2) ; 
                              \draw[thick,dash pattern=on5pt off3pt] (6,0)node[below right] {$B$} -- (6,2.7) ; 

  \draw[purple, thick] (4.3,0) parabola (8,2);
\draw[red, thick] (2,0) parabola (6,2);
\draw[purple, thick] (3,0) parabola (7,2);

\end{tikzpicture}
\end{center}
\caption{\it {{The minimizers of a one-dimensional one-phase problem,
in dependence of the right endpoint.}} }\label{1K2}
\end{figure}
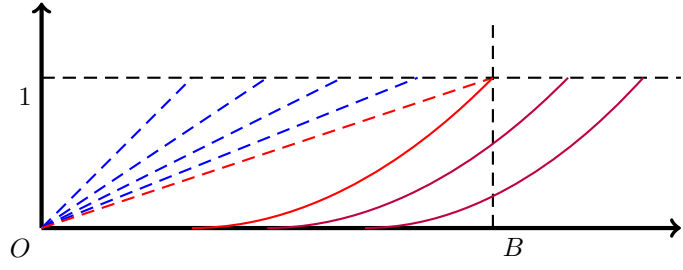

The description of the different one-phase minimizers
in dependence of the endpoint~$A$ is sketched in Figure~\ref{1K2}.\medskip

It is also worth pointing out \label{PAGP}
that 
\begin{equation}\label{NOSI}
{\mbox{the one-phase minimizers described here are {\em not}
super-biharmonic,}}\end{equation}
and this creates a major difference with respect to the case
of minimizers, compare with Lemma~\ref{POBIA}: indeed, if~$\varphi\in C^\infty_0((0,A),[0,+\infty))$
and~$A>\frac1{\sqrt3}+\sqrt3$, from~\eqref{89k934-601-139498-1} and~\eqref{U:CA}
we see that
\begin{eqnarray*}&&
\int_0^A \ddot u\ddot\varphi=
\int_a^A \left(1-\frac{x-a}{\sqrt3}\right)\,\ddot\varphi=
\left(1-\frac{A-a}{\sqrt3}\right)\,\dot\varphi(A)
-\dot\varphi(a)
-\int_a^A \frac{d}{dx}\left(1-\frac{x-a}{\sqrt3}\right)\,\dot\varphi\\&&\qquad=
0-\dot\varphi(a)
+\frac{1}{\sqrt3}\,\int_a^A \dot\varphi=-\dot\varphi(a)-\frac{\varphi(a)}{\sqrt3},
\end{eqnarray*}
which has no sign, thus proving~\eqref{NOSI}.
\bigskip

\noindent{\bf Example 2.}
Having clarified condition~\eqref{FB:X1} in a concrete example,
we aim now at clarifying the role of condition~\eqref{FB:X2}.
Such condition is, in a sense, more unusual, since
it prescribes the matching of the second derivatives
at the free boundary points with nontrivial slopes, with
the bulk term of the energy producing a discontinuity on the third derivatives.

To understand this phenomenon in a concrete example, we fix a small parameter~$\e>0$
and minimize the energy functional
$$ J[u]=\int_{-1}^{1} \left( |\ddot u(x)|^2+\e\chi_{\{u>0\}}(x)\right)\,dx,$$
subject to the Navier conditions
\begin{equation}\label{Imap}
u(-1)=-1,\qquad\ddot u(-1)=0,\qquad u(1)=1,\qquad\ddot u(1)=0.\end{equation}
If we call~$u_\e$ such minimizer, we can bound the energy of~$u_\e$
with that of the identity function. This produces a uniform bound for~$u_\e$
in~$W^{2,2}((-1,1))$, which implies that~$u_\e$ converges in~$C^1((-1,1))$
to the identity function as~$\e{{\to}}0$.
Consequently, for a fixed and small~$\e>0$, we can find
some~$a\in(-1,1)$, which depends on~$\e$, such that
$$ u_\e(x)=\left\{
\begin{matrix}
\underline\alpha (a-x)+\underline\beta (a-x)^2+\underline\gamma (a-x)^3 &{\mbox{ if }} x\in(-1,a),\\
\\
\overline\alpha(x-a)+\overline\beta(x-a)^2+\overline\gamma(x-a)^3 &{\mbox{ if }} x\in[a,1).
\end{matrix}
\right. $$
The condition that~$u_\e\in C^1((-1,1))$ (with derivative close to~$1$
when~$\e$ is small) implies that~$-\underline\alpha=\overline\alpha=\alpha$,
for some~$\alpha>0$ (which depends on~$\e$
and it is close to~$1$ when~$\e$ is small).
Imposing the boundary conditions in~\eqref{Imap}, we find
\begin{equation}\label{PAlalga1} \underline\beta=-
\frac{3(1-\alpha (1+a))}{2(1+a)^2},\qquad
\underline\gamma=\frac{1-\alpha (1+a)}{2(1+a)^3},\qquad
\overline\beta=\frac{3(1-\alpha(1-a))}{2(1-a)^2},\qquad
\overline\gamma=\frac{\alpha(1-a)-1}{2(1-a)^3}.
\end{equation}
Therefore, the energy of~$u_\e$ corresponds to the function
\begin{eqnarray*}
\Psi(a,\alpha) &:=& J[u_\e]\\
&=& \int_{-1}^a |2\underline\beta +6\underline\gamma (a-x)|^2\,dx
+\int_a^1 |2\overline\beta +6\overline\gamma (x-a)|^2\,dx+\e\,(1-a)\\
&=& \left( \frac{3(1-\alpha (1+a))}{(1+a)^3}\right)^2
\int_{-1}^a |1+x|^2\,dx+
\left( \frac{3(1-\alpha (1-a))}{(1-a)^3}\right)^2
\int_a^1 |1-x|^2\,dx+\e\,(1-a)\\&=&
\frac{3(1-\alpha (1+a))^2}{(1+a)^3}
+\frac{3(1-\alpha (1-a))^2}{(1-a)^3}+\e\,(1-a).
\end{eqnarray*}
Thus, we have to minimize such function for~$(a,\alpha)\in
(-1,1)\times(0,+\infty)$,
and in fact we know that such minimum is localized
at~$(0,1)$ when~$\e=0$. Therefore, to find the minima of~$\Psi$, we solve the system
\begin{equation}\label{Sysma} \left\{\begin{matrix}
& 0=\partial_a\Psi=
\displaystyle\frac{
12a\big(\alpha a^4(\alpha+2)+2a^2(2\alpha-\alpha^2+3)+\alpha^2-6\alpha+6\big)
}{(1-a^2)^4}-\e
,\\ & 0=\partial_\alpha\Psi=12\,\displaystyle\frac{\alpha-1-a^2(1+\alpha)}{(1-a^2)^2}.
\end{matrix}\right.\end{equation}
The latter equation produces
\begin{equation}\label{a qua} a^2=\frac{\alpha-1}{1+\alpha}.\end{equation}
We notice that, by~\eqref{PAlalga1},
$$ 
\frac23\left(\overline\beta-\underline\beta\right)=
\frac{1-\alpha(1-a)}{(1-a)^2}
+\frac{1-\alpha (1+a)}{(1+a)^2}=\frac{2\big((\alpha+1) a^2-\alpha+1\big)}{(1-a^2)^2}.
$$
Hence, in view of~\eqref{a qua},
\begin{eqnarray*}
\frac23\left(\overline\beta-\underline\beta\right)=
\frac{2\left((\alpha +1)\frac{\alpha-1}{1+\alpha}-\alpha+1\right)}{(1-a^2)^2}
=
\frac{2\left( \alpha-1-\alpha+1\right)}{(1-a^2)^2}=0,
\end{eqnarray*}
and so~$\overline\beta=\underline\beta$.
This says that the second derivatives match at the free boundary point,
in agreement with the condition in~\eqref{FB:X2}.

In addition, by~\eqref{PAlalga1},
\begin{equation}\label{p023AKAqkw}
\begin{split}
& 4\alpha(\overline\gamma+\underline\gamma)=2\alpha\left(
\frac{\alpha(1-a)-1}{(1-a)^3}+\frac{1-\alpha (1+a)}{(1+a)^3}\right)\\&\qquad=
-\frac{4\alpha a\,(a^2(2\alpha+1) -2\alpha+3)}{(1-a^2)^3}
=-\frac{4\alpha a\,( -2\alpha a^4-a^4+4\alpha a^2-2a^2-2\alpha+3)}{(1-a^2)^4}
.\end{split}\end{equation}
On the other hand, the
first equation in~\eqref{Sysma} says that
$$\displaystyle\frac{
12a}{(1-a^2)^4}=\frac{\e}{
\alpha a^4(\alpha+2)+2a^2(2\alpha-\alpha^2+3)+\alpha^2-6\alpha+6
}.$$
Using this information in~\eqref{p023AKAqkw}, we deduce that
\begin{equation}\label{00a0203aa} 12\alpha(\overline\gamma+\underline\gamma)=
-\frac{\e\,\alpha \,( -2\alpha a^4-a^4+4\alpha a^2-2a^2-2\alpha+3)}{
\alpha a^4(\alpha+2)+2a^2(2\alpha-\alpha^2+3)+\alpha^2-6\alpha+6 }.\end{equation}
Moreover, in view of~\eqref{a qua},
\begin{eqnarray*}&&
-2\alpha a^4-a^4+4\alpha a^2-2a^2-2\alpha+3= \frac4{(1+\alpha)^2}\\
{\mbox{and }}\qquad&&
\alpha a^4(\alpha+2)+2a^2(2\alpha-\alpha^2+3)+\alpha^2-6\alpha+6 =
\frac{4\alpha}{(1+\alpha)^2}.\end{eqnarray*}
Hence, we insert these identities into~\eqref{00a0203aa}
and we find that
$$ 2\dot u(a)\,\big( \dddot u(a^+)-\dddot u(a^-)\big)
=
12\alpha(\overline\gamma+\underline\gamma)=-
\frac{\e\,\alpha \,\frac4{(1+\alpha)^2}}{\frac{4\alpha}{(1+\alpha)^2} }
=-\e,$$
in agreement with the third derivative prescription in~\eqref{FB:X2}.

\noindent{\bf Example 3.} As a variation of Example~2,
we point out that positive data can yield minimizers which change sign,
thus providing an important difference with respect to the classical cases
in which the energy is driven by the standard Dirichlet form.
This example is interesting also because it shows that, in our framework,
this ``loss of Maximum Principle'' can occur even when the domain is
a ball (in fact, even in one dimension, when the domain is an interval)
and even when the data are strictly positive.

In this sense, this example is instructive since it shows that, even in domains in which
the Maximum Principle holds for biharmonic equations
(such as the ball, as established in~\cite{BOGGIO}), the
Maximum Principle can be violated in our framework due to the important
role played by the ``bulk'' term in the energy functional.

To construct our example, we take~$A>0$ and we look for minimizers in~$(-A,A)$
with boundary
conditions~$u(A)=u(-A)=1$ and~$\ddot{u}(A)=\ddot{u}(-A)=0$.

First of all, we observe that
\begin{equation}\label{55:4we6}
J[u]\le C,\end{equation}
for some~$C>0$ independent of~$A$.
To this end,
we take~$\phi\in C^\infty(\R,[0,+\infty))$ such that~$\phi(x)=0$
for all~$x\le 1$ and~$\phi(x)=x-2$ for all~$x\ge 5/2$. Then, assuming~$A\ge5$,
we define
$$ v(x):=\begin{cases}
\phi(x+3-A) & {\mbox{ if $x\in(A-4,A]$,}}\\
0&{\mbox{ if $x\in[-A+4,A-4]$,}}\\
\phi(-x+3-A) & {\mbox{ if $x\in[-A,-A+4)$.}}\\
\end{cases}$$
We observe that~$v(A)=\phi(3)=3-2=1$
and~$v(-A)=\phi(-x+3-A)=\phi(3)=1$.
Moreover~$\ddot{v}(A)=\ddot{\phi}(3)=0$ and~$\ddot{v}(-A)=\ddot{\phi}(3)=0$.
Therefore
\begin{equation*}\begin{split}&
J[u]\le J[v]\le\int_{-A}^A \Big(|\ddot{v}|^2+\I v\Big)=
\int_{ [-A,-A+4)\cup(A-4,A]}\Big(|\ddot{v}|^2+\I v\Big)\\&\qquad\le
\int_{ [-A,-A+4)} |\ddot{\phi}(-x+3-A)|^2\,dx
+\int_{(A-4,A]} |\ddot{\phi}(x+3-A)|^2\,dx
\\&\qquad=
\int_{ [-1,3)} |\ddot{\phi}(y)|^2\,dx
+\int_{(-1,3]} |\ddot{\phi}(y)|^2\,dx
+8
\\&\qquad\le 8(\|\phi\|_{C^2([-1,3])}+1),\end{split}
\end{equation*}
which proves~\eqref{55:4we6}.

Now
we show that, if~$A$ is sufficiently large, then
\begin{equation}\label{gfdHJAddo23r4t}
{\mbox{the minimizer~$u$ cannot be strictly positive in~$(-A,A)$.}}
\end{equation}
To check this, we argue by contradiction, supposing that~$u>0$ in~$(-A,A)$.
Therefore~$\ddddot{u}=0$ and therefore~$u$ must be a polynomial of degree~$3$,
namely
$$ u(x)=a_0+a_1 x+a_2 x^2+a_3 x^3.$$
As a consequence,
$$ 0=\ddot{u}(\pm A)=2 a_2\pm 6a_3 A,$$
and hence
$$ 2 a_2+ 6a_3 A=0=2 a_2- 6a_3 A,$$
which yields~$a_3=0$ and as a result~$a_2=0$. Accordingly,
$$ 1=u(\pm A)=a_0\pm a_1 A,$$
giving that
$$ a_0+ a_1 A=1=a_0- a_1 A,$$
and therefore~$a_1=0$, which also implies that~$a_0=1$.
In this way, we found that~$u(x)=1$ for all~$x\in(-A,A)$, and
consequently~$ J[u]=2A$. This is in contradiction with~\eqref{55:4we6}
as long as~$A$ is sufficiently large, and so we have established~\eqref{gfdHJAddo23r4t}.

We now strengthen~\eqref{gfdHJAddo23r4t} by proving that
\begin{equation}\label{86862013960958414506-23}
{\mbox{the set~$\{u<0\}$ is nonempty.}}\end{equation}
For this, we first use~\eqref{gfdHJAddo23r4t} to find a point~$\bar x\in(-A,A)$
such that~$u(\bar x)\le0$. If~$u(\bar x)<0$ we are done,
hence we can suppose that~$0=u(\bar x)\le u(x)$ for all~$x\in(-A,A)$.
By the finiteness of the energy and Sobolev embedding, we know
that the one-phase minimizer is~$C^{1,\alpha}(0,A)$, for some~$\alpha\in(0,1)$.
In particular, we can take~$\bar x$ as large as possible in the zero set
of~$u$, finding that~$u>0$ in~$(\bar x,A]$,
and therefore we can write that
$$ 0<u(x)\le C_0\,|x-\bar x|^{1+\alpha}\qquad{\mbox{for all }}x\in(\bar x,A],$$
for some~$C_0>0$.

Notice also that
$$ 1=u(A)-u(\bar x)\le \|u\|_{C^1((-A,A))}(A-\bar x),$$
and therefore~$A-\bar x\ge c_0$, for some~$c_0>0$.

Now, given~$\e>0$, to be taken conveniently small in what follows,
we define
\begin{equation}\label{234FRAG8435hd73IS}\delta:=\left(\frac{\e}{C_0}\right)^{\frac1{1+\alpha}},\end{equation}
and in this way~$\delta<c_0$ if $\e$ is sufficiently small. Furthermore,
we observe that if~$x\in (\bar x,\bar x+\delta]\subset(\bar x,A]$ then
$$ 0<u(x)\le C_0\delta^{1+\alpha}=\e,$$
that is
\begin{equation*}
(\bar x,\bar x+\delta]\subseteq\{0<u\le\e\}.
\end{equation*}
For this reason,
\begin{equation}\label{FRAG8435hd73IS}
\delta\le|\{0<u\le\e\}|=|\{u>0\}|-|\{u>\e\}|.
\end{equation}
We now define
$$ u_\e(x):=\frac{u(x)-\e}{1-\e},$$
and we point out that
$$u_\e(\pm A)=\frac{u(\pm A)-\e}{1-\e}=\frac{1-\e}{1-\e}=1\qquad{\mbox{and}}\qquad
\ddot{u}_\e(\pm A)=\frac{\ddot{u}(\pm A)}{1-\e}=0.$$
This says that~$u_\e$ is a competitor for~$u$, hence, recalling~\eqref{55:4we6} and~\eqref{FRAG8435hd73IS},
\begin{eqnarray*}
0&\le& J[u_\e]-J[u]\\&=&
\int_{-A}^A \left( |\ddot{u}_\e|^2-|\ddot{u}|^2+\chi_{\{u_\e>0\}}-\chi_{\{u>0\}}\right)\\
&=&\int_{-A}^A \left( \left|
\frac{\ddot{u}}{1-\e}
\right|^2-|\ddot{u}|^2+\chi_{\{u>\e\}}-\chi_{\{u>0\}}\right)
\\&\le&\frac{2\e-\e^2}{(1-\e)^2}
\int_{-A}^A  |\ddot{u}|^2+|\{u>\e\}|-|\{u>0\}|\\&\le&
C_1\e-\delta,
\end{eqnarray*}
for some~$C_1>0$.

{F}rom this and~\eqref{234FRAG8435hd73IS}, it follows that
$$ C_1\ge\frac{\delta}{\e}=\frac1\e\;\left(\frac{\e}{C_0}\right)^{\frac1{1+\alpha}}=
\frac{1}{C_0^{\frac1{1+\alpha}}\;\e^{\frac\alpha{1+\alpha}}},
$$
which produces a contradiction when~$\e$ is sufficiently small and thus completes the proof
of~\eqref{86862013960958414506-23}.

\noindent{\bf Example 4.}
A natural question arising from Proposition~\ref{PALLS3u545}
(in view of \label{HAskdkjdjfjfagbqew4445tf}
of Lemma~\ref{POBIA} 
and~\eqref{FB:X1}) is whether a function~$u\in C^{1,1}([-1,1])$
satisfying
\begin{equation}\label{GDAbxarvik6m-1}
\begin{cases}
\ddddot{u}=0 \quad{\mbox{ in }} \{u>0\}\cup\{u<0\},\\
\ddot u(-1)=\ddot u(1)=0,\\
{\mbox{$u>0$ in $(0,1)$ and $u<0$ in $(-1,0)$,}}\\ 
{\mbox{$\dot u(0)=0$ and~$|\ddot u(0^+)|^2-|\ddot u(0^-)|^2=1$}}
\end{cases}
\end{equation}
needs necessarily to satisfy
\begin{equation}\label{GDAbxarvik6m-2}
\ddot u\ge0 \quad{\mbox{ a.e. in }}(-1,1).
\end{equation}
Were a statement like this true, the result of Proposition~\ref{PALLS3u545}
could be strengthened (at least in dimension~$1$) by taking into account not
only minimizers but solutions of Navier equations with
prescribed free boundary conditions.
The following example shows that this is not the case,
namely~\eqref{GDAbxarvik6m-1}
does not imply~\eqref{GDAbxarvik6m-2}.

Let
$$ u(x):=\begin{cases}
\displaystyle\frac{\sqrt2\;x^2(3-x)}{6} & {\mbox{ if }}x\in[0,1],\\
\\
-\displaystyle\frac{x^2(3+x)}{6} & {\mbox{ if }}x\in[-1,0).
\end{cases}$$
We remark that
$$ \ddot u(x)=\begin{cases}
\sqrt2(1-x) & {\mbox{ if }}x\in(0,1],\\
\\
-(1+x) & {\mbox{ if }}x\in[-1,0),
\end{cases}$$
from which the system in~\eqref{GDAbxarvik6m-1} plainly follows.

Nevertheless, the claim in~\eqref{GDAbxarvik6m-2} does not hold, since~$\ddot u<0$
in~$(-1,0)$.

\section{Mechanical interpretation of the
free boundary condition~\eqref{FB:X2}}\label{INTER}

In the classical description of
the displacement of
a thin beam, one assumes that the energy density stored by bending 
the beam is proportional to the square of the curvature.
Namely, supposing that the beam takes the form of a small
graphical deformation~$u:[0,1]\to\R$ from a horizontal segment, with
endpoints normalized at~$0$ and~$1$,
such energy takes the form of 
\begin{equation}\label{fJ1}
J_1[u]=\frac\kappa2\,\int_0^1 \frac{|\ddot u(x)|^2}{
(1+|\dot u(x)|^2)^3
}\,\sqrt{ 1+|\dot u(x)|^2 }\,dx,\end{equation}
being the first term the square of the curvature and the second
the length element.
The parameter~$\kappa>0$ takes into account the stiffness
of the specific material of the beam.
Roughly speaking, the rationale of~\eqref{fJ1}
is that the rigidity of the material
will try to prevent the beam to increase its curvature
(with a quadratic law per unit length).
For small deformations of a beam, the terms~$|\dot u(x)|^2$
are often supposed to be negligible, hence~\eqref{fJ1}
is replaced by
\begin{equation}\label{J1:EN}
J_1[u]=\frac\kappa2\,\int_0^1 |\ddot u(x)|^2\,dx.\end{equation}
We refer to
Section~1.1.1 in~\cite{GAZ}
and the references therein for additional
information on the energy theory of thin beams.

\begin{figure}
    \centering
    \includegraphics[width=10cm]{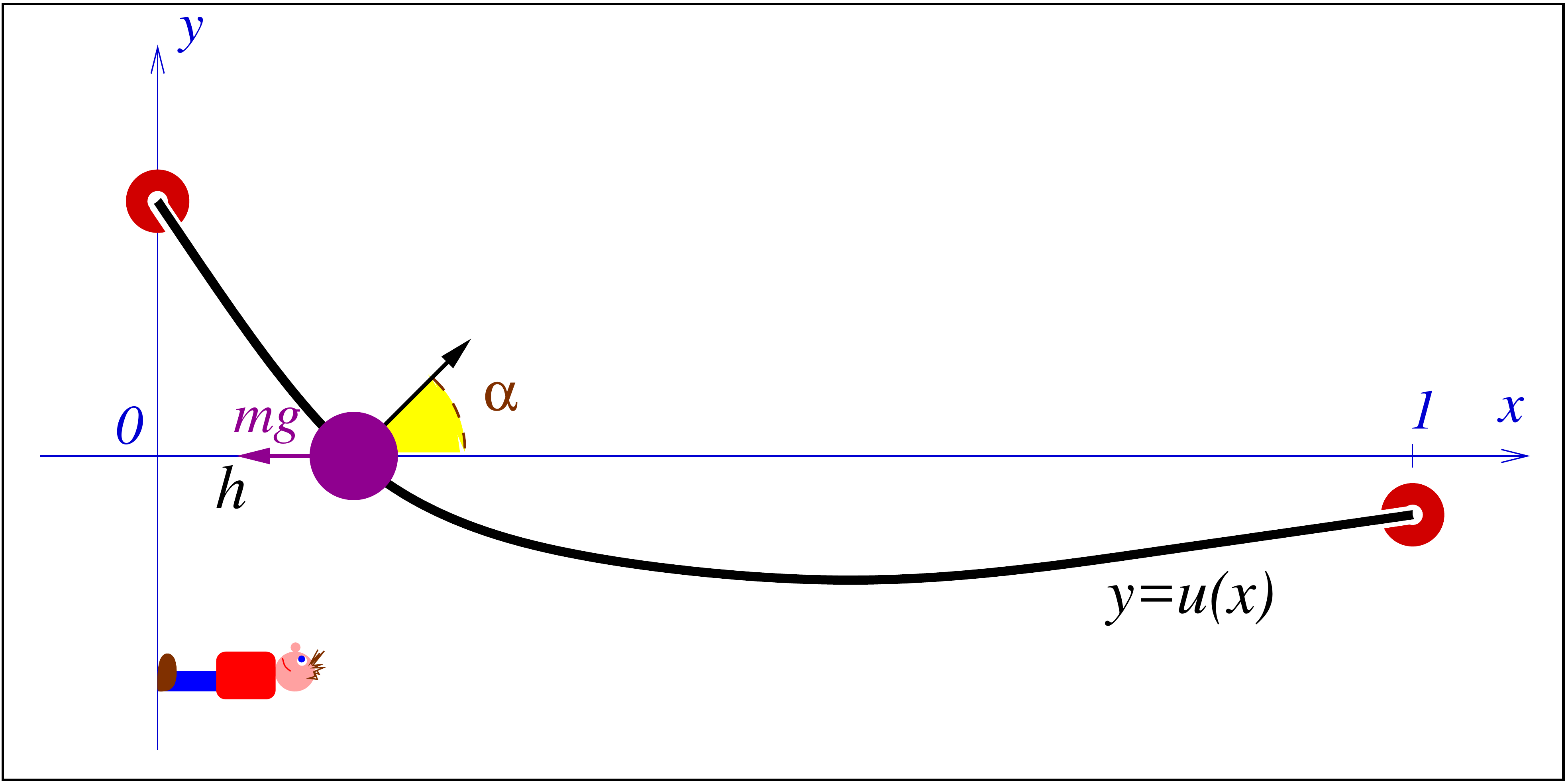}
    \caption{\it {{A simple one-dimensional mechanical
realization of~\eqref{defJ}.}}}
    \label{1K2-1923}
\end{figure}

We now consider a beam of negligible mass 
and a particle of mass~$m$
in a gravitational field with acceleration~$g$,
see\footnote{Of course for a ``real'' observer,
the $x$-axis in Figure \ref{1K2-1923}
would be ``vertical''. We prefer to draw
the picture consistently with the mathematical
formulation in~\eqref{defJ} and thus to follow
the standard convention of placing the $x$-axis
``horizontally''.}
Figure \ref{1K2-1923}.

With respect to Figure \ref{1K2-1923},
we notice that the height~$h>0$ of the particle
corresponds to the one-dimensional measure of the set~$\{u>0\}$,
being the beam represented by the graph~$\{y=u(x),\;x\in[0,1]\}$.
Hence, in this setting,
the gravitation potential energy of the particle is
\begin{equation*}
J_2[u]= mgh = mg\,\int_0^1 \chi_{\{u>0\}}(x)\,dx
.\end{equation*}
{F}rom this and~\eqref{J1:EN}, we obtain that
the full energy of the system is given by
\begin{equation*}
J[u]=J_1[u]+J_2[u]= \int_0^1 \frac\kappa2\,|\ddot u(x)|^2+mg\,
\chi_{\{u>0\}}(x)\,dx
.\end{equation*}
Of course, the functional in~\eqref{defJ}
corresponds to the choice
\begin{equation}\label{CHOICES}
\kappa=2,\qquad m=1,\qquad g=1.\end{equation}
In a balanced configuration,
at points~$x\ne h$, the beam is free and so it satisfies
the equation~$\ddddot u(x)=0$.
On the other hand,
at the point~$h$, the weight of the point mass
needs to be balanced by the force produced by the stiffness of
the beam, that is (in the distributional sense)
\begin{equation}\label{dd 0102}
\kappa \ddddot u + mg\, \Xi=0,\end{equation}
where $\Xi$ is the variation of the measure of~$\{u>0\}$
(which is a distribution concentrated at the point~$h$).
That is, if~$u(h)=0$ and~$\dot u(h)\neq 0$,
given a test function~$\varphi$
and denoting by~$h_\e=h+\e \tilde h+o(\e)$, for some~$\tilde h\in\R$, the point
such that~$(u+\e\varphi)(h_\e)=0$, we have that
\begin{equation}\label{0-1203-034-3522}\begin{split}&
\int_0^1 \Xi(x)\,\varphi(x)\,dx= \lim_{\e\to0}\frac{J_2[u+\e\varphi]
-J_2[u]}{\e}\\ &\qquad\qquad=
\lim_{\e\to0} \int_{0}^1
\frac{\chi_{\{u+\e\varphi>0\}}(x)-
\chi_{\{u>0\}}(x)}{\e}\,dx=
\lim_{\e\to0} \frac{h_\e-h}{\e}=\tilde h.\end{split}\end{equation}
Also, we have that
\begin{eqnarray*}&&0=(u+\e\varphi)(h_\e)=
u(h+\e \tilde h+o(\e))+\e\varphi(h+\e \tilde h+o(\e))\\&&\qquad=
u(h)+\e \dot u(h)\,\tilde h+\e\varphi(h)+o(\e)
=\e\Big( \dot u(h)\,\tilde h+\varphi(h)\Big)+o(\e),\end{eqnarray*}
which gives that~$\tilde h=-\varphi(h)/\dot u(h)$.
Therefore, we deduce from~\eqref{0-1203-034-3522} that
$$ \int_0^1 \Xi(x)\,\varphi(x)=
-\frac{\varphi(h)}{\dot u(h)},$$
and so
$$ \Xi=-\frac{\delta_h }{\dot u},$$
where $\delta_h$ is the Dirac's Delta at the point~$h$.
By inserting this into~\eqref{dd 0102}
we find that
\begin{equation}\label{090k8KLA03kf849775766729djgghd}
\frac{d}{dx} \dddot u=
\ddddot u =- \frac{mg\, \Xi}{\kappa}
= \frac{mg\,\delta_h }{\kappa\,\dot u}
\end{equation}
which is compatible
with
\begin{equation}\label{COWLQQQ1}
\dddot u(h^+)-\dddot u(h^-)
= \frac{mg}{\kappa\,\dot u(h)}.
\end{equation}
Indeed, if~$\e$, $\eta>0$, using~\eqref{090k8KLA03kf849775766729djgghd},
given a smooth function~$f:\R\to[0,1]$ with~$f=1$ in~$[h-\e,h+\e]$,
$f=0$ outside~$[h-\e-\eta,h+\e+\eta]$
and $f$ monotone in~$(h-\e-\eta,h-\e)$ and in~$(h+\e,h+\e+\eta)$, after an integration by parts
we see that
$$ -\int^{h-\e}_{h-\e-\eta}\left(\dot f \,\dddot u\right)-
\int_{h+\e}^{h+\e+\eta}\left(\dot f \,\dddot u\right)=-
\int_\R\left(\dot f \,\dddot u\right)
=
\int_\R\left(f\frac{d}{dx} \dddot u\right)=\int_\R \frac{fmg\,\delta_h }{\kappa\,\dot u}
=\frac{f(h)\,mg }{\kappa\,\dot u(h)}=\frac{mg }{\kappa\,\dot u(h)}.$$
Assuming~$\dddot u$ to be bounded and continuous when~$x\ne h$,
fixed~$\e$, we can write that
\begin{eqnarray*}
&& \left|\int^{h-\e}_{h-\e-\eta}\left(\dot f \;(\dddot u-\dddot u(h-\e))\right)\right|\le
\sup_{x\in[h-\e-\eta,h-\e]}|\dddot u(x)-\dddot u(h-\e)|\;
\int^{h-\e}_{h-\e-\eta}\dot f
\\&&\qquad\qquad\qquad
=\sup_{x\in[h-\e-\eta,h-\e]}|\dddot u(x)-\dddot u(h-\e)|=o(1)
,\end{eqnarray*}
and similarly
$$
\int_{h+\e}^{h+\e+\eta}\left(\dot f \;(\dddot u-\dddot u(h+\e))\right)=o(1)$$ 
as~$\eta\to0$. As a result,\begin{eqnarray*}&& \int^{h-\e}_{h-\e-\eta}\left(\dot f \,\dddot u\right)=
\int^{h-\e}_{h-\e-\eta}\left(\dot f \;(\dddot u-\dddot u(h-\e))\right)+
\int^{h-\e}_{h-\e-\eta}\left(\dot f \;\dddot u(h-\e)\right)\\&&\qquad=o(1)+
(f(h-\e)-f(h-\e-\eta))
\dddot u(h-\e)
=o(1)+
\dddot u(h-\e)
\end{eqnarray*}
and
\begin{eqnarray*}&& \int_{h+\e}^{h+\e+\eta}\left(\dot f \,\dddot u\right)=
\int_{h+\e}^{h+\e+\eta}\left(\dot f \;(\dddot u-\dddot u(h+\e))\right)+
\int_{h+\e}^{h+\e+\eta}\left(\dot f \;\dddot u(h+\e)\right)\\&&\qquad=o(1)+
(f(h+\e+\eta)-f(h+\e))
\dddot u(h+\e)=o(1)-\dddot u(h+\e),
\end{eqnarray*}
as~$\eta\to0$.

{F}rom these considerations, we find that
$$ o(1)+\dddot u(h+\e)-\dddot u(h-\e)
=\frac{mg }{\kappa\,\dot u(h)},$$
hence, sending first~$\eta\to0$ and then~$\e\to0$,
$$ \dddot u(h^+)-\dddot u(h^-)
=\frac{mg }{\kappa\,\dot u(h)},$$
that is~\eqref{COWLQQQ1}.

In particular, with the choices in~\eqref{CHOICES}, we obtain the condition
$$ \dddot u(h^+)-\dddot u(h^-)
= \frac{1}{2\dot u(h)},$$
which is~\eqref{FB:X2} (notice indeed that~$h^+$ comes
in Figure~\ref{1K2-1923} from the negative part of~$u$
and~$h^-$ comes
in Figure~\ref{1K2-1923}
from the positive part of~$u$,
therefore we have that~$\dddot u(h^+)=\dddot u^{(2)}$
and~$\dddot u(h^-)=\dddot u^{(1)}$).

It is interesting to observe that there is also a 
derivation based on elementary dynamics of \eqref{COWLQQQ1}.
Namely, the stiffness of the beam produces a force
at the point~$(h,u(h))$,
normal to the beam for small displacements, 
whose intensity is minus~$\kappa$ times the second
variation of the curvatures, that is
$$ -\kappa \frac{d^2}{dx^2} \ddot u (h)=-\kappa\ddddot u(h). $$
In the setting of Figure \ref{1K2-1923}, the projection of this force along the $x$-axis
is
$$ -\kappa\ddddot u(h) \,\cos\alpha \simeq
-\kappa\ddddot u(h) \,\cot\alpha
= \kappa\ddddot u(h) \,\dot u(h),$$
where the small displacement ansatz has been used.
At the equilibrium, this must balance the weight of the point mass,
therefore we obtain that
$$ \kappa\ddddot u(h) \,\dot u(h)=mg,$$
that is~\eqref{COWLQQQ1} at the point~$x=h$.

\section{A dichotomy argument, and proof of Theorem~\ref{growth}}\label{sec:dic}

We remark that if~$u$ is a minimizer of~$J$ in~$\Omega$
in the admissible class in~\eqref{ADMI}
and~$\Om'$ is a subdomain of~$\Om$, then it
is not necessarily true that~$u$
is a minimizer of~$J$ in~$\Omega'$
in the admissible class in~\eqref{ADMI} with~$\Om$ replaced by~$\Om'$.
This is due to the fact that
the admissible class in~\eqref{ADMI} with~$\Om$ replaced by~$\Om'$
does not prevent the Laplacian of~$u-u_0$ to become singular
at~$\partial\Om'$, and this provides an important difference with
respect to the classical cases dealing with the standard Dirichlet energy.
To circumvent this problem, we will consider local minimizers
in subdomains:

\begin{definition}
Let~$\Om'$ be a subdomain of~$\Om$ with smooth boundary.
We say that~$u$ is a local minimizer in~$\Om'$ if, in the notation of~\eqref{defJ},
$$ J[u,\Om']\le J[v,\Om']$$
for every~$v\in W^{2,2}(\Om')$ such that~$v-u\in W^{2,2}_0(\Om')$.
\end{definition}

In this way, we have:

\begin{lemma}\label{HANSSSL}
If~$u$ is a minimizer in~$\Om$, then it
is a local minimizer in every
subdomain~$\Om'\subset\subset\Om$ with smooth boundary.
\end{lemma}

\begin{proof}
Let~$v\in W^{2,2}(\Om')$ such that~$v-u\in W^{2,2}_0(\Om')$.
By the
extension results in classical Sobolev spaces (see e.g. Proposition~IX.18
in~\cite{MR697382}), we can extend~$v$ outside~$\Om'$
by setting~$v(x):=u(x)$ for all~$x\in\Om\setminus\Om'$,
and we have that~$v-u\in W^{2,2}_0(\Om)\subseteq W^{1,2}_0(\Om)$.
In particular, recalling~\eqref{ADMI},
we have that~$v\in{\mathcal{A}}$ and thus
$$ 0\le J[v,\Om]-J[u,\Om]=
\int_{\Om\setminus\Om'} \Big(|\Delta u|^2+\I u-|\Delta v|^2+\I v\Big)+
J[v,\Om']-J[u,\Om'].
$$
Since~$u=v$ in~$\Om\setminus\Om'$, this gives that~$0\le J[v,\Om']-J[u,\Om']$,
as desired.
\end{proof}

Before proving Theorem~\ref{growth}, we show a result
concerning the convergence of the blow-up sequence of a minimizer.

\begin{lemma}\label{lemma:blowup}
Let~$D\subset\subset\Omega$.
Let~$u_k\in W^{2,2}(D)$, with~$k\in \mathbb N$, be a sequence of local
minimizers
of 
\begin{equation}\label{TiC-0}
\int_D\Big( |\Delta u_k|^2 +M_k\I {u_k}\Big),\end{equation}
with~$M_k\in (0, 1)$,
such that $0\in\partial\{u_k>0\}$ and~$|\na u_k(0)|=0$.

Fix~$R>0$ such that~$B_{5R}\subset\subset D$, and suppose that
\begin{equation}\label{QGROAD}
\sup_{B_{4R}} u_k\le C_0(R),
\end{equation}
and
\begin{equation}\label{QGROAD-PQi}
\|\Delta u_k\|_{L^1(B_{4R})}\le \hat C_0(R),
\end{equation}
for some~$C_0(R)$, $\hat C_0(R)>0$.

Then, there exists
a positive constant~$C(R)$, independent of~$k$, such that
\begin{eqnarray}\label{stima}
&& \|u_k\|_{W^{2, 2}(B_R)}\le C(R),\\
{\mbox{and }} \quad &&\|\Delta u_k\|_{BMO(B_R)}\le C(R),\label{stima2}
\end{eqnarray}
for any~$k\in\N$. 

Furthermore, 
if~$u_k:\R^n\to\R$ and the minimization property in~\eqref{TiC-0}
holds true in any
domain~$D\subset\R^n$, and the corresponding assumptions
in~\eqref{QGROAD} and~\eqref{QGROAD-PQi} are satisfied, then
there exists~$u_0:\R^n\to\R$ such that,
up to subsequences, as~$k\to+\infty$, $u_k\to u_0$
in~$W^{2,2}_{\rm{loc}}(\R^n)
\cap C^{1,\alpha}_{\rm{loc}}(\R^n)$, for any~$\alpha\in(0,1)$.
\end{lemma}

\begin{proof}
To check~\eqref{stima}, we observe that,
in virtue of Lemma~\ref{POBIA},
\begin{equation}\label{jdferufgru845}
\int_{B_{2R}}\Delta u_k\,\Delta \phi\le0,\end{equation}
for any~$\phi\in W^{2, 2}_0(B_{2R})$.
Now, we take~$\xi\in C^{\infty}_0(B_{2R},[0,1])$
such that
\begin{equation}\label{08iu968yt}
{\mbox{$\xi=1$ in~$B_R$}},\qquad
|\nabla\xi|\le\frac{C}{R}
\qquad  {\mbox{and}}\qquad|D^2\xi|\le\frac{C}{R^2},
\end{equation}
for some~$C>0$, we set~$m_k:=\min_{B_{4R}}u_k$, and we
choose~$\phi:=(u_k-m_k)\xi^2\ge 0$
in~\eqref{jdferufgru845}. In this way,
setting
\begin{eqnarray*}
I_1&:=&2 \int_{B_{2R}}
\Delta u_k\,\na u_k\cdot\na \xi^2\\
{\mbox{and }} \qquad I_2&:=& \int_{B_{2R}} (u_k-m_k)\,\Delta u_k\,\Delta \xi^2,
\end{eqnarray*}
we have that
\begin{equation}\label{reuty45bhbb}
0\ge \int_{B_{2R}} \Delta u_k\,\Delta\big((u_k-m_k)\xi^2\big)=
\int_{B_{2R}} (\Delta u_k)^2\xi^2+I_1+I_2.
\end{equation}
Now, thanks to Corollary~\ref{lem-subham}, we can use
the standard method to prove
Caccioppoli inequality: 
namely we take~$\eta\in C^{\infty}_0(B_{4R},[0,1])$
such that~$\eta=1$ in~$B_{2R}$ and~$
|\nabla\eta|\le\frac{C}{R}$
and we infer from Corollary~\ref{lem-subham} and~\eqref{QGROAD-PQi} that
\begin{equation}\label{87795438ws8492875a}\begin{split}&
\widehat C \int_{B_{4R}}(u_k-m_k)\,\eta^2\ge - \int_{B_{4R}} \Delta u_k\,(u_k-m_k)\,\eta^2
=
\int_{B_{4R}} |\nabla u_k|^2\eta^2
+\int_{B_{4R}} 2\eta\,(u_k-m_k)\,\nabla\eta\cdot\nabla u_k\\ &\qquad\ge
\frac12\int_{B_{4R}} |\nabla u_k|^2\eta^2-
C\int_{B_{4R}} (u_k-m_k)^2|\nabla\eta|^2.\end{split}
\end{equation}
We remark that, in view of Corollary~\ref{lem-subham} and~\eqref{QGROAD-PQi},
we can choose here~$\widehat C$ proportional to~$\frac{\tilde C(R)}{R^n}$.
Hence, the result in~\eqref{87795438ws8492875a} yields that
\begin{equation}\label{ineq-Caccio-main}
\int_{B_{2R}}|\na u_k|^2\le
\frac{C}{R^2}\int_{B_{4R}}(u_k-m_k)^2
+ C\int_{B_{4R}}(u_k-m_k)
\end{equation}
for some~$C>0$, possibly varying from line to line. 

Hence, by Young's inequality, \eqref{08iu968yt} and~\eqref{ineq-Caccio-main}, we get 
\begin{equation}\begin{split}\label{iewtubv}
|I_1|\le\; &2\left(\e\int_{B_{2R}}(\Delta u_k)^2\xi^2
+\frac1\e\int_{B_{2R}}|\na u_k|^2\,|\na \xi|^2\right)\\
\le\; &2\left(\e\int_{B_{2R}}(\Delta u_k)^2\xi^2
+\frac{C}{\e\,R^2}\int_{B_{2R}}|\na u_k|^2\right)\\
\le\;& 2\left(\e\int_{B_{2R}}(\Delta u_k)^2\xi^2
+\frac{C}{\e\,R^4}\int_{B_{4R}} (u_k-m_k)^2+
\frac{C}{R^2}\int_{B_{4R}}(u_k -m_k)
\right).
\end{split}\end{equation}
Furthermore, noticing that~$(u_k-m_k) \Delta u_k|\na \xi|^2\ge 
-\widehat C(u_k-m_k)|\na \xi|^2$, thanks
to Corollary~\ref{lem-subham}, and
making again use of Young's inequality, we obtain that 
\begin{eqnarray*}
I_2
&=&
\int_{B_{2R}} (u_k-m_k)\,\Delta u_k\,
\Big(2\xi\Delta \xi+|\na \xi|^2\Big)\\&\ge& 2\int_{B_{2R}} (u_k-m_k)\,\Delta u_k
\,\xi\,\Delta \xi -\widehat C\int_{B_{2R}}(u_k-m_k)|\na \xi|^2\\
&\ge& 
-2\left(\e\int_{B_{2R}}(\Delta u_k)^2\xi^2
+\frac1\e\int_{B_{2R}} (u_k-m_k)^2(\Delta\xi)^2\right)
-\widehat C\int_{B_{2R}}(u_k-m_k)|\na \xi|^2\\
&\ge& 
-2\left(\e\int_{B_{2R}}(\Delta u_k)^2\xi^2
+\frac{C}{\e \,R^4}\int_{B_{2R}} (u_k-m_k)^2\right)
-\frac{C}{R^2}\int_{B_{2R}}(u_k-m_k).
\end{eqnarray*}
{F}rom this, \eqref{reuty45bhbb} and~\eqref{iewtubv}, we conclude that
\begin{eqnarray*}
\int_{B_{2R}} (\Delta u_k)^2\xi^2
&\le&
2\left(\e\int_{B_{2R}}(\Delta u_k)^2\xi^2
+\frac{C}{\e\,R^4}\int_{B_{2R}} (u_k-m_k)^2
+
\frac{C}{R^2}\int_{B_{4R}}(u_k -m_k) \right)
\\&&\qquad +
2
\left(\e\int_{B_{2R}}(\Delta u_k)^2\xi^2+
\frac{C}{\e \,R^4}\int_{B_{2R}} (u_k-m_k)^2\right)
+\frac{C}{R^2}\int_{B_{2R}}(u_k-m_k),
\end{eqnarray*}
which, in turn, implies that
$$
(1-4\e)\int_{B_{2R}} (\Delta u_k)^2\xi^2\le
\frac{C}{\e\,R^4}\int_{B_{2R}} (u_k-m_k)^2
+\frac{ C}{R^2}\int_{B_{2R}}(u_k-m_k)
\le 
\frac{C}\e+  C,
$$
where the last step follows from \eqref{QGROAD}.
Choosing $\e=\frac1{8}$ and recalling~\eqref{08iu968yt},
we obtain that
$$ \int_{B_R} (\Delta u_k)^2\le \int_{B_{2R}} (\Delta u_k)^2\xi^2
\le C,
$$
up to renaming~$C>0$, that does not depend on~$k$.
This implies the desired estimate in~\eqref{stima}.

Moreover, the estimate in~\eqref{stima2}
follows from the BMO estimates in Section~\ref{sec:BMO}.
 
{F}inally, from the uniform estimate in~\eqref{stima}, we can apply
a customary compactness argument to conclude that there exists
a function~$u_0$ such that, up to a subsequence,
$u_k\to u_0$
in~$W^{2,2}_{\rm{loc}}(\R^n)
\cap C^{1,\alpha}_{\rm{loc}}(\R^n)$, for any~$\alpha\in(0,1)$,
as~$k\to+\infty$. This completes the proof of Lemma~\ref{lemma:blowup}.
\end{proof}

With this, we are now in the position of completing the proof of
Theorem~\ref{growth}.

\begin{proof}[Proof of Theorem~\ref{growth}]
We suppose that~$B_1(x_0)\subset\subset D$, with~$x_0$ as in the statement of Theorem~\ref{growth}.
We claim that there exist an integer~$k_0>0$ and a structural
constant~$C>0$,
depending only on~$\delta$, $n$ and~${\rm{dist}}(D,\Omega)$,
such that the following inequality holds:
\begin{equation}\label{eq-dyadic-0}
\sup_{B_{2^{-k-1}}(x_0)} |u| \le \max 
\left\{
\frac C{2^{2k}}, \;\frac{\sup_{B_{2^{-k}}(x_0)}|u|}{2^2}, \ldots,
\frac{\sup_{B_{2^{-k+m}}(x_0)}|u|}{2^{2(m+1)}}, \dots,
\frac{\sup_{B_1(x_0)}|u|}{2^{2(k+1)}}
\right\}, 
\end{equation}
for any~$k\ge k_0$.

Indeed, if~\eqref{eq-dyadic-0} fails, then, for any~$j\in\N$,
there exist singular free boundary points~$x_j\in D$, integers~$k_j$
and minimizers~$u_j$ (with~$\|u_j\|_{W^{2,2}(\Om)}=\|u\|_{W^{2,2}(\Om)}$
be given)
such that 
\begin{equation}\label{eq-dyadic-1}
\sup_{B_{2^{-k_j-1}}(x_j)} |u_j| > \max 
\left\{ 
\frac j{2^{2k_j}},\; \frac{\sup_{ B_{2^{-k_j}}(x_j)}|u_j|}{2^2}, 
\dots, \frac{\sup_{B_{2^{-k_j+m}}(x_j)}|u_j|}{2^{2(m+1)}}, \dots, \frac{\sup_{B_1
(x_j)}|u_j|}{2^{2(k_j+1)}}
\right\}.
\end{equation}
We denote by~$S_j:=\sup_{B_{2^{-k_j-1}}(x_j)}|u_j|$
and we consider the scaled functions 
\begin{equation*}
v_j(x):=\frac{u_j(x_j+2^{-k_j}x)}{S_j}.
\end{equation*}
In this way, \eqref{eq-dyadic-1} gives that
\begin{equation}\label{jquesyter56} 1 > \max 
\left\{ 
\frac j{2^{2k_j}\,S_j},\; \frac{\sup_{ B_{1}} |v_j|}{2^2}, 
\dots, \frac{\sup_{B_{2^{m}}}|v_j| }{2^{2(m+1)}}, \dots,
\frac{\sup_{B_{2^{k_j}}}|v_j|}{2^{2(k_j+1)}}
\right\}.\end{equation}
{F}rom this, we have that the functions~$v_j$
satisfy the following properties:
\begin{equation}\begin{split}\label{QGROAD:0}
& \sup_{B_{1/2}}v_j=1,\\
\\ 
& v_j(0)=|\na v_j(0)|=0,\\
\\ 
& \sup_{B_{2^m}}|v_j|\le 4 \cdot
2^{2m},\quad {\mbox{for any }} m<k_j,\\
\\
& \sigma_j:=\frac1{2^{2k_j}\;S_j}<\frac1j.
\end{split}\end{equation}
We also remark that, from the scaling properties of the functional~$J$,
we have that
\begin{equation}\label{NAVIE1}
\int_{B_R}\Big(|\Delta v_j|^2+\sigma_j^2\I {v_j}\Big)
=2^{k_jn}\sigma_j^2\int_{B_{R2^{-k_j}}(x_j)}\Big(|\Delta u_j|^2
+\I {u_j}\Big),
\end{equation}
for every fixed $R<2^{k_j}$.

We claim that
\begin{equation}\label{NAVIE2}
{\mbox{$v_j$ is a local minimizer
in~$B_R$.}}
\end{equation}
Indeed, by Lemma~\ref{HANSSSL},
we know that~$u$ is a local minimizer in~$B_{R2^{-k_j}}(x_j)$.
Hence,
if~$w_j$ is such that~$w_j-v_j\in W^{2,2}_0(B_R)$, we define, for all~$y\in B_{R2^{-k_j}}(x_j)$,
$$ W_j(y):=S_j \,w_j\big(2^{k_j}(y-x_j)\big).$$
In this way, we have that~$W_j\in W^{2, 2}_0(B_{R2^{-k_j}}(x_j))$,
thus yielding, in light of~\eqref{NAVIE1}, that
\begin{eqnarray*}
0&\ge& 2^{k_jn}\sigma_j^2\left(\int_{B_{R2^{-k_j}}(x_j)}\Big(|\Delta u_j|^2
+\I {u_j}\Big)
-\int_{B_{R2^{-k_j}}(x_j)}\Big(|\Delta W_j|^2
+\I {W_j}\Big)\right)\\&=&
\int_{B_R}\Big(|\Delta v_j|^2+\sigma_j^2\I {v_j}\Big)-
\int_{B_R}\Big(|\Delta w_j|^2+\sigma_j^2\I {w_j}\Big).
\end{eqnarray*}
This completes the proof of~\eqref{NAVIE2}.

Now, by assumption, $u_j$ is not $\delta$-rank-2 flat at each
level~$r=2^{-k}$, for any~$k\ge1$, at $x_j$.
As a consequence, $v_j$ is not $\delta$-rank-2 flat in~$B_1$.
So, recalling~\eqref{flatdef} and Definition~\ref{def:flat}, this means that
\begin{equation}\label{TOU1}
h(1, 0)=\inf_{p\in P_2}h_{{\rm{min}}}(1, x_0, p)\ge \delta .
\end{equation}
Also, we have that condition~\eqref{QGROAD} is guaranteed in this case,
in view of~\eqref{QGROAD:0}. In addition, we have that~\eqref{QGROAD-PQi}
holds true here, since, in view of~\eqref{QGROAD:0},
if~$20R\in[2^{m-1},2^m]$,
$$ v_j-\min_{B_{20R}}v_j\le
2\sup_{B_{20R}}|v_j|\le
2\sup_{B_{2^m}}|v_j|\le8\cdot 2^{2m}\le 8\cdot (40 R)^{2}\le CR^2,$$
and consequently, by Lemma~\ref{eq-Hessian} and~\eqref{jquesyter56},
\begin{eqnarray*} &&\int_{B_{5R}} |\Delta v_j(x)|\,dx=\int_{B_{5R}}
\frac{2^{-2k_j}|\Delta u_j(x_j+2^{-k_j}x)|}{S_j}\,dx=
\frac{2^{(n-2)k_j}}{S_j}
\int_{B_{5R2^{-k_j}}(x_j)}
|\Delta u_j(y)|\,dy
\\&&\qquad\le \frac{2^{nk_j}}{2^{2k_j}S_j}
\int_{B_{5R2^{-k_j}}(x_j)}
|D^2 u_j(y)|\,dy=
\frac{CR^n}{2^{2k_j}S_j}
\fint_{B_{5R2^{-k_j}}(x_j)}
|D^2 u_j(y)|\,dy
\\&&\qquad\le
\frac{CR^n}{2^{2k_j}S_j}\sqrt{
\fint_{B_{5R2^{-k_j}}(x_j)}
|D^2 u_j(y)|^2\,dy}
\\&&\qquad\le
\frac{CR^n}{2^{2k_j}S_j}\left(
\frac{2^{4k_j}}{R^4}\fint_{B_{20R2^{-k_j}}(x_j)}\left(u_j-\min_{B_{20R2^{-k_j}}(x_j)}u_j\right)^2
+\frac{2^{2k_j}}{R^2}\fint_{B_{20R2^{-k_j}}(x_j)}\left(u_j-\min_{B_{20R2^{-k_j}}(x_j)}u_j\right)
\right)^{1/2}
\\&&\qquad=
\frac{CR^n}{2^{2k_j}S_j}\left(
\frac{ 2^{4k_j}S_j^2}{R^4}\fint_{B_{20R}}\left(v_j-\min_{B_{20R}}v_j\right)^2
+\frac{2^{2k_j}S_j}{R^2}\fint_{B_{20R}}\left(v_j-\min_{B_{20R}}v_j\right)
\right)^{1/2}
\\&&\qquad\le
\frac{CR^n}{2^{2k_j}S_j}\left(
2^{4k_j}S_j^2
+ 2^{2k_j}S_j
\right)^{1/2}
\le
CR^n+
\frac{CR^n}{\sqrt{2^{2k_j}S_j}}\le
CR^n+
\frac{CR^n}{\sqrt{j}}\le CR^n.
\end{eqnarray*}
Therefore, recalling~\eqref{NAVIE1},
from Lemma~\ref{lemma:blowup}, applied here with $M_j:=\sigma_j^2$,
we know that, 
up to a subsequence, still denoted by~$v_j$,
there exists a function~$v_\infty$ such that
\begin{equation}\label{doewty8he-3382}
{\mbox{$v_j\to v_\infty$ in~$W^{2, 2}(B_R)\cap
C^{1, \alpha}(B_R)$, for any~$\alpha\in(0, 1)$, as~$j\to+\infty$.}}\end{equation}
Moreover, we have that~$\Delta v_j\in BMO(B_R)$ uniformly.
Consequently~$v_\infty\in W^{2, 2}(B_R)\cap C^{1, \alpha}(B_R)$,
for all $\alpha\in(0, 1)$, and $\Delta v_\infty\in BMO(B_R)$.
Furthermore, 
\begin{equation}\begin{split}\label{yrei8y57hv}
& \Delta^2 v_\infty=0 \;{\mbox{ in }}\R^n, \qquad \sup_{B_{1/2}}v_\infty=1,
\\
& | v_\infty(x)|\le 8|x|^2 \;{\mbox{ for any }}x\in\R^n
\\
& {\mbox{ and }} \quad
v_\infty(0)=|\na v_\infty(0)|=0.
\end{split}\end{equation}
Let now~$f:=\Delta v_\infty$, then we have that $f$ is harmonic in $\R^n$.
Moreover, by
Lemma~\ref{eq-Hessian} and the second line in~\eqref{yrei8y57hv}, 
we see that, for any~$r>0$,
$$ \frac1{r^n}\int_{B_r}|D^2v_\infty|^2\le \frac{C}{r^{n+4}}\int_{B_r}
\big(v_\infty- \min_{B_{4r}} v_\infty\big)^2 +
 \frac{C}{r^{n+2}}\int_{B_r}
\big(v_\infty- \min_{B_{4r}} v_\infty\big)\le C,$$
up to renaming~$C>0$.
Thus, from the Liouville Theorem we infer that~$f$ must be constant, i.e.
$\Delta v_\infty=C_0$, for some~$C_0\in\R$.

Consequently, $v_\infty-\frac{C_0}{2n}\,|x|^2$ is harmonic in~$\R^n$ 
with quadratic growth. Hence, by using the Liouville Theorem 
once again, we have that~$v_\infty(x)=g(x)+\frac{C_0}{2n}\,|x|^2$,
where~$g$ is a second order polynomial. 
Moreover, since $\na v_\infty(0)=0$, we deduce that~$g=cp$,
for some~$c\in \R$ and~$p\in P_2$ (recall~\eqref{poliu7u65}).  

Therefore, we can write
$$ v_\infty(x)=x\cdot Ax,$$
for some constant and symmetric matrix~$A$. 
Consequently, recalling the notation in~\eqref{spx},
\begin{equation}\label{8093y9rrho}
\fb{v_\infty}= S(p, 0)
\end{equation}
for some $p\in P_2$. 
On the other hand, from our construction in~\eqref{TOU1},
we have that
$$ \HD\big(\fb{ v_j}\cap B_1, S(p, 0)\cap B_1\big)\ge \delta
$$
(recall the definitions of~$\HD$ and~$h_{{\rm{min}}}$
in~\eqref{defHD} and~\eqref{defflat:BIS}, respectively).
As a consequence,
there exist points~$z_j\in \fb{ v_j}\cap B_1$
such that
\begin{equation}\label{07445de5fh}
\dist(z_j, S(p, 0))\ge \delta.
\end{equation}
Now we extract a converging sequence, still denoted $z_j$,
such that~$z_j\to z_0$ as~$j\to+\infty$,
and we see from the uniform convergence of~$v_j$
given in~\eqref{doewty8he-3382}
that~$v_\infty(z_0)=0$, which implies that~$z_0\in S(p, 0)$,
thanks to~\eqref{8093y9rrho}.
On the other hand, we also have that~$\dist(z_0, S(p, 0))\ge \delta $,
in virtue of~\eqref{07445de5fh}.
Therefore, we reach a contradiction, and so
the proof of Theorem~\ref{growth} is finished.
\end{proof}

\section{Non-degeneracy, and proof of
Theorems~\ref{thm-nondeg} and \ref{thm-Hausdorff}}\label{sec-nondeg}

In this section we deal with weak and strong nondegeneracy properties of the minimizers.
Due to the lack of Harnack inequalities for biharmonic functions,
the strong nondegeneracy result does not follow immediately from the weak one,
unless we impose some additional conditions
on the set~$\po u$.

\subsection{Weak nondegeneracy, and proof of Theorem~\ref{thm-nondeg}}
Here we prove the weak nondegeneracy for~$u^+$,
according to the statement in Theorem~\ref{thm-nondeg}.

\begin{proof}[Proof of Theorem~\ref{thm-nondeg}]
We prove the claims in~$\bf 1^\circ$ and~$\bf 2^\circ$
together, distinguishing the different structures of the two cases when needed.

After rescaling~$u$ by defining~$r^{-2}u(x_0+r x)$,
we may assume without loss of generality that~$r=1$ and~$x_0=0$.
Also, denote by
\begin{equation}\label{defgamma00}
\gamma:=\sup_{B_{1}}|u|. 
\end{equation} 
We remark that in the setting of~$\bf 2^\circ$, we have that
\begin{equation}\label{67:698677654w123456uuofejn}
B_1\subseteq\{u>0\}
\end{equation}
and therefore
\begin{equation}\label{superur}
\gamma:=\sup_{B_{1}}u. 
\end{equation}
We also remark that, in the setting of~$\bf 1^\circ$,
in light of~\eqref{iehnfnb0094}, we have that
\begin{equation}\label{167:698677654w123456uuofejn}\big|\po u\cap B_{\frac1{16}}\big|
\ge \theta_* \,|B_{\frac1{16}}|.
\end{equation}
As a matter of fact, in case~$\bf 2^\circ$,
the statement in~\eqref{167:698677654w123456uuofejn}
is also true, with~$\theta_*:=1$,
as a consequence of~\eqref{67:698677654w123456uuofejn}.
Hence, we will exploit~\eqref{167:698677654w123456uuofejn}
in both the cases~$\bf 1^\circ$ and~$\bf 2^\circ$,
with the convention that~$\theta_*=1$ in the latter case.

We also point out that, in~$B_{1/8}$
\begin{equation}\label{267:698677654w123456uuofejn}
u-\min_{B_{\frac18}} u \le 2\gamma.
\end{equation}
Indeed, in case~$\bf 1^\circ$, the claim in~\eqref{267:698677654w123456uuofejn}
follows from~\eqref{defgamma00}. Instead, in case~$\bf 2^\circ$,
we exploit~\eqref{67:698677654w123456uuofejn} to write that
$$ u-\min_{B_{\frac18}} u\le u\le\gamma,$$
thus completing the proof of~\eqref{267:698677654w123456uuofejn}.

Now, let 
$\psi\in C^\infty(\R^n,[0,1])$ such that $\psi=0$
in~$B_{\frac{1}{16}}$, $\psi>0$ in~$\R^n\setminus \overline{
B_{\frac1{16}}}$ and 
$\psi= 1$ in~$\R^n\setminus B_{\frac1{8}}$. Set $v:=\psi u$.
Then~$u-v\in W^{2,2}_0(B_{\frac{1}{8}})$, and so~$v$
is a competitor for~$u$ in~$B_{\frac{1}{8}}$. 
Therefore, from the local minimality of~$u$ (as warranted by
Lemma~\ref{HANSSSL}) we have that
\begin{equation*}
\int_{B_{\frac18}}\Big(|\Delta u|^2+\I u\Big)\le \int_D \Big(|\Delta v|^2+\I v\Big),
\end{equation*}
where $D:=B_{\frac1{8}}\setminus \overline{
B_{\frac1{16}}}$. {F}rom this,
and recalling the definitions of~$v$ and~$\psi$, we obtain that
\begin{eqnarray*}
\big|\po u\cap B_{\frac1{16}}\big|
&\le &
\int_{B_{\frac1{16}}}\Big(|\Delta u|^2+\I u\Big)\\
&\le&  \int_D\Big( |\Delta v|^2+\I v\Big)-  \int_D \Big(|\Delta u|^2+\I u\Big)\\
&=&\int_D\Big( |\Delta v|^2-|\Delta u|^2\Big)\\
&\le&\int_{D}|\Delta v|^2.
\end{eqnarray*}
Hence, using Lemma \ref{eq-Hessian} and~\eqref{267:698677654w123456uuofejn},
it follows that 
\begin{equation}\begin{split}\label{0unborih}
\big|\po u\cap B_{\frac1{16}}\big|
\le\; &\int_D(u\Delta \psi +2\na u\na \psi+\psi\Delta u)^2\\
\le \;&2\|\psi\|_{C^2(B_{1/8})}\int_{B_{\frac18}} u^2+4|\na u|^2+|D^2 u|^2\\
\le\;& C \|\psi\|_{C^2(B_{1/8})} \int_{B_{\frac18}}\left(
\left(u-\min_{B_{\frac18}} u\right)^2+
\left(u-\min_{B_{\frac18}} u\right)\right)\\
\le\;& C\|\psi\|_{C^2(B_{1/8})}\,\gamma(1+\gamma),  
\end{split}\end{equation}
for some~$C>0$, possibly varying from line to line.
 
Combining this with~\eqref{167:698677654w123456uuofejn}
and~\eqref{0unborih},
we conclude that 
\[
\gamma(1+\gamma)\ge \frac{\big|\po u\cap B_{\frac1{16}}\big|}{C\|\psi\|_{C^2(B_{1/8})}}
\ge
\frac{\theta_* \,|B_{\frac1{16}}|}{ \|\psi\|_{C^2(B_{1/8})}},
\]
which gives the desired result (using~\eqref{defgamma00}
in case~$\bf 1^\circ$ and~\eqref{superur}
in case~$\bf 2^\circ$).
\end{proof}

\subsection{Whitney's covering}\label{WHYWTH}

Here we recall the Whitney's decomposition method,
to obtain suitable conditions which
allow us to use Theorem \ref{thm-nondeg}
(in our setting, the structural assumptions of
Theorem \ref{thm-nondeg}
will be provided by formula~\eqref{def-c-cover}).
Suppose that~$E\subset \R^n$ is a
nonempty compact set, then~$\R^n\setminus E$ can be
represented as a union of closed dyadic 
cubes~$Q^k_j$ with  mutually disjoint interiors 
\[
\R^n\setminus E=\bigcup_{k\in \mathbb Z}\bigcup_{j=1}^{N_k} Q^k_j
\]
such that 
\[
c_1\le \frac{\dist(Q^k_j, E)}{\diam\, Q^k_j}\le c_2
\]
for two universal constants $c_1$, $c_2>0$. Here $Q^k_j$ is a cube 
with side length equal to~$2^{-k}$.

Let now~$E:=\{u\le 0\}\cap \overline {Q_1(x_0)}$,
where~$Q_1(x_0)$ is the unit cube centered at~$x_0\in \fb u$,
and consider the Whitney's decomposition for~$\R^n\setminus E$.
Let~$k_0\in \mathbb N$ be fixed, 
and suppose that for every~$k\ge k_0$ there exists~$c>0$ such that, for some~$Q_j^k$, we have 
\begin{equation}\label{def-c-cover}
\dist\left(x_0,Q_j^k\right)\le c2^{-k}.
\end{equation}
Then 
$u^+$ is strongly nondegenerate at $x_0$. 
To see this, for every large $k$ let us take a cube $Q^k_j$ such that \eqref{def-c-cover}
holds. Then,
if~$x_1$ is the center of~$Q^k_j$, we have that~$u(x_1)>0$
and~$\dist(x_1,\partial\{u>0\})\ge2^{-k-1}$,
Hence, in view of claim~$\bf 2^\circ$ of
Theorem \ref{thm-nondeg}, we find that
\begin{equation}\label{1563212863rtqd}
\sup_{B_{2^{-k-1}}(x_1)} u^+\ge\bar c (2^{-k-1})^2=
\frac{\bar c}{4}\,2^{-2k}.\end{equation}
On the other hand, by~\eqref{def-c-cover}, we see that
$$ |x_0-x_1|\le c2^{-k}+\sqrt{n}2^{-k}=(c+\sqrt{n})2^{-k},$$
and accordingly~$B_{c^*\,2^{-k}}(x_0)\supseteq B_{2^{-k-1}}(x_1)$,
with~$c^*:=c+\sqrt{n}+\frac12$.
Therefore, by~\eqref{1563212863rtqd},
$$ \sup_{B_{c^*\,2^{-k}}(x_0)} u^+\ge
\frac{\bar c}{4}\,2^{-2k},$$
as desired.

\begin{definition}\label{def:cov}
If \eqref{def-c-cover} holds, then we say that $\fb u$ satisfies a
weak $c$-covering condition at $x_0\in \fb u$. 
\end{definition}

We remark that the standard $c$-covering condition,
that was introduced in~\cite{Martio}, is stronger
than~\eqref{def-c-cover}
and indeed it requires that
$$\dist\left(x_0, \bigcup_{j=1}^{N_k}Q_j^k\right)\le c2^{-k}.$$
Moreover, it is known that the weak $c$-covering condition of Definition~\ref{def:cov}
is satisfied by the John domains, see~\cite{Martio}.

In order to recall the definition of John domain, 
we let~$0<\alpha\le\beta<\infty$. 
A domain $D\subset \R^n$ is called an $(\alpha,\beta)$-John domain,
denoted by $D\in {\mathcal{J}}(\alpha,\beta)$, 
if there exists~$x_0\in D$ such that every~$x\in D$ has a rectifiable
path~$\gamma: [0, d]\to D$ 
with arc length as parameter such that~$\gamma(0)=x$, $\gamma(d)=x_0$,
$d\le \beta$ and
\[
\dist (\gamma(t), \p D)\ge \frac{\alpha}d t, \quad \text{for all}\ \ t\in[0, d].
\]
The point $x_0$ is called a center of $D$. 
A domain $D$ is called a John domain if $D\in {\mathcal{J}}(\alpha,\beta)$
for some~$\alpha$ and~$\beta$.
The class of all John domains in~$\R^n$ 
is denoted by~${\mathcal{J}}$. 

For more on such coverings and applications  
of Whitney's decompositions we refer to~\cite{Martio}.
\medskip 

Alternative sufficient geometric conditions
on $\po u$ guaranteeing the strong nondegeneracy of 
$u$ can be given. Note that in order to pass from weak to strong nondegeneracy at some 
$z\in \fb u$ it is enough to have 
a small ball $B'\subset B_r(z)\cap \po u$ and $c>0$
such that $\diam\, B'\ge c r$ for every small $r$,
since this guarantees~\eqref{iehnfnb0094}.
 
\begin{definition}
We say that $\fb u$ satisfies a
nonuniform interior cone condition if for every~$x\in \fb u$ there exist a positive 
number~$r_x>0$ and a cone~$K_x$ with vertex at~$x$,
such that~$B_{r_x}(x)\cap K_x\subset \po u$.

We also say that $\fb u$ satisfies a
uniform interior cone condition if there exist a positive 
number~$r>0$ and a cone~$K$ with vertex at~$0$,
such that for every~$x\in \fb u$ we have that~$B_{r}(x)\cap (x+K)\subset \po u$.
\end{definition}

From our observation above and Theorem~\ref{thm-nondeg}
we immediately obtain the following result:

\begin{corollary} Let~$u$ be a minimizer for~$J$
in~$\Omega$, and~$x_0\in\Omega$.
Suppose that $\po u$ satisfies the interior cone
condition at $x_0\in \fb u$, then $|u|$ is nondegenerate at $x_0$. Moreover,  if 
$\po u$ satisfies the uniform interior cone condition and~$B_1\subset \Om$,
then
$$\sup_{B_r(z)}u^+\ge C_0 r^2,$$ for any~$z\in \fb u\cap B_1$,
for some~$C_0>0$.
\end{corollary}

\subsection{The biharmonic measure, and proof of Theorem~\ref{thm-Hausdorff}}\label{BIHAMOE}

In this subsection, we describe the main features of the
measure induced by the bi-Laplacian of a minimizer.
For this, we observe that,
since, by Lemma \ref{POBIA}, $\Delta u$ is super-harmonic,
\begin{equation}\label{7yh0128eUDIS}
{\mbox{there exists a nonnegative measure 
$\mathscr M_u$ such that $-\Delta^2 u=\mathscr M_u$. }}\end{equation}
Hence, for any~$\psi\in C_0^\infty(\Om)$, we have that 
\begin{eqnarray}\label{bi-measure}
\int_\Omega \mathscr M_u \psi=\int_\Omega(-\Delta u)\Delta \psi.
\end{eqnarray}
Recalling the notion of flatness introduced in Definition~\ref{def:flat},
we have the following:

\begin{lemma}\label{LEMMA M}
Let~$u$ be a minimizer of the functional~$J$ defined in~\eqref{defJ}, let~$\delta>0$
and let~$x_0\in\partial\{u>0\}$
such that~$\nabla u(x_0)=0$ and~$\partial\{u>0\}$ is not~$\delta$-rank-2 flat at~$x_0$
at any level~$r>0$ with~$B_r(x_0)\subset\subset\Omega$.
Then,
\begin{equation}\label{fdnbbtoru69f}
\mathscr M_u(B_r(x_0))\le Cr^{n-2}\end{equation}
for any~$r>0$ as above, for some~$C>0$.
\end{lemma}

\begin{proof}
Without loss of generality, we take~$x_0=0$.
We consider a function~$\psi_0\in C_0^\infty(B_2,[0,1])$, with~$\psi_0=1$
in~$B_1$, and we let~$\psi(x):=\psi_0(x/r)$. 
In this way, $\psi=1$ in~$B_r$ and~$|D^2\psi|\le C/r^2$ for some~$C>0$.

We now exploit~\eqref{bi-measure} with such~$\psi$. Then,
by Corollary~\ref{eq-Hessian:cor}, we have that
$$ \mathscr M_u(B_r)\le \int_\Omega \mathscr M_u \psi
= \int_\Omega(-\Delta u)\Delta \psi\le
\sqrt{\int_{B_{2r}}|\Delta u|^2}\sqrt{\int_{B_{2r}}|\Delta \psi|^2}
\le C r^{\frac{n}{2}} r^{\frac{n-4}2},
$$
which implies the desired result, up to renaming~$C>0$.
\end{proof}

We remark that a full counterpart of Lemma~\ref{LEMMA M}
does not hold for the one-phase problem
(in particular~$\mathscr M_u$ as defined in~\eqref{7yh0128eUDIS}
and~\eqref{bi-measure}
does not need to have a sign, see~\eqref{NOSI}).
Nevertheless, the following result holds:

\begin{lemma}\label{NEBVERAMAL}
Let~$u$ be a one-phase minimizer of~$J$. Assume that~$u\in C^{1,1}(\Omega)$
and~$\partial\{u>0\}$ has null Lebesgue measure. 
Let~$\varphi\in C^\infty_0(B_1,[0,1])$ with
$$ \int_{B_1}\varphi=1.$$
For any~$\delta>0$, let
$$\varphi_\delta(x):=\frac1{\delta^n}\varphi\left(\frac{x}\delta\right),$$
and~$u_\delta:=u*\varphi_\delta$.
Then, for any~$\Omega'\subset\subset\Omega$, we have that
$$ \lim_{\delta\to0}\int_{\Omega'} \Delta^2 u_\delta \,u_\delta=0.$$
\end{lemma}

\begin{proof} 
Let
$$\Gamma_\delta:=\bigcup_{p\in\partial\{u>0\}}B_\delta(p).$$
We claim that
\begin{equation}\label{9ikjwdish88uhbsxcvj}
{\mbox{if~$x\in \Omega\setminus\Gamma_\delta$, then~$\Delta^2 u(x)=0$.}}
\end{equation}
To prove this, we argue by contradiction and we suppose that
there exists~$x\in \Omega\setminus\Gamma_\delta$ such that
\begin{equation}\label{hUA8jA;aL}
{\mbox{$\Delta^2 u(x)$ is either not defined or not null}}.\end{equation}
We observe that
\begin{equation}\label{hUA8jA;aL2}
{\mbox{there exists~$\rho$, $a>0$
such that~$u\ge a$ in $B_\rho(x)$.}}\end{equation}
Because, if not,
for any~$k\in\N$, there exists~$x_k$ such that~$|x-x_k|+u(x_k)\le 1/k$,
and thus~$u(x)=0$. Since~$x$ lies outside~$\Gamma_\delta$,
it cannot be a free boundary point, hence~$u$ must
vanish in a neighborhood of~$x$. Consequently, $\Delta^2u$ vanishes
in a neighborhood of~$x$, and this is in contradiction with~\eqref{hUA8jA;aL},
thus proving~\eqref{hUA8jA;aL2}.

Then, from~\eqref{hUA8jA;aL2} and
Lemma~\ref{POBIA}, it follows that~$u$ is
biharmonic in $B_\rho(x)$. Once again, this is
in contradiction with~\eqref{hUA8jA;aL}, and thus the proof of~\eqref{9ikjwdish88uhbsxcvj}
is complete.

Now, by taking~$\delta$ sufficiently small, we suppose that
the distance from~$\Omega'$ to~$\partial\Omega$ is larger than~$\delta$.
Thus, from~\eqref{9ikjwdish88uhbsxcvj}
we obtain that, if~$x\in\Omega'\setminus\Gamma_{2\delta}$
and~$y\in B_\delta$, then~$x-y\in\Omega'\setminus\Gamma_{\delta}$,
hence~$\Delta^2u(x-y)=0$.

Consequently, for every~$x\in\Omega'\setminus\Gamma_{2\delta}$,
$$ \Delta^2u_\delta(x)=\int_{B_\delta} \Delta^2u(x-y)\,\varphi_\delta(y)\,dy=0.$$This implies that
\begin{equation}\label{VHJAvjach}
\int_{\Omega'} \Delta^2 u_\delta \,u_\delta=
\int_{\Omega'\cap\Gamma_{2\delta}} \Delta^2 u_\delta \,u_\delta.
\end{equation}
We also remark that
\begin{equation}\label{9ikmAMMMXaLA}
\begin{split} &|\Delta^2u_\delta(x)|\le
\int_{B_\delta} |u(x-y)|\,|\Delta^2\varphi_\delta(y)|\,dy
=\frac{1}{\delta^{n+4}}
\int_{B_\delta} |u(x-y)|\,\left|\Delta^2\varphi\left(\frac{y}{\delta}\right)\right|\,dy
\\&\qquad=\frac{1}{\delta^{4}}
\int_{B_1} |u(x-\delta y)|\,\left|\Delta^2\varphi(y)\right|\,dy
\le \frac{C}{\delta^4}\int_{B_1} u(x-\delta y)\,dy,
\end{split}\end{equation}
for some~$C>0$.
Now, if~$x\in \Gamma_{2\delta}$ and~$y\in B_1$,
we have that there exists~$p\in\partial\{u>0\}\subseteq\{u=0\}$
such that~$|p-x|\le2\delta$
and accordingly~$|(x-\delta y)-p|\le|x-p|+\delta\le 3\delta$.
Then, in this setting, the regularity of~$u$ implies that
\begin{equation}\label{8ijoKKAK}
u(x-\delta y)\le 9\|u\|_{C^{1,1}(\Omega)}\delta^2.\end{equation}
In particular, recalling~\eqref{9ikmAMMMXaLA},
we find that, if~$x\in \Gamma_{2\delta}$,
\begin{equation}\label{8ijoKKAK-x2}
|\Delta^2u_\delta(x)|\le\frac{C}{\delta^2},
\end{equation}
up to renaming~$C>0$, also depending on~$\|u\|_{C^{1,1}(\Omega)}$.

{F}rom~\eqref{8ijoKKAK} we also deduce that, if~$x\in \Gamma_{2\delta}$,
$$ |u_\delta(x)|\le
\int_{B_1} u(x-\delta y)\,\varphi(y)\,dy\le 9\|u\|_{C^{1,1}(\Omega)}\delta^2.
$$
Using this information and~\eqref{8ijoKKAK-x2}
we conclude that, if~$x\in \Gamma_{2\delta}$,
$$ |\Delta^2u_\delta(x)\,u_\delta(x)|\le C,$$
and therefore
$$\left| \int_{\Omega'\cap\Gamma_{2\delta}} 
\Delta^2 u_\delta \,u_\delta\right|\le C\,|\Omega'\cap\Gamma_{2\delta}|,$$
up to renaming~$C>0$ once again.

This and~\eqref{VHJAvjach} give that
$$ \left|\int_{\Omega'} \Delta^2 u_\delta \,u_\delta\right|
\le C\,|\Omega'\cap\Gamma_{2\delta}|.$$
Hence, taking the limit as~$\delta\to0$,
$$\lim_{\delta\to0}\left|\int_{\Omega'} \Delta^2 u_\delta \,u_\delta\right|\le
\,|\Omega'\cap\partial\{ u>0\}|.$$
This gives the desired result.
\end{proof}

Now we prove a counterpart of~\eqref{fdnbbtoru69f} at 
nondegenerate points of the free boundary of the minimizers.
For this, recalling the setting in formula~\eqref{flatdef},
we let~$\mathcal N_\delta $ be the set of free boundary
points~$x$ with the property that there exists~$r_x>0$ small enough
such that~$h(r, x)\ge \delta r$ for every~$r<r_x$.
Moreover, in the spirit of Definition~\ref{def:sing}, we also denote by
$$ \mathcal N_\delta^{\text{sing}}:=\left\{ x\in{\mathcal{N}}_\delta {\mbox{ s.t. }}
\nabla u(x)=0\right\}.
$$

\begin{lemma}\label{lem-blya-don4} Let~$u$ be a minimizer of~$J$.
Let $D\subset \Om$ and suppose that there exists~$
\bar c>0$ such that 
\begin{equation}\label{9yyhnb}
\liminf_{r\to 0}\frac{\sup_{B_r(x)}|u|}{r^2}\ge \bar c\end{equation}
for every $x\in \fb u\cap \overline D$.
Then there exists~$c_0(\delta)>0$, depending on 
$n$, $\delta$, $\bar c$, $\|u\|_{W^{2,2}(\Om)}$ and $\dist(\overline D, \p\Om)$, such that 
\begin{equation}\label{bi-meas-est}
\liminf_{r\to 0}\frac{\mathscr M_u(B_r(x))}{r^{n-2}}\ge c_0(\delta),
\quad {\mbox{ for any }} x\in \mathcal N_\delta^{\text{sing}}.
\end{equation}
\end{lemma}

\begin{proof}
We argue by contradiction.
If \eqref{bi-meas-est} fails, then there exists a sequence~$
x_j\in \mathcal N_\delta^{\text{sing}}$ such that 
\begin{equation}\label{fail-blya-don}
\liminf_{r\to 0}\frac{\mathscr M_u(B_r(x_j))}{r^{n-2}}<\e_j
\end{equation}
with $\e_j\to 0$. 
Since $x_j\in \mathcal N_\delta^{\text{sing}}$,
there exists a sequence~$r_j\to 0$  such that 
\begin{equation}\label{kfoeryuhb}
h(r_j, x_j)\ge \delta r_j.\end{equation} 
Now we define 
$$U_j(x):=\frac{u(x_j+r_jx)}{r_j^2}.$$ 
By construction, recalling~\eqref{9yyhnb}, we have that~$\{U_j\}$ is 
nondegenerate with quadratic growth, i.e. there exists~$C>0$ independent of $j$ such that 
\begin{equation}\label{fail-blya-don2}
\frac1C R^2\le \sup_{B_R}|U_j|\le C R^2\quad {\mbox{ for any }}
R<\frac1{r_j}. 
\end{equation}
Moreover, by~\eqref{fail-blya-don} and~\eqref{kfoeryuhb}, we see that
\begin{equation}\label{fail-blya-don1}
h(1, 0)\ge \delta \quad \text{and}\quad \mathcal M_{U_j}(B_R)\le \e_j R^{n-2}\to 0 
\end{equation}
for every fixed $R>0$. 

As a consequence, using a customary compactness argument,
we can extract a converging subsequence, still denoted by~$U_j$, 
such that~$U_j\to U_0$ locally uniformly as~$j\to+\infty$.
Then \eqref{fail-blya-don1} translates into 
\begin{equation}\label{8.13}
h(1, 0)\ge \delta \quad
\text{and}\quad\mathcal M_{U_0}(B_R)=0
\end{equation}
for every fixed $R>0$. 
In other words, in view of \eqref{fail-blya-don2}, we have that~$U_0$ is an entire nontrivial biharmonic function with quadratic growth. 

On the other hand, applying Corollary~\ref{eq-Hessian:cor} we also have that 
\begin{equation*}
\int_{B_R}|D^2 u|^2\le CR^{n}. 
\end{equation*}
This, together with the Liouville Theorem,
implies that
\begin{equation}\label{065unwtq}
{\mbox{$U_0$ is a quadratic polynomial.}}\end{equation}
Accordingly, there exists~$\alpha\in \R$  such that
$p:=\alpha U_0\in P_2$ (recall the notation
in~\eqref{poliu7u65}).
{F}rom \eqref{8.13}, we conclude that
$$\HD(S(p, 0)\cap B_1, \fb {U_0}\cap B_1)\ge \delta,$$
which is a contradiction with~\eqref{065unwtq}. The proof of Lemma~\ref{lem-blya-don4}
is thus finished.
\end{proof}

We are now in position to complete our analysis of
the free boundary regularity results which follow from the study of the biharmonic
measure by proving Theorem~\ref{thm-Hausdorff}.

\begin{proof}[Proof of Theorem~\ref{thm-Hausdorff}]
We start by proving $\bf 1^\circ$. For this, let~$D\subset\subset\Omega$
and~$x\in \mathcal F_\delta:=\big(\fb u\cap D\big)\setminus \mathcal N_\delta$,
where~$\mathcal N_\delta$ has been introduced before
Lemma~\ref{lem-blya-don4}.
Then there exists~$r_x>0$ such that
$$|\fb u\cap B_{r_x}(x)|\le C(n)\delta r^n_x,$$
where $C(n)$ is a dimensional constant.
In this way, we can cover~$\mathcal F_\delta$ with balls~$B_{r_x}(x)$, and
we can then extract a Besicovitch covering such that
\begin{equation}\label{orhvna} |\mathcal F_\delta\cap D|\le C(n)\delta\,|D|.\end{equation}
Then, sending $\delta\to 0$ the result in $\bf 1^\circ$ follows.

We now focus on $\bf 2^\circ$.
In this case, thanks to~\eqref{9eybn-sdfsn} we
can use Lemma \ref{lem-blya-don4}
and find a Besicovitch covering by balls $B_{r_x}(x)$ of $\mathcal N_\delta^{\text{sing}}$ such that 
\begin{equation}\label{565-uijbn}
c_0(\delta)\sum r_x^{n-2}\le \mathscr M_u(D')<\infty
\end{equation}
where $D'\Supset D$ is a subdomain of $\Omega$ such that
$$\dist(D, \partial D')<\sup_{x\in \fb u\cap D} r_x:=r_0.$$
Therefore, letting $r_0\to 0$ in~\eqref{565-uijbn}, we get that
\begin{equation}\label{orhnbfp2}
\H^{n-2}(\mathcal N_\delta^{\text{sing}}\cap D)<+\infty.\end{equation}
Furthermore, since the free boundary is~$C^1$ near points in~${\mathcal{N}}_\delta
\setminus{\mathcal{N}}_\delta^{\text{sing}}$, we have that
$$ \H^{n-2}\big(({\mathcal{N}}_\delta\setminus
{\mathcal{N}}_\delta^{\text{sing}})\cap D\big)<+\infty, $$
which, together with~\eqref{orhnbfp2}, implies that
\begin{equation}\label{orhnbfp2:BIS}
\H^{n-2}(\mathcal N_\delta\cap D)<+\infty.\end{equation}
This gives the second claim in $\bf 2^\circ$.
We now prove the first claim in $\bf 2^\circ$. For this, we use~\eqref{orhvna}
and~\eqref{orhnbfp2:BIS}
to obtain that
$$ |\fb u\cap D|\le |\mathcal F_\delta\cap D|+|\mathcal N_\delta\cap D|
=|\mathcal F_\delta\cap D|\le C(n)\delta \,|D| .$$
Then, sending $\delta\to 0$, we complete the proof of $\bf 2^\circ$.  
\end{proof}

\begin{remark}{\rm
If $\po u$ is a John domain,
then $|u|$ is nondegenerate, due to the discussion in Subsection~\ref{WHYWTH}.
Alternatively, as in Theorem~\ref{thm-nondeg}, if $\po u$ has uniformly positive 
Lebesgue density then $|u|$ is nondegenerate.}
\end{remark}

\section{Stratification of free boundary, and
proof of Theorem~\ref{thm-strata}}\label{sec:stra}

In this section we reformulate some results obtained in
Section~\ref{sec:dic}
related to the dichotomy between the notion of rank-2 flatness and
the quadratic growth of the minimizer.

For this, to describe an appropriate flatness rate of the minimizers, we recall
Definition~\ref{def:sing}
and we also
define a suitable class, in the following way:

\begin{definition}\label{def-P-class} Fix~$r>0$.
We say that $u\in \mathcal P_r$ if:
\begin{itemize}
\item  $u\in W^{2,2}(B_r)$ is a minimizer of $J$ in~\eqref{defJ}
in $B_r$, among functions $v\in W^{2, 2}(B_r)$, and~$v-u\in W^{1, 2}_0(B_r)$, 
\item and $0\in \fbs u$.
\end{itemize}
If, in addition, given~$\delta>0$,
\begin{itemize}
\item 
the free boundary is not $(\delta, r)$-rank-2 flat at~$0$,
\end{itemize}
then we say that $u\in \mathcal P_r(\delta)$.
\end{definition}

In the setting of Definition~\ref{def-P-class}, Theorem~\ref{growth}
can be reformulated as follows:

\begin{prop}
Let~$u\in \mathcal P_r(\delta)$. Then there exist~$r_0>0$ and~$C>0$, possibly
depending on $n$, $\delta$, $ r$
and~$\|u\|_{W^{2,2}(\Omega)}$, such that 
\[
|u(x)|\le C|x|^2, \quad  {\mbox{ for any }}x\in B_{r_0}.  
\]
\end{prop}

Furthermore, recalling the definition of~$h(r, x_0)$ in~\eqref{flatdef},
a refinement of Theorem~\ref{growth} can be formulated as follows:

\begin{theorem}\label{thm-growth-refine}
Let $u\in \mathcal P_1$. Let~$\delta\in(0,1)$, $ k>10$ and~$r_k:=2^{-k}$.
Then, either $h(0,r_k)<\delta r_k$,
or there exists~$C>0$, possibly
depending on~$n$, $\delta$ and~$\|u\|_{W^{2,2}(\Omega)}$,
such that 
\[
\sup_{B_{r_k/2}}|u|\le Cr_k^2.
\]
\end{theorem}

We are now ready to complete the proof of Theorem~\ref{thm-strata}.

\begin{proof}[Proof of Theorem~\ref{thm-strata}]
Notice that~\eqref{FOR1} and~\eqref{FOR2} follow as a consequence of
Theorem~\ref{thm-growth-refine}. Therefore, to complete
the proof of Theorem~\ref{thm-strata},
it only remains to prove that~$u^+$ is strongly nondegenerate
at~$z\in{\mathcal{F}}$. After rescaling $U_r(x):=r^{-2}u(z+r x)$,
we see that it is enough to show that
\begin{equation}\label{r4tungr67wnc}
\sup_{B_1}U_r^+\ge \hat C,
\end{equation}
for some~$\hat{C}>0$ (which here can depend
on~$n$,
$\dist(z,\partial\Omega)$ and the minimizer~$u$
itself).

To check this,
we first prove that
\begin{equation}\label{POLY2}\begin{split}&
{\mbox{if $p$ is a homogeneous polynomial of degree two,}}\\&{\mbox{then $\{p=0\}$
is contained in the union of finitely many hypersurfaces.}}\end{split}
\end{equation}
Indeed, up to a linear transformation, and possibly exchanging the order of the variables,
we can suppose that
$$ p(x)=\sum_{i=1}^n a_i x_i^2,$$
with~$(a_1,\dots,a_m)\in\R\setminus\{0\}$ and~$a_{m+1}=\dots=a_n=0$,
for some~$m\in\{1,\dots,n\}$. Therefore the zero set of~$p$ is obtained by
the zero set of the polynomial
$$ \R^m\ni x\mapsto \tilde p(x)=\sum_{i=1}^m a_i x_i^2,$$
up to a Cartesian product with an $(n-m)$-dimensional linear space.
Also, 
\begin{equation}\label{TRANS}
{\mbox{if~$x\in\{\tilde p=0\}$, then~$tx\in\{\tilde p=0\}$ for all~$t\in\R$,}}\end{equation}
therefore
\begin{equation}\label{TRANS3} \{\tilde p=0\}=\big\{ tx,\; x\in\{\tilde p=0\}\cap \S^{m-1}\big\}.\end{equation}
Furthermore
\begin{equation}\label{TRANS2}\{\nabla \tilde p=0\}=\{(2a_1x_1,\dots,2a_mx_m)=0\}=\{ 0\}.\end{equation}
Therefore, by~\eqref{TRANS2},
in the vicinity of any~$x\in\{\tilde p=0\}\cap \S^{m-1}$,
the set~$\{\tilde p=0\}$ is an $(m-1)$-dimensional surface, which,
in view of~\eqref{TRANS}, is transverse to~$\S^{m-1}$.
Consequently, we have that~$\{\tilde p=0\}\cap \S^{m-1}$ is
the union of~$(m-2)$-dimensional surfaces.
In addition, from~\eqref{TRANS2} we know that these surfaces
cannot accumulate to each other, and so~$\{\tilde p=0\}\cap \S^{m-1}$ is
the union of finitely many~$(m-2)$-dimensional surfaces.

This and~\eqref{TRANS3} imply that~$\{\tilde p=0\}$ is
the union of finitely many~$(m-1)$-dimensional surfaces.
Accordingly, we have that~$\{p=0\}$
is the union of finitely many surfaces
of dimension~$(m-1)+(n-m)=n-1$. This completes the proof of~\eqref{POLY2}.

We also stress that, in light of~\eqref{TRANS}, the intersection of
the hypersurfaces described in~\eqref{POLY2}
and~$\S^{n-1}$ have codimension~$1$
inside~$\S^{n-1}$. In particular,
for every~$p\in \S^{n-1}$ outside these hypersurfaces there exists~$\rho(p)\in\left(0,\frac12\right)$
such that~$B_{\rho(p)}(p)$ does not intersect these hypersurfaces.

Given~$x\in B_1\setminus\{0\}$, we now use the notation~$\hat{x}:=x/|x|$ and we claim that
\begin{equation}\label{PEW2345O}
{\mbox{there exists~$x_*\in B_{1/2}\setminus\{0\}$ such that~$U_r(x_*)>0$ and~$\hat{x}_*$
lies outside the hypersurfaces~\eqref{POLY2}.}}
\end{equation}
Indeed, we can assume that~$|B_{1/2}\cap\{U_r>0\}|>0$ (otherwise~$u\le0$,
contradicting the assumption that~$z\in\partial\{u>0\}$), and from this we obtain~\eqref{PEW2345O}.

{F}rom~\eqref{PEW2345O}, we deduce that~$B_{\rho(\hat x_*)}(\hat x_*)$ does not intersect the hypersurfaces in~\eqref{POLY2}. Hence, by~\eqref{TRANS}, setting~$r(x_*):=|x_*|\rho(\hat x_*)$,
we see that~$B_{r(x_*)}(x_*)$
does not intersect the hypersurfaces in~\eqref{POLY2}.

Then, from~\eqref{FOR1}, it follows that
if $r=r_k$ is sufficiently small, then~$B_{r(x_*)/2}(x_*)$
does not intersect~$\partial\{U_r>0\}$. For this reason, since~$U_r(x_*)>0$,
we conclude that~$B_{r(x_*)/2}(x_*)\subseteq\{U_r>0\}$.

Consequently, we are in the position of using
claim~$\bf 2^\circ$ in Theorem~\ref{thm-nondeg}, thus obtaining that
\begin{equation}\label{0909-2435-1pqwe-0} \sup_{B_{r(x_*)/2}(x_*)}U_r^+\ge \bar c\, (r(x_*)/2)^2=
\frac{\bar c\, (r(x_*))^2}4=\frac{\bar c\, (\rho(\hat x_*))^2}4\,|x_*|^2
,\end{equation}
for some~$\bar c>0$.

Now we claim that
\begin{equation} \label{0909-2435-1pqwe-1} B_{1}\supseteq B_{r(x_*)/2}(x_*).
\end{equation}
Indeed, if~$y\in B_{r(x_*)/2}(x_*)$, we have that
$$ |y|\le|y-x_*|+|x_*|\le \frac{r(x_*)}2+|x_*|=\frac{\rho(\hat x_*)\,|x_*|}2+|x_*|\le
\frac{|x_*|}{4}+|x_*|=\frac{5|x_*|}{4}\le\frac{5}{8}<1
,$$
thus proving~\eqref{0909-2435-1pqwe-1}.

Then, from~\eqref{0909-2435-1pqwe-0} and~\eqref{0909-2435-1pqwe-1} we obtain that
$$ \sup_{B_1}U_r^+\ge \frac{\bar c\, (\rho(\hat x_*))^2}4\,|x_*|^2=: \hat C,$$
and~\eqref{r4tungr67wnc} follows,
as desired.
\end{proof}

\section{Monotonicity formula: proof of Theorem~\ref{lemma:F}}\label{PF:MO}

This section is devoted to the proof of Theorem~\ref{lemma:F},
which is based on a series of careful integration by parts
aimed at spotting suitable integral cancellations. In addition,
some ``high order of differentiability'' terms naturally appear in the computations,
which need to be suitably removed in order to rigorously make sense
of the formal manipulations. 
We start with some general computations valid in~$\R^n$,
then, from~\eqref{ijNAYNTE} on, we specialize to the case~$n=2$.
In this part of the paper, for the sake of shortness, we suppose that
the assumptions of Theorem~\ref{lemma:F} are always satisfied
without further mentioning them.
Without loss of generality, we also suppose that~$B_2\subset\subset\Omega$.
Then, we have the following identity:

\begin{lemma}
For every~$r_1$, $r_2\in(0,3/2)$,
\begin{equation}\label{71qy81qush}\begin{split}
4\int_{r_1}^{r_2}R(r)\,dr
+2T(r_2)-2T(r_1)
+D(r_2)-D(r_1)=0,\end{split}\end{equation}
where
\begin{equation}\label{R1r}
\begin{split}
R(r)\,&:=\frac1{r^{n+1}}\sum_{m=1}^n
\int_{B_r} \Delta u\, \nabla u_{m}\cdot e_m
-\sum_{m=1}^n\int_{\partial B_{r}} \Delta u\nabla u_m\cdot
\frac{x^m\,x}{r^{n+2}}\\&=
\frac1{r^{n+1}}
\int_{B_r} |\Delta u|^2
-\frac{1}{r^n}\int_{\partial B_{r}} \Delta u\,\partial^2_r u
,\\
T(r)\,&:=
\sum_{m=1}^n\int_{\partial B_{r}} \Delta u\,u_m\,\frac{x^m}{r^{n+1}}\\
&=\frac1{r^n}\int_{\partial B_{r}} \Delta u\,\partial_ru
\\{\mbox{and }} \qquad D(r)\,&:=
\frac1{r^n}\int_{B_r} \big( |\Delta u|^2+\chi_{\{u>0\}}\big),
\end{split}\end{equation}
and the notation~$\partial_r:=\frac{x}{|x|}\cdot\nabla$ has been used.
\end{lemma}

\begin{proof} Fix~$r\in(0,3/2)$.
We let~$\delta>0$ (to be taken as small as we wish in what follows),
and consider a smooth function~$\eta=\eta_\delta$ supported in~$B_{r+\delta}$.
Fixed~$\e>0$, we also consider the mollifier~$\rho_\e(x):=\frac1{\e^n}\rho\left(\frac{x}\e\right)$,
for a given even function~$\rho\in C^\infty_0(B_1)$. We also define~$\phi=(\phi^1,\dots,\phi^n)
:\R^n\to\R^n$ as
\begin{eqnarray*} \R^n\ni x=(x^1,\dots,x^n)\longmapsto\phi^m(x)&:=&(\psi^m*\rho_\e)(x),
\\{\mbox{where }}\qquad\psi^m(x)&:=& x^m\eta(x).\end{eqnarray*}
Let also
\begin{equation}\label{Fm defx} F^m(x):=\Delta u(x)\,u_m(x).\end{equation}
In view of~\eqref{3287uUUSp} and~\eqref{wduep}
(if~$u$ is a minimizer), or recalling that~$u$ is assumed to be
in~$C^{1,1}(\Omega)$ (if~$u$ is a one-phase minimizer),
we know that
\begin{equation*}
F^m\in L^p(B_1)\qquad{\mbox{for every }}p\in(1,+\infty).\end{equation*}
We observe that~$\psi^m$ is supported in~$B_{r+\delta}$
and so~$\phi^m$ is supported in~$B_{r+\delta+\e}\subset B_1$,
as long as $\delta$ and~$\e$ are sufficiently small. Consequently,
\begin{equation}\label{767ui6329512-3}
\begin{split}
&\int_\Omega \Delta u \,u_m\,\Delta\phi^m=
\int_{\R^n} \Delta u \,u_m\,\Delta\phi^m\\&\qquad=
\int_{\R^n} F^m\,\big( \Delta\psi^m*\rho_\e\big)
=\iint_{\R^n\times B_\e(x)}
F^m(x)\,\Delta\psi^m(y)\,\rho_\e(x-y)\,dx\,dy\\
&\qquad=\iint_{B_\e(x)\times\R^n}
F^m(x)\,\Delta\psi^m(y)\,\rho_\e(y-x)\,dx\,dy=
\iint_{\R^n}
(F^m*\rho_\e)(y)\,\Delta\psi^m(y)\,dy\\&\qquad=
\int_{\Omega}
F^m_\e\,\Delta\psi^m=-\int_\Omega \nabla F^m_\e\cdot\nabla\psi^m,
\end{split}
\end{equation}
with
\begin{equation}\label{Fm defx:22} F^m_\e:=F^m*\rho_\e.\end{equation}
Similarly, we have that
\begin{equation}
\label{8j8ij2w9o2} \int_\Omega \Delta u\nabla u_m\cdot \nabla\phi^m=
\int_\Omega \Delta u\nabla u_m\cdot (\nabla\psi^m*\rho_\e)=
\int_\Omega \big((\Delta u\nabla u_m)*\rho_\e\big)\cdot \nabla\psi^m
.\end{equation}
Also,
$$ \int_\Omega\Big( |\Delta u|^2+\chi_{\{u>0\}}\Big)\mbox{\rm div}\phi=
\sum_{m=1}^n \int_\Omega\Big( |\Delta u|^2+\chi_{\{u>0\}}\Big)(\psi^m_m*\rho_\e)
=\int_\Omega\Big(\big( |\Delta u|^2+\chi_{\{u>0\}}\big)*\rho_\e\Big)\,{\rm div}\psi
.$$
Then, we plug this information,
\eqref{767ui6329512-3} and~\eqref{8j8ij2w9o2} into~\eqref{AS D} and we see that
\begin{equation}\label{PLklj83uejISKH}\begin{split}
0\,&=2\int_\Omega \Delta u\sum_{m=1}^n\Big( 2\nabla u_m\cdot \nabla\phi^m
+u_m\Delta\phi^m\Big)-
\int_\Omega
\Big( |\Delta u|^2+\chi_{\{u>0\}}\Big)\mbox{\rm div}\phi\\
&=
4\sum_{m=1}^n\int_\Omega \big((\Delta u\nabla u_m)*\rho_\e\big)\cdot \nabla\psi^m
-2\sum_{m=1}^n\int_\Omega \nabla F^m_\e\cdot\nabla\psi^m-
\int_\Omega\Big(\big( |\Delta u|^2+\chi_{\{u>0\}}\big)*\rho_\e\Big)\,{\rm div}\psi.
\end{split}\end{equation}
Since the latter identity only involves the first derivatives of~$\psi^m$,
up to an approximation argument we can choose~$\eta$ to be the radial Lipschitz function
defined by
$$ \eta(x):=\begin{cases} 1 & {\mbox{ if }}x\in B_r,\\
\displaystyle\frac{r+\delta-|x|}{\delta}& {\mbox{ if }}x\in B_{r+\delta}\setminus B_r,\\
0 & {\mbox{ if }}x\in \R^n\setminus B_{r+\delta}.
\end{cases}$$
In this way, we have that
\begin{eqnarray*}&& \nabla \eta(x)=
-\frac{x}{\delta\,|x|}\,\chi_{ B_{r+\delta}\setminus B_r}(x)\\
{\mbox{and }}&&\nabla\psi^m(x)=e_m\eta(x)-\frac{x^m\,x}{\delta\,|x|}\,
\chi_{ B_{r+\delta}\setminus B_r}(x),
\end{eqnarray*}
which also gives that
$$ {\rm div}\psi(x)=
n\eta(x)-\frac{|x|}{\delta}\,\chi_{ B_{r+\delta}\setminus B_r}(x).$$
Therefore, we infer from~\eqref{PLklj83uejISKH} that
\begin{equation*}\begin{split}
0\,&=
2\sum_{m=1}^n\int_{B_r} \Big(2\big((\Delta u\nabla u_m)*\rho_\e\big)
-\nabla F^m_\e\Big)\cdot e_m-
n\int_{B_r}\Big(\big( |\Delta u|^2+\chi_{\{u>0\}}\big)*\rho_\e\Big)\\
&\qquad+
2\sum_{m=1}^n\int_{B_{r+\delta}\setminus B_r} \Big( 2\big((\Delta u\nabla u_m)*\rho_\e\big)
-\nabla F^m_\e\Big)\cdot
\left( e_m\eta(x)-\frac{x^m\,x}{\delta\,|x|}\right)
\\&\qquad
-
\int_{B_{r+\delta}\setminus B_r}\Big(\big( |\Delta u|^2+\chi_{\{u>0\}}\big)*\rho_\e\Big)\,
\left( n\eta(x)-\frac{|x|}{\delta}\right).
\end{split}\end{equation*}
Then, sending~$\delta\to0^+$, we deduce that
\begin{equation}\label{56yui9-Adefx}\begin{split}
0\,&=
2\sum_{m=1}^n\int_{B_r} \Big(2\big((\Delta u\nabla u_m)*\rho_\e\big)
-\nabla F^m_\e\Big)\cdot e_m-
n\int_{B_r}\Big(\big( |\Delta u|^2+\chi_{\{u>0\}}\big)*\rho_\e\Big)\\
&\qquad-
2\sum_{m=1}^n\int_{\partial B_{r}} \Big( 2\big((\Delta u\nabla u_m)*\rho_\e\big)
-\nabla F^m_\e\Big)\cdot
\frac{x^m\,x}{r}
\\&\qquad
+r\int_{\partial B_r}\Big(\big( |\Delta u|^2+\chi_{\{u>0\}}\big)*\rho_\e\Big)\,\\
&=
2\sum_{m=1}^n\int_{B_r} G^m_\e\cdot e_m-
n\int_{B_r}\Big(\big( |\Delta u|^2+\chi_{\{u>0\}}\big)*\rho_\e\Big)\\
&\qquad-
2\sum_{m=1}^n\int_{\partial B_{r}} G^m_\e\cdot
\frac{x^m\,x}{r}
+r\int_{\partial B_r}\Big(\big( |\Delta u|^2+\chi_{\{u>0\}}\big)*\rho_\e\Big)\,
,\end{split}\end{equation}
where
\begin{equation}\label{GM} G^m_\e:=
2\big((\Delta u\nabla u_m)*\rho_\e\big)-\nabla F^m_\e
.\end{equation}
Furthermore, letting
\begin{equation}\label{GMD}
D_\e(r):=\frac1{r^n}\int_{B_r}\Big(\big( |\Delta u|^2+\chi_{\{u>0\}}\big)*\rho_\e\Big),\end{equation}
we have that
\begin{equation}\label{8quak18iwjed9ihw1627} D_\e'(r)=\frac1{r^n}\int_{\partial B_r}\Big(\big( |\Delta u|^2+\chi_{\{u>0\}}\big)*\rho_\e\Big)
-\frac{n}{r^{n+1}}\int_{B_r}\Big(\big( |\Delta u|^2+\chi_{\{u>0\}}\big)*\rho_\e\Big).
\end{equation}
Thus, we multiply~\eqref{56yui9-Adefx} by~$\frac{1}{r^{n+1}}$
and we exploit~\eqref{8quak18iwjed9ihw1627} to conclude that
\begin{equation}\label{56yui9-Adefx-BIS}\begin{split}
0\,&=
\frac{2}{r^{n+1}}\sum_{m=1}^n\int_{B_r} G^m_\e\cdot e_m
-2\sum_{m=1}^n\int_{\partial B_{r}} G^m_\e\cdot
\frac{x^m\,x}{r^{n+2}}
+D'_\e(r)\\
&=2Z_\e(r)+D'_\e(r),\end{split}\end{equation}
where
\begin{equation}\label{56yui9-Adefx-TRIS} Z_\e(r):=
\frac{1}{r^{n+1}}\sum_{m=1}^n\int_{B_r} G^m_\e\cdot e_m
-\sum_{m=1}^n\int_{\partial B_{r}} G^m_\e\cdot
\frac{x^m\,x}{r^{n+2}}.\end{equation}
Now, in light of~\eqref{Fm defx}, we observe that~$\nabla F_m$ (and thus~$\nabla F_m^\e$)
involves third derivatives, and therefore we aim at ``lowering the order of derivative'' of
this term from~\eqref{56yui9-Adefx-TRIS} in view of~\eqref{GM}
(and this goal will be accomplished
via a suitable averaging procedure).
To this end, we observe that
\begin{equation}\label{18688ujows920iwudw8u}
\int_{B_r} \nabla F^m_\e\cdot e_m=\int_{B_r}{\rm div}(F^m_\e e_m)=
\int_{\partial B_r} F^m_\e e_m\cdot\frac{x}{r}=\int_{\partial B_r} F^m_\e\, \frac{x^m}{r}.
\end{equation}
We notice that the last term in~\eqref{18688ujows920iwudw8u}
does not contain any third order derivatives.
As for the boundary term in~\eqref{56yui9-Adefx}
that involves the third derivative, we have that
\begin{eqnarray*}
\int_{\partial B_{r}} \na F^m_\e\cdot
\frac{x^m\,x}{r^{n+2}}
&=&
\int_{\partial B_{1}} \na F^m_\e(rx)\cdot
\frac{x^m\,x}{r}\\
&=&
\int_{\partial B_{1}} \p_r(F^m_\e(rx))\cdot
\frac{x^m}{r}\\
&=&
\frac{d}{dr}\left\{\int_{\partial B_{1}} F^m_\e(rx)
\frac{x^m}{r}\right\}+\int_{\partial B_{1}} F^m_\e(rx)
\frac{x^m}{r^{2}}\\
&=&
\frac{d}{dr}\left\{\int_{\partial B_{r}} F^m_\e
\frac{x^m}{r^{n+1}}\right\}+\int_{\partial B_{r}} F^m_\e
\frac{x^m}{r^{n+2}}.
\end{eqnarray*}
As a consequence, using the latter identity, \eqref{GM} and~\eqref{18688ujows920iwudw8u},
we find that
\begin{eqnarray*}
\int_{B_r} G^m_\e\cdot e_m&=&
2\int_{B_r}\big((\Delta u\nabla u_m)*\rho_\e\big)\cdot e_m-\int_{B_r}\nabla F^m_\e\cdot e_m
\\&=&2\int_{B_r}\big((\Delta u\nabla u_m)*\rho_\e\big)\cdot e_m-
\int_{\partial B_r} F^m_\e\, \frac{x^m}{r}
\\
{\mbox{and }}\qquad\int_{\partial B_{r}} G^m_\e\cdot
\frac{x^m\,x}{r^{n+2}}&=&
2\int_{\partial B_{r}}\big((\Delta u\nabla u_m)*\rho_\e\big)\cdot
\frac{x^m\,x}{r^{n+2}}-
\int_{\partial B_{r}}\nabla F^m_\e\cdot
\frac{x^m\,x}{r^{n+2}}\\&=&2
\int_{\partial B_{r}}\big((\Delta u\nabla u_m)*\rho_\e\big)\cdot
\frac{x^m\,x}{r^{n+2}}-
\int_{\partial B_{r}} F^m_\e
\frac{x^m}{r^{n+2}}-
\frac{d}{dr}\left\{\int_{\partial B_{r}} F^m_\e
\frac{x^m}{r^{n+1}}\right\}.
\end{eqnarray*}
{F}rom this and~\eqref{56yui9-Adefx-TRIS}, we obtain that
\begin{eqnarray*}
Z_\e(r)&=&
\frac2{r^{n+1}}\sum_{m=1}^n\int_{B_r}\big((\Delta u\nabla u_m)*\rho_\e\big)\cdot e_m
-2\sum_{m=1}^n
\int_{\partial B_{r}}\big((\Delta u\nabla u_m)*\rho_\e\big)\cdot
\frac{x^m\,x}{r^{n+2}}+\sum_{m=1}^n
\frac{d}{dr}\left\{\int_{\partial B_{r}} F^m_\e
\frac{x^m}{r^{n+1}}\right\}\\
&=& 2R_\e(r)+T_\e'(r),\end{eqnarray*}
with
\begin{equation}\label{9ijdkscn8uefigvidsakgdisgjausud}
\begin{split}&
R_\e(r):=\frac1{r^{n+1}}\sum_{m=1}^n
\int_{B_r}\big((\Delta u\nabla u_m)*\rho_\e\big)\cdot e_m
-\sum_{m=1}^n
\int_{\partial B_{r}}\big((\Delta u\nabla u_m)*\rho_\e\big)\cdot
\frac{x^m\,x}{r^{n+2}}\\ {\mbox{and }}\;&T_\e(r):=\sum_{m=1}^n
\left\{\int_{\partial B_{r}} F^m_\e
\frac{x^m}{r^{n+1}}\right\}=
\sum_{m=1}^n
\left\{\int_{\partial B_{r}} (\Delta u\,u_m)*\rho_\e\,
\frac{x^m}{r^{n+1}}\right\}
,\end{split}
\end{equation}
where we have also used~\eqref{Fm defx}
and~\eqref{Fm defx:22}.

Consequently, integrating~\eqref{56yui9-Adefx-BIS},
\begin{equation}\label{R1qwertytu6r}
\begin{split}
0\,&= 2\int_{r_1}^{r_2} Z_\e(r)\,dr+D_\e(r_2)-D_\e(r_1)\\
&= 4\int_{r_1}^{r_2} R_\e(r)\,dr+2T_\e(r_2)-2T_\e(r_1)+D_\e(r_2)-D_\e(r_1).
\end{split}\end{equation}
Comparing~\eqref{R1r}
with~\eqref{9ijdkscn8uefigvidsakgdisgjausud},
we see that~$R_\e\to R$ and~$T_\e\to T$ as~$\e\to0$,
thanks to~\eqref{3287uUUSp}
and~\eqref{wduep}.

We thereby obtain the desired claim in~\eqref{71qy81qush}
by passing to the limit the identity in~\eqref{R1qwertytu6r}.
\end{proof}

We also point out the following useful calculation:

\begin{lemma}
In the notation stated by~\eqref{R1r},
we have that
\begin{equation}\label{DV-1} 4\int_{r_1}^{r_2}
\left( \frac1{r^n}\int_{\partial B_r}\Delta u\,\Big(2\frac{u_r}{r}
- \partial^2_r u- 2\frac{u}{r^2}\Big) \right)\,dr-4V(r_2)+4V(r_1)
+2T(r_2)-2T(r_1)
+D(r_2)-D(r_1)=0,\end{equation}
where
\begin{equation}\label{DV-2}
V(r):= \frac1{r^{n+1}}\int_{\partial B_r}\Delta u u.
\end{equation}
\end{lemma}

\begin{proof}
For any smooth function~$v$,
\begin{equation}\label{9:9:ia1ap}
\begin{split}
& \int_{B_r} |\Delta v|^2 = \int_{B_r}\Big( {\rm div}(\Delta v\nabla v)-
\nabla\Delta v\cdot\nabla v\Big)
=\int_{\partial B_r}\Delta v\,v_r-\int_{B_r}\nabla\Delta v\cdot\nabla v\\
&\qquad=\int_{\partial B_r}\Delta v\,v_r-\int_{B_r}{\rm div}(v\nabla\Delta v)
+\int_{B_r} \Delta^2v\,v\\
&\qquad=\int_{\partial B_r}\Delta v\,v_r-\int_{\partial B_r} v\,\Delta v_r
+\int_{B_r} \Delta^2v\,v.
\end{split}\end{equation}
We also observe that
\begin{eqnarray*}&&
\frac{d}{dr} \left(\frac1{r^{n+1}}\int_{\partial B_r}\Delta v v\right)\\&=&
\frac{d}{dr} \left(\frac1{r^{2}}\int_{\partial B_1}\Delta v(r\theta) v(r\theta)\right)\\&=&
-\frac2{r^{3}}\int_{\partial B_1}\Delta v(r\theta) v(r\theta)
+
\frac1{r^{2}}\int_{\partial B_1}\Delta v_r(r\theta) v(r\theta)
+
\frac1{r^{2}}\int_{\partial B_1}\Delta v(r\theta) v_r(r\theta)\\&=&
-\frac2{r^{n+2}}\int_{\partial B_r}\Delta v\,v
+
\frac1{r^{n+1}}\int_{\partial B_r}\Delta v_r\, v
+
\frac1{r^{n+1}}\int_{\partial B_r}\Delta v\,v_r.
\end{eqnarray*}
{F}rom this and~\eqref{9:9:ia1ap}, we obtain that,
for any smooth function~$v$,
\begin{eqnarray*}
&& \frac1{r^{n+1}}\int_{B_r} |\Delta v|^2
-\frac1{r^n}\int_{\partial B_{r}} \Delta v\,\partial^2_r v\\
&=&
\frac1{r^{n+1}}\int_{\partial B_r}\Delta v\,v_r-\frac1{r^{n+1}}\int_{\partial B_r} v\,\Delta v_r
+\frac1{r^{n+1}}\int_{B_r} \Delta^2v\,v-\frac1{r^n}\int_{\partial B_{r}} \Delta v\,\partial^2_r v\\
&=&\frac1{r^{n}}\int_{\partial B_r}\Delta v\,\left(2\frac{v_r}{r}-\partial^2_r v
-2\frac{v}{r^2}\right)
+\frac1{r^{n+1}}\int_{B_r} \Delta^2v\,v
-\frac{d}{dr} \left(\frac1{r^{n+1}}\int_{\partial B_r}\Delta v v\right).
\end{eqnarray*}
Integrating this identity and setting
\begin{equation}\label{DV-2v}
V_v(r):= \frac1{r^{n+1}}\int_{\partial B_r}\Delta vv,
\end{equation}
we thereby obtain that
\begin{equation}\label{PamduST}
\begin{split}&
\int_{r_1}^{r_2}\left( \frac1{r^{n+1}}\int_{B_r} |\Delta v|^2
-\frac1{r^n}\int_{\partial B_{r}} \Delta v\,\partial^2_r v
\right)\,dr
\\ =\;&
\int_{r_1}^{r_2}\left(
\frac1{r^{n}}\int_{\partial B_r}\Delta v\,\left(2\frac{v_r}{r}-\partial^2_r v
-2\frac{v}{r^2}\right)
+\frac1{r^{n+1}}\int_{B_r} \Delta^2v\,v\right)\,dr-V_v(r_2)+V_v(r_1)
.\end{split}\end{equation}
The idea is now to take~$v$ as a mollification
of~$u$, and use 
either~\eqref{7yh0128eUDIS} (if~$u$ is a minimizer)
or Lemma~\ref{NEBVERAMAL} (if~$u$ is a one-phase minimizer).
In this way, the term
$$ \int_{B_r} \Delta^2v\,v$$
approaches either
$$ \int_{B_r} u\,\mathscr M_u,$$
in the notation of~\eqref{7yh0128eUDIS} (if~$u$ is a minimizer),
or~$0$ (if~$u$ is a one-phase minimizer, due to
Lemma~\ref{NEBVERAMAL}).

To make the notation uniform,
we therefore define~$\mathscr M_u^*:=\mathscr M_u$
if~$u$ is a minimizer
and~$\mathscr M_u^*:=0$
if~$u$ is a one-phase minimizer: then,
approximating~$u$, passing to the limit~\eqref{PamduST}
and comparing~\eqref{DV-2v} with~\eqref{DV-2},
we can write
\begin{eqnarray*}&&
\int_{r_1}^{r_2}\left( \frac1{r^{n+1}}\int_{B_r} |\Delta u|^2
-\frac1{r^n}\int_{\partial B_{r}} \Delta u\,\partial^2_r u
\right)\,dr
\\ &=&
\int_{r_1}^{r_2}\left(
\frac1{r^{n}}\int_{\partial B_r}\Delta u\,\left(2\frac{u_r}{r}-\partial^2_r u
-2\frac{u}{r^2}\right)
-\frac1{r^{n+1}}\int_{B_r} u\,\mathscr M_u^*\right)\,dr-V(r_2)+V(r_1)
.\end{eqnarray*}
That is, recalling~\eqref{R1r},
$$ \int_{r_1}^{r_2}R(r)\,dr=
\int_{r_1}^{r_2}\left(
\frac1{r^{n}}\int_{\partial B_r}\Delta u\,\left(2\frac{u_r}{r}-\partial^2_r u
-2\frac{u}{r^2}\right)
-\frac1{r^{n+1}}\int_{B_r} u\,\mathscr M_u^*\right)\,dr-V(r_2)+V(r_1).$$
{F}rom this and~\eqref{71qy81qush} we obtain that
\begin{equation}\label{P9876AOM}
\begin{split}&
2T(r_1)-2T(r_2)
+D(r_1)-D(r_2)\\
=\;&
4\int_{r_1}^{r_2}\left(
\frac1{r^{n}}\int_{\partial B_r}\Delta u\,\left(2\frac{u_r}{r}-\partial^2_r u
-2\frac{u}{r^2}\right)
-\frac1{r^{n+1}}\int_{B_r} u\,\mathscr M_u^*\right)\,dr-4V(r_2)+4V(r_1).\end{split}
\end{equation}
Now we claim that
\begin{equation}\label{6gaTGV}
\int_{B_r} u\,\mathscr M_u^*=0.
\end{equation}
For this, 
since~$\mathscr M_u^*=0$ in the one-phase problem,
we can suppose that~$u$ is a minimizer, in which case~$\mathscr M_u^*
=\mathscr M_u$.
Then,
let us fix~$\delta\in(0,1)$. {F}rom Lemma~\ref{POBIA}, we know that
\begin{eqnarray*}
-\int_{B_r\cap\{ |u|\ge\delta\}} u\,\mathscr M_u
=\int_{B_r\cap\{ u\ge\delta\}} u\,\Delta^2u+\int_{B_r\cap\{ u\le-\delta\}} u\,\Delta^2u=0.
\end{eqnarray*}
Therefore, exploiting Lemma~\ref{LEMMA M},
$$ \left| \int_{B_r} u\,\mathscr M_u \right|
=\left| \int_{B_r\cap\{ |u|<\delta\}} u\,\mathscr M_u\right|\le\delta
\mathscr M_u(B_r)\le C\delta r^{n-2},$$
for some~$C>0$. Then, sending~$\delta\to0^+$,
we obtain~\eqref{6gaTGV} as desired.

Then, the identities in~\eqref{P9876AOM}
and~\eqref{6gaTGV}
lead to~\eqref{DV-1}.
\end{proof}

Now we restrict the previous calculations to the case $n=2$,
and we complete the proof of \eqref{MONOFORMULA}.

\begin{proof}[Proof of \eqref{MONOFORMULA}]
Using using polar coordinates~$(r,\theta)$, we compute 
\begin{equation}\label{ijNAYNTE}
\begin{split}
-\frac1{r^n}\int_{\partial B_r}\Delta u\,\Big(2\frac{u_r}{r}
- \partial^2_r u- 2\frac{u}{r^2}\Big)
=
& \int_{\partial B_1}\frac1{r}\Delta u\,\Big(
u_{rr}-2\frac{u_r}{r}+2\frac{u}{r^2}\Big)
\\=\;&\int_{\partial B_1}\frac1{r}
\Big( u_{rr}+\frac{u_r}r+\frac{u_{\theta\theta}}{r^2} \Big)\Big(u_{rr}-2\frac{u_r}{r}+2\frac{u}{r^2}\Big)
\\=\;&
A(r)+B(r),
\end{split}\end{equation}
where 
\begin{equation}\label{7wqtfychv78rtef7465y45ihgbksajgdf}
\begin{split}&
A(r):= \int_{\p B_1}\frac1{r^3}u_{\theta\theta}\Big(
u_{rr}-2\frac{u_r}{r}+2\frac{u}{r^2}\Big)\\ {\mbox{and }}\quad&
B(r):=\int_{\p B_1}\frac1{r}\Big( u_{rr}+\frac{u_r}r\Big)\Big(u_{rr}-2\frac{u_r}{r}+2\frac{u}{r^2}\Big)
.\end{split}
\end{equation}
Now we perform several
integrations by parts that involve the terms related to~$A(r)$. First of all, we see that
\begin{equation}\label{COMA-MA1}
\begin{split}
\frac1{r^3}\int_{\p B_1}u_{\theta\theta}u_{rr}=
&-
\frac1{r^3}\int_{\p B_1} u_\theta u_{\theta rr}\\
=&
-\frac d{dr}\int_{\p B_1}\frac{u_\theta u_{r\theta}}{r^3}+\int_{\p B_1}\frac{u_{r\theta}^2}{r^3}-3
\int_{\p B_1}\frac{u_\theta u_{\theta r}}{r^4}.
\end{split}
\end{equation}
Similarly, we have that
\begin{equation}\label{COMA-MA2}
-2\int_{\p B_1}\frac1{r^4} u_{\theta\theta} u_r=
2\int_{\p B_1}\frac{u_\theta u_{\theta r}}{r^4}= 2\int_{\p B_1}\frac{u_\theta u_{\theta r}}{r^4}
\end{equation}
and 
\begin{equation}\label{COMA-MA3}
2\int_{\p B_1}\frac1{r^5} u_{\theta\theta} u=-2\int_{\p B_1}\frac{u_\theta^2}{r^5}.
\end{equation}
Combining~\eqref{COMA-MA1}, \eqref{COMA-MA2} and~\eqref{COMA-MA3},
and recalling~\eqref{7wqtfychv78rtef7465y45ihgbksajgdf},
we get 
\begin{equation}\label{BIR0}
\begin{split}
A(r)=
&
-\frac d{dr}\left(
\int_{\p B_1}\frac{u_\theta u_{r\theta}}{r^3}\right)+\int_{\p B_1}\frac{u_{r\theta}^2}{r^3}-3
\int_{\p B_1}\frac{u_\theta u_{\theta r}}{r^4}
+2\int_{\p B_1}\frac{u_\theta u_{\theta r}}{r^4}
-2\int_{\p B_1}\frac{u_\theta^2}{r^5}\\
=&
-\frac d{dr}\left(\int_{\p B_1}\frac{u_\theta u_{r\theta}}{r^3}\right)
+\int_{\p B_1}\frac{u_{r\theta}^2}{r^3}-
\int_{\p B_1}\frac{u_\theta u_{\theta r}}{r^4}
-2\int_{\p B_1}\frac{u_\theta^2}{r^5}\\
= 
&
-\frac d{dr} \left(\int_{\p B_1} \frac{u_\theta u_{r\theta}}{r^3}\right)
+\int_{\p B_1}\frac1{r^3}\left( u_{\theta r}-\frac{2 u_\theta}r \right)^2
+
3\int_{\p B_1}\frac{u_\theta u_{\theta r}}{r^4}
-6\int_{\p B_1}\frac{u_\theta^2}{r^5}\\
=&
-\frac d{dr} \left( \int_{\p B_1} \frac{u_\theta u_{r\theta}}{r^3}
+ \frac32\int_{\p B_1}\frac{u_\theta^2}{r^4}\right)
+\int_{\p B_1}\frac1{r^3}\left( u_{\theta r}-\frac{2 u_r}r \right)^2
\\=&
-\frac d{dr} \left( \int_{\p B_r} \frac{u_\theta u_{r\theta}}{r^4}
+ \frac32\int_{\p B_r}\frac{u_\theta^2}{r^5}\right)
+\int_{\p B_r}\frac1{r^4}\left( u_{\theta r}-\frac{2 u_r}r \right)^2
.\end{split}
\end{equation}
{F}rom~\eqref{7wqtfychv78rtef7465y45ihgbksajgdf},
we also compute that
\begin{equation}\label{BIR}
\begin{split}
B(r)=&
\int_{\p B_1}\frac1r\left(
u_{rr}^2-\frac{2u_{rr}u_r}r+\frac{2uu_{rr}}{r^2}+\frac{u_r u_{rr}}r-\frac{2 u_r^2}{r^2}
+\frac{2u u_r}{r^3}
\right)\\
=&
\int_{\p B_1}\frac1r\left(
u_{rr}^2-\frac{u_{rr}u_r}r+\frac{2uu_{rr}}{r^2}-\frac{2 u_r^2}{r^2}
+\frac{2u u_r}{r^3}
\right)\\
=&
\int_{\p B_1}\frac1r\left(u_{rr}-\frac{3u_r}r+4\frac u{r^2}\right)^2
+
\frac1r
\left(
\frac{5u_ru_{rr}}r-\frac{6uu_{rr}}{r^2}-\frac{11u_r^2}{r^2}+\frac{26u u_r}{r^3}
-\frac{16u^2}{r^4}
\right)\\
=&
\int_{\p B_1}\frac1r\left(u_{rr}-\frac{3u_r}r+4\frac u{r^2}\right)^2
+\frac d{dr}\left(\int_{\p B_1}\frac{5u_r^2}{2r^2}
-\int_{\p B_1}\frac{6u u_r}{r^3}+
\int_{\p B_1}\frac{4u^2}{r^4}\right)\\
=&
\int_{\p B_r}\frac1{r^2}\left(u_{rr}-\frac{3u_r}r+4\frac u{r^2}\right)^2
+\frac d{dr}\left(\int_{\p B_r}\frac{5u_r^2}{2r^3}
-\int_{\p B_r}\frac{6u u_r}{r^4}+
\int_{\p B_r}\frac{4u^2}{r^5}\right).
\end{split}
\end{equation}
Using~\eqref{BIR0} and~\eqref{BIR}, we conclude that
\begin{equation} \label{ABOAiw}
A(r)+B(r)=\frac1{r^2}
\int_{\p B_r}\left[\left( \frac{u_{\theta r}}r-\frac{2 u_r}{r^2} \right)^2+
\left(u_{rr}-\frac{3u_r}r+4\frac u{r^2}\right)^2\right]+W'(r),
\end{equation}
where
\begin{equation}\label{87333efgrhrggh6r8yfc8i21qrutyerod-12ytge}
W(r):=
\int_{\p B_r}\left( \frac{5u_r^2}{2r^3}-\frac{6u u_r}{r^4}+
\frac{4u^2}{r^5}
- \frac{u_\theta u_{r\theta}}{r^4}
- \frac{3u_\theta^2}{2r^5}\right).
\end{equation}
Now, from~\eqref{DV-1} 
and~\eqref{ijNAYNTE}, we see that
\begin{eqnarray*}&&
-4V(r_2)+4V(r_1)
+2T(r_2)-2T(r_1)
+D(r_2)-D(r_1)\\&=&-
4\int_{r_1}^{r_2}
\left( \frac1{r^n}\int_{\partial B_r}\Delta u\,\Big(2\frac{u_r}{r}
- \partial^2_r u- 2\frac{u}{r^2}\Big) \right)\,dr
\\&=& 4\int_{r_1}^{r_2}\big(
A(r)+B(r)\big)\,dr.\end{eqnarray*}
This and~\eqref{ABOAiw} give that
\begin{equation}\label{MAHijyhJAHNNA}\begin{split}
&-V(r_2)+V(r_1)
+\frac{T(r_2)-T(r_1)}2
+\frac{D(r_2)-D(r_1)}4-W(r_2)+W(r_1)\\
=\;&
\int_{r_1}^{r_2}\left\{
\frac1{r^2}
\int_{\p B_r}\left[\left( \frac{u_{\theta r}}r-\frac{2 u_r}{r^2} \right)^2+
\left(u_{rr}-\frac{3u_r}r+4\frac u{r^2}\right)^2\right]\right\}
.
\end{split}\end{equation}
Recalling \eqref{EMMEDE}, \eqref{R1r}, \eqref{DV-2} and \eqref{87333efgrhrggh6r8yfc8i21qrutyerod-12ytge},
we see that
\begin{eqnarray*}
&&-V(r)
+\frac{T(r)}2
+\frac{D(r)}4-\int_{\p B_r}\left( \frac{5u_r^2}{2r^3}-\frac{6u u_r}{r^4}+
\frac{4u^2}{r^5}
- \frac{u_\theta u_{r\theta}}{r^4}
- \frac{3u_\theta^2}{2r^5}\right)\\
&=&-\frac1{r^{3}}\int_{\partial B_r}\Delta u u
+\frac1{2r^2}\int_{\partial B_{r}} \Delta u\,\partial_ru
+\frac1{4r^2}\int_{B_r} \big( |\Delta u|^2+\chi_{\{u>0\}}\big)\\&&\qquad-
\int_{\p B_r}\left( \frac{5u_r^2}{2r^3}-\frac{6u u_r}{r^4}+
\frac{4u^2}{r^5}
- \frac{u_\theta u_{r\theta}}{r^4}
- \frac{3u_\theta^2}{2r^5}\right)
\\
&=& E(r).
\end{eqnarray*}
This and~\eqref{MAHijyhJAHNNA} establish \eqref{MONOFORMULA},
as desired. \end{proof}

Now, since the proof of \eqref{MONOFORMULA} has been completed,
to finish the proof of Theorem~\ref{lemma:F},
we only need to show that the function~$E$ defined in~\eqref{EMMEDE}
is bounded
and to check that
if~$E$ is constant 
then~$u$
is a homogeneous function of degree two.

These goals will be accomplished by the following arguments:

\begin{proof}[Proof of the boundedness of $E$]
To show that~$E$ is bounded, we claim that
the exist~$C>0$ and a sequence~$r_k\to0^+$ such that
\begin{equation}\label{BOU SPH}
\int_{\partial B_{r_k}} \left(\frac{|\nabla u|^2}{r_k^3}
+\frac{|D^2u|^2}{r_k}\right)\le C.
\end{equation}
The proof of~\eqref{BOU SPH} needs to distinguish the
case in which $u$ is a minimizer from the case in which~$u$
is a one-phase minimizer.
Suppose first that~$u$ is a one-phase minimizer.
Then, since~$u(0)=0\le u(x)$ for any~$x\in\Omega$
and~$u$ is assumed to be~$C^{1,1}(\Omega)$,
we can write that~$|\nabla u(x)|\le C|x|$
and~$|D^2u(x)|\le C$, for some~$C>0$, from which~\eqref{BOU SPH}
plainly follows in this case.

Now, we prove~\eqref{BOU SPH} assuming that~$u$ is a minimizer.
We argue by contradiction, supposing that~\eqref{BOU SPH}
does not hold. Then, for any~$\bar{C}>0$ there exists~$\bar{r}\in(0,1)$
such that for any~$r\in(0,\bar{r})$ we have that
$$ \int_{\partial B_{r}} \left(\frac{|\nabla u|^2}{r^3}
+\frac{|D^2u|^2}{r}\right)\ge \bar{C}.$$
This,
Corollary~\ref{eq-Hessian:cor} (if~$u$ is a minimizer)
or the fact that~$u$ is assumed to be in~$C^{1,1}(\Omega)$
(if~$u$ is a one-phase minimizer) lead that, for a suitable~$C>0$,
\begin{eqnarray*}
C &\ge&
\frac1{{\bar{r}}^{4}}\int_{B_{{\bar{r}}}}|\na  u|^2
+\frac1{{\bar{r}}^2}\int_{B_{\bar{r}}}|D^2 u|^2
\\ &=&
\frac1{{\bar{r}}^{4}}\int_0^{{\bar{r}}} \left(\int_{\partial B_r} |\na  u|^2\right)\,dr
+\frac1{{\bar{r}}^2}\int_0^{{\bar{r}}}\left(\int_{\partial B_r}|D^2 u|^2\right)\,dr
\\&=& \frac1{{\bar{r}}}\int_0^{{\bar{r}}} \left(\int_{\partial B_r} \frac{|\na  u|^2}{
\bar{r}^3}
+\int_{\partial B_r}\frac{|D^2 u|^2}{\bar{r}}\right)\,dr\\
&\ge& \frac1{{\bar{r}}}\int_{\frac{\bar{r}}2}^{{\bar{r}}} \left(\int_{\partial B_r} \frac{|\na  u|^2}{
\bar{r}^3}
+\int_{\partial B_r}\frac{|D^2 u|^2}{\bar{r}}\right)\,dr\\
&\ge& \frac1{8{\bar{r}}}\int_{\frac{\bar{r}}2}^{{\bar{r}}} \left(\int_{\partial B_r}
\frac{|\na  u|^2}{{r}^3}
+\int_{\partial B_r}\frac{|D^2 u|^2}{{r}}\right)\,dr\\
&\ge& \frac{\bar{C}}{16},
\end{eqnarray*}
which is a contradiction if~$\bar{C}$ is suitably large, and this establishes~\eqref{BOU SPH}.

As a consequence, using the Cauchy-Schwarz inequality, Theorem~\ref{growth}
and~\eqref{BOU SPH}, 
\begin{eqnarray*}&&
\int_{\partial B_{r_k}}
\left|\frac{\Delta u\,u_r}{2r_k^2}-
\frac{5u_r^2}{2r_k^3}
-\frac{\Delta u u}{r_k^3}+\frac{6u u_r}{r_k^4}+
\frac{u_{\theta}u_{\theta r}}{r_k^4}
-\frac{4u^2}{r_k^5}
-\frac{3u_\theta^2}{2r_k^5}
\right|\\&\le&
C\,\int_{\partial B_{r_k}}\left(
\frac{|D^2 u|\,|\nabla u|}{r_k^{\frac12}\,r_k^{\frac32}}+
\frac{|\nabla u|^2}{r_k^3}
+\frac{|\Delta u |}{r_k^{\frac12}\,r_k^{\frac12}}+\frac{|\nabla u|}{
r_k^{\frac32}\,r_k^{\frac12}
}+\frac{1}{r_k}
\right)\\&\le&
C\,\int_{\partial B_{r_k}}\left(
\frac{|\nabla u|^2}{r_k^3}+
\frac{|D^2 u|^2}{r_k}+\frac{1}{r_k}
\right)\\
&\le& C,
\end{eqnarray*}
for some~$C>0$, possibly varying from line to line.

Using this, \eqref{EMMEDE} and Corollary~\ref{eq-Hessian:cor} (if~$u$
is a minimizer) or the assumption that~$u\in C^{1,1}(\Omega)$ (if~$u$
is a one-phase minimizer),
we thereby deduce that
\begin{equation}\label{rj-rk-al}
\begin{split}
&|E(r_k)| \\ \le\;&
\int_{\partial B_{r_k}}
\left|\frac{\Delta u\,u_r}{2r_k^2}-
\frac{5u_r^2}{2r_k^3}
-\frac{\Delta u u}{r_k^3}+\frac{6u u_r}{r_k^4}+
\frac{u_{\theta}u_{\theta r}}{r_k^4}
-\frac{4u^2}{r_k^5}
-\frac{3u_\theta^2}{2r_k^5}
\right|
+
\frac1{4r_k^2}\int_{B_{r_k}} \big( |\Delta u|^2+\chi_{\{u>0\}}\big)
\\ \le\; &C+\frac1{4r_k^2}\int_{B_{r_k}}\chi_{\{u>0\}}\\
\le\;& C
,\end{split}
\end{equation}
up to renaming~$C>0$.

Now, fix~$r\in(0,1)$. Let~$\bar k$ sufficiently large, such that~$r_{\bar k}\in(0,r)$.
{F}rom~\eqref{MONOFORMULA}, we know that
$$ E(r_{\bar k})\le E(r)\le E(1).$$
Hence, by~\eqref{rj-rk-al},
$$ -C\le E(r)\le E(1),$$
and this shows that~$E$ is bounded, as desired.
\end{proof}

Having already checked the validity of
the monotonicity formula in~\eqref{MONOFORMULA} and
the fact that~$E$ is bounded, in order to complete the proof of
Theorem~\ref{lemma:F}, we only need to show that
if~$E$ is constant in~$(0,\tau)$,
then~$u$
is a homogeneous function of degree two. This is now a simple consequence
of~\eqref{MONOFORMULA}. The detailed argument goes as follows.

\begin{proof}[Proof of the case of constant $E$]
Suppose now that~$E$ is constant in~$(0,\tau)$. Then, 
by~\eqref{MONOFORMULA},
\begin{eqnarray*}
&& -\frac{\partial}{\partial \theta}
\left(-\frac{u_r}{r}+ \frac{2u}{r^2}\right)
=\frac{u_{r\theta}}{r^2}-\frac{2u_\theta}{r}=0\\
{\mbox{and }} && -r\,\frac{\partial}{\partial r}
\left(-\frac{u_r}{r}+\frac{2u}{r^2}\right)=u_{rr}-\frac{3u_r}{r}+
\frac {4u}{r^2}=0,
\end{eqnarray*}
which, in turn, gives that
$$\na \left(-\frac{u_r}{r}
+2 \frac{u}{r^2}\right)=0.$$ Consequently, the function~$
-\frac{u_r}{r}+ \frac{2u}{r^2}$ is constant for~$|x|\in(0,\tau)$, hence we write
\begin{equation}\label{ODE} -\frac{u_r}{r}+ \frac{2u}{r^2}=c,\end{equation}
for some~$c\in\R$.

Now we define
\begin{equation}\label{ODE2} v(r,\theta):=u(r,\theta)+cr^2\log r.\end{equation}
Using~\eqref{ODE}, we obtain that
$$ v_r=u_r+2cr\log r+cr=\frac{2u}r+2cr\log r=\frac{2v}r.
$$
Integrating this equation, fixed~$\bar r\in(0,\tau)$, we find that
$$ v(r,\theta)=\frac{r^2\,v(\bar r,\theta)}{\bar r^2}.$$
This and~\eqref{ODE2} give that
$$ u(r,\theta)=\frac{r^2\,v(\bar r,\theta)}{\bar r^2}-cr^2\log r.$$
Therefore, exploiting Theorem~\ref{growth} (if~$u$ is a minimizer)
or the assumption that~$u\in C^{1,1}(\Omega)$ (if~$u$
is a one-phase minimizer),
$$ C\ge\frac{|u(r,\theta)|}{r^2}\ge |c|\,|\log r|-\frac{|v(\bar r,\theta)|}{\bar r^2},$$
for some~$C>0$ and therefore
$$ |c|\le\lim_{r\to0} \frac{|v(\bar r,\theta)|}{\bar r^2\,|\log r|}+\frac{C}{|\log r|}=0.$$
Hence, we get that~$c=0$ and, as a consequence,
we can write~\eqref{ODE} as
\begin{equation*} -\frac{u_r}{r}+ \frac{2u}{r^2}=0\end{equation*}
for any~$x\in B_\tau$, and therefore~$\na u (x)\cdot x=2u(x)$ for any~$x\in B_\tau$. 
Observing that this is
the Euler equation for homogeneous functions of degree two, we thus obtain the homogeneity
of~$u$.
The proof of Theorem~\ref{lemma:F} is thereby complete.
\end{proof}

\medskip

We finish this section by an explicit computation of 
the energy $E$ for the homogeneous functions of degree two on the plane. 
It will be used later in the proof of Theorem~\ref{BYPA}.


\begin{lemma}\label{7udjhHHANSLL}
Let~${\mathscr{C}}\subseteq\R^2$ be a cone in~$\R^2$,
written in polar coordinates as
$$ {\mathscr{C}}=\big\{ (r,\theta)\in(0,+\infty) \times(\theta_1,\theta_2)\big\},$$
for some~$0\le\theta_1<\theta_2\le2\pi$.

Let~$u:{\mathscr{C}}\to\R$ be a homogeneous
function of the form~$u(x)=r^2 g(\theta)$, with~$g\in C^2([\theta_1,\theta_2])$,
$g>0$ in~$(\theta_1,\theta_2)$, and
$$ g(\theta_1)=g(\theta_2)=0 \qquad{\mbox{
and }}\qquad g'(\theta_1)=g'(\theta_2)=0.
$$
Assume also that~$\Delta u$ is
constant in~${\mathscr{C}}$.
Then, for any~$r>0$,
\begin{equation}\label{0-2euefhi03ejJJJJ}\begin{split}&
\int_{{\mathscr{C}}\cap\partial B_r}\left(
\frac{\Delta u\,u_r}{2r^2}-
\frac{5u_r^2}{2r^3}
-\frac{\Delta uu}{r^3}+\frac{6uu_r}{r^4}+
\frac{u_{\theta}u_{\theta r}}{r^4}
-\frac{4u^2}{r^5}
-\frac{3u_\theta^2}{2r^5}\right)+
\frac1{4r^2}\int_{{\mathscr{C}}\cap B_r}\big(|\Delta u|^2+\chi_{\{u>0\}}\big)
\\&\qquad=\frac{\pi}{4}\,
\frac{|\{u>0\}\cap B_r|}{|B_r|}=\frac{\theta_2-\theta_1}{8}.\end{split}
\end{equation}
\end{lemma}

\begin{proof} 
By assumption, in~${\mathscr{C}}$ we have that
\begin{equation}\label{lapcost}
C_0=\Delta u=4 g+g'',\end{equation}
for some~$C_0\in\R$,  
and
\begin{eqnarray*}
&& \frac{\Delta u\,u_r}{2r^2}-
\frac{5u_r^2}{2r^3}
-\frac{\Delta uu}{r^3}+\frac{6uu_r}{r^4}+
\frac{u_{\theta}u_{\theta r}}{r^4}
-\frac{4u^2}{r^5}
-\frac{3u_\theta^2}{2r^5}\\
&=&\frac{(4 g+g'')g}{r}-
\frac{10 g^2}{r}
-\frac{(4 g+g'')g}{r}+\frac{12 g^2}{r}+
\frac{2(g')^2}{r}
-\frac{4g^2}{r}
-\frac{3(g')^2}{2r}\\
&=&-\frac{2g^2}{r}+\frac{(g')^2}{2r}.
\end{eqnarray*}
Therefore,
after an integration by parts, and recalling~\eqref{lapcost}, we have that
\begin{equation}\label{HYHshi}
\begin{split}&
\int_{{\mathscr{C}}\cap\partial B_r}\left(
\frac{\Delta u\,u_r}{2r^2}-
\frac{5u_r^2}{2r^3}
-\frac{\Delta uu}{r^3}+\frac{6uu_r}{r^4}+
\frac{u_{\theta}u_{\theta r}}{r^4}
-\frac{4u^2}{r^5}
-\frac{3u_\theta^2}{2r^5}\right)\\
=\;&\int_{\theta_1}^{\theta_2}\left( -2g^2+\frac{(g')^2}{2}\right)\\
=\;&\int_{\theta_1}^{\theta_2}\left( - 2g^2 -\frac{g''g}{2}\right)\\
=\;&-\frac12\int_{\theta_1}^{\theta_2}g( 4g+ g'')\\
=\;&-\frac{C_0}{2}\int_{\theta_1}^{\theta_2}g\\
=\;&\frac{C_0}{8}\int_{\theta_1}^{\theta_2}(g''-C_0)\\
=\;&-\frac{C_0^2\,(\theta_2-\theta_1)}{8}.
\end{split}
\end{equation}
On the other hand,
$$ \frac1{4r^2}\int_{{\mathscr{C}}\cap B_r}|\Delta u|^2=
\frac1{8}\int_{\theta_1}^{\theta_2} (4 g+g'')^2=\frac{C_0^2\,(\theta_2-\theta_1)}{8}.
$$
This and~\eqref{HYHshi} give that
\begin{equation*}\begin{split}&
\int_{{\mathscr{C}}\cap\partial B_r}\left(
\frac{\Delta u\,u_r}{2r^2}-
\frac{5u_r^2}{2r^3}
-\frac{\Delta uu}{r^3}+\frac{6uu_r}{r^4}+
\frac{u_{\theta}u_{\theta r}}{r^4}
-\frac{4u^2}{r^5}
-\frac{3u_\theta^2}{2r^5}\right)+
\frac1{4r^2}\int_{{\mathscr{C}}\cap B_r}\big(|\Delta u|^2+\chi_{\{u>0\}}\big)\\&\qquad
=\frac1{4r^2}\int_{B_r} \chi_{\{u>0\}},\end{split}
\end{equation*}
which proves~\eqref{0-2euefhi03ejJJJJ}.
\end{proof}

\section{Monotonicity formula: homogeneity of the blow-up limits, and
proof of Theorem~\ref{thm-hom-blow}}\label{sec:class}

In this section, we apply the results in Theorem~\ref{lemma:F}
to study the homogeneity properties of the blow-up limits of
the minimizers of~$J$ at free boundary points
with vanishing gradient, thus proving Theorem~\ref{thm-hom-blow}.

\begin{proof}[Proof of Theorem~\ref{thm-hom-blow}]
Suppose that~$u$ does not vanish identically.
We let 
\begin{equation}\label{blyaaa-1}
Q(u,x):=Q(u, r,\theta)=\left(-\dfrac {u_{r\theta}}{r}
+2 \frac{u_\theta}{r^2}\right)^2+\left(u_{rr}-3\frac{u_r}{r}+4\frac u{r^2}
\right)^2.
\end{equation}
Note that $Q$ is invariant with respect to quadratic scaling. 
Indeed, if we define, for any~$s>0$,
$$ u_s(x):=\frac{u(sx)}{s^2},$$
we have that
\begin{equation}\begin{split}\label{scaling}
Q(u_s,x)=\;&\left(-\dfrac {(u_s)_{r\theta}}{r}
+2 \frac{(u_s)_\theta}{r^2}\right)^2+\left((u_s)_{rr}-3\frac{(u_s)_r}{r}+4\frac{u_s}{r^2}
\right)^2\\ =\;& 
\left(-\dfrac {u_{r\theta}(sx)}{sr}
+2 \frac{u_\theta(sx)}{(sr)^2}\right)^2
+\left(u_{rr}(sx)-3\frac{u_r(sx)}{sr}+4\frac{u(sx)}{(sr)^2}
\right)^2\\
=\;& Q(u,sx).
\end{split}\end{equation}
Now, in view of~\eqref{MONOFORMULA} and \eqref{blyaaa-1},
we observe that
\begin{equation}\label{46e5dytcCCe6rfg49orghr}
\begin{split}
E(\tau_2)-E(\tau_1)\,&=
\int_{\tau_1}^{\tau_2}\left\{
\frac1{r^2}\int_{\partial B_r}
\left[\left( \frac{u_{\theta r}}{r}-\frac{2u_\theta}{r^2}\right)^2
+\left( u_{rr}-\frac{3u_r}{r}+\frac{4u}{r^2}\right)^2\right]
\right\}\,dr\\
&= \int_{\tau_1}^{\tau_2}\left(
\frac1{r^2}\int_{\partial B_r}
Q(u,x)\,dx
\right)\,dr.\end{split}\end{equation}
As a consequence, for any~$s>0$, using the changes of variables~$\rho=r/s$
and~$y=x/s$, and making use of~\eqref{scaling},
we see that
\begin{equation}\label{4rdcwqsa7ujewds9k-x}
\begin{split}
E(s\tau_2)-E(s\tau_1)\,
&= \int_{s\tau_1}^{s\tau_2}\left(
\frac1{r^2}\int_{\partial B_r}
Q(u,x)\,dx
\right)\,dr\\
&= \int_{\tau_1}^{\tau_2}\left(
\frac1{\rho^2}\int_{\partial B_{\rho}}
Q(u,sy)\,dy
\right)\,d\rho\\
&= \int_{\tau_1}^{\tau_2}\left(
\frac1{\rho^2}\int_{\partial B_{\rho}}
Q(u_s,y)\,dy
\right)\,d\rho.
\end{split}\end{equation}
On the other hand, by Theorem~\ref{lemma:F},
we know that~$E$ is monotone and bounded, and therefore
the limit as~$\vartheta\to0^+$ of~$E(\vartheta)$ exists and it
is finite. Consequently, we have that
\[
E(s\tau_2)-E(s\tau_1)\to 0 \quad {\mbox{ as }}\, s\to 0.
\]
Hence, recalling~\eqref{4rdcwqsa7ujewds9k-x}, we conclude that
\begin{equation}\label{jewbgej0587}
\int_{\tau_1}^{\tau_2}\left(
\frac1{\rho^2}\int_{\partial B_{\rho}}
Q(u_s,y)\,dy
\right)\,d\rho\to 0\quad {\mbox{ as }}\, s\to 0.\end{equation}
Also, by compactness (ensured here, if $u$ is a minimizer,
by \eqref{34934jsq92358858678}, which
in turns allows us to exploit Corollary~\ref{eq-Hessian:cor}, and, if~$u$ is a
one-phase minimizer by the assumption that~$u\in C^{1,1}(\Omega)$),
we have that~$u_{s}$ converges to some~$u_0$, up to a subsequence.
Therefore, by~\eqref{jewbgej0587}, 
\[
\int_{\tau_1}^{\tau_2}\left(
\frac1{\rho^2}\int_{\partial B_{\rho}}
Q(u_0,y)\,dy
\right)\,d\rho=0
\]
for all $\tau_2>\tau_1>0$. 
Thus, since $Q\ge0$, due to~\eqref{blyaaa-1},
it follows that~$Q(u_0, y)=0$.
Consequently, by~\eqref{46e5dytcCCe6rfg49orghr},
we have that the function~$E$ relative to the minimizer~$u_0$
is identically constant. Therefore, in view of the last claim in Theorem~\ref{lemma:F},
it follows that~$u_0$ is a homogeneous function of degree two.\end{proof}

\section{Regularity of the free boundary in two dimensions:
explicit computations, classification results in~$2D$,
and proof of Theorem~\ref{BYPA}}\label{sec:reg}

In this  section we study the regularity of free boundary
of minimizers in dimension~$2$. Some of the results presented rely on direct
calculations, while others are
obtained by the monotone quantity $E$
that has been analyzed in
Theorems~\ref{lemma:F} and~\ref{thm-hom-blow}. 
In this setting,
we have the following classification result for one-phase minimizers:

\begin{theorem}\label{gDg0}
Let~$u\in C^1(\R^n)$ be a one-phase local minimizer
in any ball of~$\R^n$, with~$0\in \fbs u$.
Let~$u=r^2g(\theta)$, 
where~$(r,\theta)$ denotes the polar coordinates.
Then, the following dichotomy holds:
\begin{itemize}
\item either~$u$
is a homogeneous polynomial of degree two,
\item or, up to a rotation,
\begin{equation*}
u(x)=a\,(x_1^+)^2,
\end{equation*}
for some~$a>0$.
\end{itemize}
\end{theorem}

\begin{proof} 
A direct computation shows that
\begin{equation}\label{9ikde9uo32ew9uoewhdxhi65464783-zx}
\Delta u=u_{rr}+\frac{u_r}{r} + \frac{1}{r^2}\, \Delta_{{\mathbb{S}}^1}u
= 2g +2g+ g''
=
g''+4g.
\end{equation}
Accordingly, by Lemma~\ref{POBIA}, we have that,
in the positivity set of~$u$, we have
$$r^2\Delta^2 u=\frac{d^2}{d\theta^2}(g''+4g)=0.$$
{F}rom this, we deduce that
\begin{equation}\label{fherger25324}
g''(\theta)+4g(\theta)=c_1\theta+c_2,\qquad{\mbox{for all }}\theta\in\{g\ne0\}
,\end{equation}
for some constants~$c_1$ and~$c_2$.
We notice that~\eqref{fherger25324}
has explicit solution
\begin{equation}\label{g explicit}
\begin{split}
g(\theta)\,&=
\frac{c_1\theta}{4} + \frac{c_2}{4} +c_3\cos(2\theta)
+c_4 \sin(2\theta)\\
&=\frac{c_1\theta}{4} + \frac{c_2}{4} +c_3(\cos^2\theta-\sin^2\theta)
+2c_4 \sin\theta\cos\theta,\end{split}
\end{equation}
for some constants~$c_3$ and~$c_4$.

Since~$0$ is a free boundary point for~$u$, we have that~$g$ cannot
vanish identically. Hence, we distinguish
some cases, depending on the number of zeros of~$g$.
First of all, we consider the cases in which
either~$g>0$ for all~$\theta\in[0,2\pi)$
or~$g$ vanishes only at one point. Then, 
in this case the free boundary is contained in a ray and,
up to a rotation,
we can assume that~$g(\theta)>0$ for all~$\theta\in(0,2\pi)$
and so~\eqref{g explicit} is valid for all~$\theta\in(0,2\pi)$.
The periodicity of~$g$ then implies that
$$ 0=\lim_{\theta\to0^+}g(\theta)-\lim_{\theta\to2\pi^-}g(\theta)
=-\frac{c_1\pi}{2},$$
and so~$c_1=0$. As a consequence, by~\eqref{g explicit},
$$ u(r,\theta)= \frac{c_2r^2}{4} +c_3r^2(\cos^2\theta-\sin^2\theta)
+2r^2c_4 \sin\theta\cos\theta=
\frac{c_2(x_1^2+x_2^2)}{4} +c_3(x_1^2-x_2^2)
+2 c_4 x_1x_2,$$
which
is a homogeneous polynomial of degree two,
thus proving the desired claim in this case.

Now we suppose that~$g$ vanishes at least
at two points, say, up to a rotation, $\theta_0$ and~$-\theta_0$,
for some~$\theta_0\in(0,\pi)$, that is
\begin{equation}\label{Coasyh2}
\begin{split}
g(\theta)>0{\mbox{ for all }}\theta\in(-\theta_0,\theta_0),\\
{\mbox{and }}g(\theta_0)=g(-\theta_0)=0.
\end{split}
\end{equation}
Then, by~\eqref{g explicit},
\begin{equation} \label{SYS1}0=g(\pm\theta_0)=\pm
\frac{c_1\theta_0}{4} + \frac{c_2}{4} +c_3\cos(2\theta_0)
\pm c_4 \sin(2\theta_0).\end{equation}
By the assumptions that~$u\in C^1(\R^n)$ and~$g\ge0$, we also know
that
\begin{equation} \label{SYS2}0=g'(\pm\theta_0)=
\frac{c_1}{4}  \mp 2c_3\sin(2\theta_0)
+2c_4 \cos(2\theta_0).\end{equation}
Then, we obtain from~\eqref{SYS1} and~\eqref{SYS2} the system
\begin{equation}\label{rUIATcoap} \left\{
\begin{matrix}
\displaystyle\frac{c_1\theta_0}{4} +c_4 \sin(2\theta_0)=0,\\
\displaystyle\frac{c_2}{4} +c_3\cos(2\theta_0)=0,\\
c_3\sin(2\theta_0)=0,\\
\displaystyle\frac{c_1}{4}  
+2c_4 \cos(2\theta_0)=0.
\end{matrix}
\right.\end{equation}
Now, if
\begin{equation}\label{9osjx8wdubjxCVA}
\theta_0\ne\pi/2,\end{equation} from~\eqref{rUIATcoap}
we have that necessarily~$c_3=0$,
and accordingly
\begin{equation*} \left\{
\begin{matrix}
\displaystyle\frac{c_1\theta_0}{4} +c_4 \sin(2\theta_0)=0,\\
\displaystyle\frac{c_2}{4} =0,\\
\displaystyle\frac{c_1}{4}  
+2c_4 \cos(2\theta_0)=0.
\end{matrix}
\right.\end{equation*}
This implies that~$c_2=0$, and so~\eqref{g explicit} becomes
\begin{equation*}
g(\theta)=
\frac{c_1\theta}{4} 
+c_4 \sin(2\theta).
\end{equation*}
In particular~$g(0)=0$, which is in contradiction with~\eqref{Coasyh2}.

This says that the case in~\eqref{9osjx8wdubjxCVA} must be ruled out,
and thus~$\theta_0=\pi/2$ (and the positivity sets of~$u$
are either one or two halfplanes). In this way, the system in~\eqref{rUIATcoap}
reduces to
\begin{equation*} \left\{
\begin{matrix}
\displaystyle\frac{c_1\pi}{8} =0,\\
\\
\displaystyle\frac{c_2}{4} -c_3=0,\\
\\
\displaystyle\frac{c_1}{4}  
-2c_4 =0,
\end{matrix}
\right.\end{equation*}
which leads to~$c_1=c_4=0$ and~$\frac{c_2}4=c_3$. Substituting these
conditions into~\eqref{g explicit}, we obtain that, for all~$\theta\in(-\pi/2,\pi/2)$,
\begin{equation*}
g(\theta)=
c_3\big(1 +\cos(2\theta)\big)= c_3\big(1+\cos^2\theta-\sin^2\theta\big),
\end{equation*}
and therefore, for all~$x=(x_1,x_2)\in\R^2$ with~$x_1>0$,
$$ u(x)=2c_3 x_1^2.$$
This gives that either~$u$ is a homogeneous polynomial of degree two,
or~$u(x)=a\,(x_1^+)^2$ for some~$a>0$, or
$$ u(x)=\left\{ \begin{matrix}
a x_1^2 & {\mbox{ if }} x_1\ge0,\\
b x_1^2 & {\mbox{ if }} x_1<0,\end{matrix}
\right.$$
with~$a$, $b\in(0,+\infty)$ and
\begin{equation}\label{aneb}
a\ne b.\end{equation}
To complete the proof of the desired result, we need to exclude this case.
To this end, we observe that
\begin{eqnarray*}
&&\big(|\Delta u(0^+,1)|^2+1\big)-
2\big(\Delta u(0^+,1) u_{11}(0^+,1)-u_1(0^+,1)\Delta u(0^+,1)\big)\\
&=&\big((2a)^2+1\big)-
2\big((2a)^2+0\big)\\
&=& 1 -4a^2,
\end{eqnarray*}
and similarly
\begin{eqnarray*}
\big(|\Delta u(0^-,1)|^2+1\big)-
2\big(\Delta u(0^-,1) u_{11}(0^-,1)-u_1(0^-,1)\Delta u(0^-,1)\big)
&=& 1 -4b^2.
\end{eqnarray*}
These identities and the free boundary condition~\eqref{Nu1} computed
at the point~$(0,1)$, where according to the 
definition in~\eqref{LAMDEFI}
we have~$\lambda^{(1)}=\lambda^{(2)}=1$,  
 lead to
$$ 1-4a^2=1-4b^2,$$
which gives that~$a^2=b^2$ and thus~$a=b$. This is in contradiction
with~\eqref{aneb}, and the desired result is established.
\end{proof}

With this, we are now in the position of completing the proof of
Theorem~\ref{BYPA}.

\begin{proof}[Proof of Theorem~\ref{BYPA}] Let~$E$
be as in Theorem~\ref{lemma:F}, and let\footnote{We observe
that the limit in~\eqref{LI001} exist, due to the monotonicity of~$E$, recall Theorem~\ref{lemma:F}.}
\begin{equation}\label{LI001}
E(0):=\lim_{\rho\to0^+} E(\rho).
\end{equation}
Let~$\bar x\in\partial\{u>0\}$.
Suppose that $u_{0,\bar x}$ is a blow-up of $u$ at~$\bar x$. Notice that
$u_{0,\bar x}$ cannot be identically equal to zero, due to~\eqref{SUPu}.
Then 
by Theorem~\ref{gDg0} we know that, after some rotation of coordinates,
\begin{equation}\label{on3e9di}\begin{split}&
{\mbox{$u_{0,\bar x}$ must be one of the following 
functions: }}\\
&\frac{a_1(x_1-\bar x_1)^2+a_2 (x_2-\bar x_2)^2}2,\quad
\frac{a(x_1-\bar x_1)^2}2,\quad
\frac{ a(
(x_1-\bar x_1)^+)^2}2
,\end{split}\end{equation}
with~$a_1$, $a_2$, $a>0$ (say, possibly depending on~$\bar x$, though
the free boundary conditions in Theorem~\ref{FREE BOU COND} have to be fulfilled).

In particular, from~\eqref{on3e9di}, we know that
\begin{equation}\label{V:cos-IKAzasol}
{\mbox{$\Delta u$ is constant in the positivity cone of~$u$.}}\end{equation}
Now, from \eqref{7uhn7yhnb7yhb02394}, we know that, if
\begin{equation}\label{BARuPA} u_{k,\bar x}(x):=\frac{u(\bar x+\rho_k x)}{\rho_k^2},\end{equation}
with~$\rho_k\to0^+$, then,
up to a subsequence, 
\begin{equation}\label{9idy61eyrwf0eouryfoifg}
{\mbox{$u_{k,\bar x}\to u_{0,\bar x}$
in~$C^{1,\alpha}_{\rm{loc}}(\R^n)$,}}\end{equation}
as~$k\to+\infty$, for any~$\alpha\in(0,1)$.

We claim that
\begin{equation}\label{on3e9diBIS}
{\mbox{$u_{0,0}$ must necessarily be }}
\frac{ a(x_1^+)^2}2,\end{equation}
namely the first and the second possibilities in~\eqref{on3e9di} are
excluded at the origin.
To prove~\eqref{on3e9diBIS}, we argue by contradiction.
If not, by~\eqref{9idy61eyrwf0eouryfoifg} and~\eqref{on3e9di}, necessarily
$$ \frac{u(\rho_k x)}{\rho_k^2}=u_{k,0}(x)\to\left\{\begin{matrix}
{\mbox{either }}
\displaystyle\frac{a_1x_1^2+a_2 x_2^2}2,\\ 
\\
{\mbox{or }}
\displaystyle\frac{ax_1^2}2\end{matrix}
\right\}
=:u_{0,0}(x)$$
in~$C^{1,\alpha}_{\rm{loc}}(\R^n)$. Therefore, using the change of variable~$y:=\rho_k x$,
\begin{eqnarray*}
&&\lim_{k\to+\infty}\frac{|B_{\rho_k}\cap \po u|}{|B_{\rho_k}|}=
\lim_{k\to+\infty}\frac{1}{|B_{\rho_k}|}\int_{B_{\rho_k}\cap\{u>0\}}\,dx\\
&&\qquad=\lim_{k\to+\infty}\frac{1}{|B_{1}|}\int_{B_{1}\cap\{u_{k,0}>0\}}\,dy
=\frac{1}{|B_{1}|}\int_{B_{1}\cap\{u_{0,0}>0\}}\,dy=1
.\end{eqnarray*}
This is a contradiction with~\eqref{LIMSUP-0}, and so~\eqref{on3e9diBIS}
is proved.

We let~$E_{k,\bar x}$ be the monotone function
in~\eqref{EMMEDE} for~$u_{k,\bar x}$ (while~$E_{\bar x}$ denotes the
same type of function for~$u$ centered at the point~$\bar x$).
Let also~$E_{0,\bar x}$ be the monotone function
in~\eqref{EMMEDE} for~$u_{0,\bar x}$.
In view of~\eqref{9idy61eyrwf0eouryfoifg}, we have that
\begin{equation}\label{9idy61eyrwf0eouryfoifg-2}
E_{0,\bar x}(r)=\lim_{k\to+\infty} E_{k,\bar x}(r).
\end{equation}
We remark that~\eqref{EMMEDE} is compatible with the blow-up scaling, namely
$$ E_{k,\bar x}(r)=E_{\bar x}(\rho_k r).$$
As a consequence, by~\eqref{LI001} and~\eqref{9idy61eyrwf0eouryfoifg-2},
\begin{equation}\label{23CO-000} 
E_{0,\bar x}(r)=\lim_{k\to+\infty}E_{\bar x}(\rho_k r)=E_{\bar x}(0).
\end{equation}

We now classify the free boundary points according to the monotone function
induced by their blow-up limits. For this,
we introduce the following notation: recalling~\eqref{on3e9di},
we say that~$\bar x$ is Type-1 if, up to a rotation,
$$u_{0,\bar x}(x)=
\frac{a_1(x_1-\bar x_1)^2+a_2 (x_2-\bar x_2)^2}2.$$ Similarly,
we say that~$\bar x$ is Type-2 if
$$u_{0,\bar x}(x)=
\frac{a(x_1-\bar x_1)^2}2,$$ and Type-3 if
$$u_{0,\bar x}(x)=
\frac{ a(
(x_1-\bar x_1)^+)^2}2.$$
In this notation, \eqref{on3e9diBIS} says that the origin is Type-3.

Now, in light of~\eqref{EMMEDE} and Lemma~\ref{7udjhHHANSLL}
(which can be utilized here thanks to~\eqref{V:cos-IKAzasol}),
we have that
\begin{equation}\label{MINIMTYPE}
E_{0,\bar x}(r)=\begin{cases}
\pi/4, & {\mbox{ if $\bar x$ is either Type-1 or Type-2,}}\\
\pi/8 ,& {\mbox{ if $\bar x$ is Type-3.}}
\end{cases}
\end{equation}
In particular, the monotone function~$E$ is minimized for Type-3
free boundary points.

Moreover, we have the following semicontinuity property:
if $x_j\in\partial\{u>0\}$ and~$x_j\to x_0$ as~$j\to+\infty$, then
\begin{equation}\label{SEMICOE}
\limsup_{j\to+\infty} E_{x_j}(0)\le E_{x_0}(0).
\end{equation}
Indeed, by the monotonicity of~$E$ in Theorem~\ref{lemma:F}
and~\eqref{EMMEDE},
for any~$r\in(0,1)$ we have that
$$ \limsup_{j\to+\infty} E_{x_j}(0)\le \limsup_{j\to+\infty} E_{x_j}(r)=E_{x_0}(r).
$$
Then, we take the limit as~$r\to0^+$ and we obtain~\eqref{SEMICOE},
as desired.

Now we claim that there exists~$r_0>0$ such that
\begin{equation}\label{TUTTO}
{\mbox{for any~$\bar x\in\partial\{u>0\}\cap B_{r_0}$ we have
that $E_{\bar x}(0)=E_{0}(0)$.}}
\end{equation}
In other words, in $B_{r_0}$ all free boundary points must be of Type-3. 
To prove this we argue by contradiction: if not there exists a sequence of points~$\bar x_j\in
\partial\{u>0\}$ such that~$\bar x_j\to0$ as~$j\to+\infty$
and
\begin{equation}\label{TUTTO1}
E_{\bar x_j}(0)\ne E_{0}(0).\end{equation}
{F}rom~\eqref{on3e9di}, 
\eqref{on3e9diBIS}, \eqref{23CO-000}, \eqref{MINIMTYPE} and~\eqref{TUTTO1},
we deduce that
$$ \left\{
\frac{\pi}8,\frac{\pi}4
\right\}\ni E_{0,\bar x_j}(r)=E_{\bar x_j}(0)\ne E_{0}(0)=E_{0,0}(r)=\frac{\pi}8,$$
and accordingly
$$ E_{\bar x_j}(0)=E_{0,\bar x_j}(r)=\frac{\pi}4 >\frac{\pi}8=E_{0,0}(r)=E_0(0).$$
This gives that
$$ \lim_{j\to+\infty}E_{\bar x_j}(0)=\frac{\pi}4>
\frac{\pi}8=E_0(0),$$
which is in contradiction with~\eqref{SEMICOE},
and so the proof of~\eqref{TUTTO} is complete.

Then, by \eqref{MINIMTYPE} and~\eqref{TUTTO}, it follows that
if~$\bar x\in\partial\{u>0\}\cap B_{r_0}$ then~$\bar x$ must
necessarily be Type-3, i.e.,
up to rotations,~$u_{0,\bar x}(x)=
\frac{ a((x_1-\bar x_1)^+)^2}2$, which is the desired result.
\end{proof}

\appendix
\section{Decay estimate for $D^2 u$}
Here we provide some decay estimates for the gradient and the Hessian
of a local minimizer of the functional~$J$ in~\eqref{defJ}.

\begin{lemma}\label{eq-Hessian}
Let $n\ge 2$, $u$ be a
minimizer for the functional~$J$
defined in~\eqref{defJ},
and~$x_0\in \fb u$.

Assume that~$B_{\bar R}\subset\subset\Om$
 
Then, there exist~$R_0\in(0,\bar R)$ and~$C>0$, depending only on~$n$,
$\|u\|_{W^{2,2}(\Om)}$ and~$\dist(B_{\bar R},\partial\Om)$,
such that 
\begin{equation*}
\frac1{R^{n+2}}\int_{B_{R}(x_0)}|\na  u|^2+\frac1{R^n}\int_{B_{R}(x_0)}|D^2 u|^2
\le \frac C{R^{n+4}}\int_{B_{4R}(x_0)}(u-m)^2+
\frac{ C}{R^{n+2}}\int_{B_{4R}}(u-m), 
\end{equation*}
for any~$R\in(0,R_0)$,
where \begin{equation}\label{Em}m:=\min_{B_{4R}(x_0)} u.\end{equation}
\end{lemma}

\begin{proof}
Without loss of generality we suppose that~$x_0=0$.
Recalling Lemma~\ref{POBIA}, we have that, for any~$\phi\in W^{2,2}(\Om)
\cap W^{1, 2}_0(\Om)$, with~$\phi\ge0$, it holds that
\begin{equation}\label{grp-1}
0\ge \int_{\Om}\Delta u\,\Delta \phi.
\end{equation}
Now, let $\phi\in C_0^\infty(\Om)$, with~$\phi\ge 0$,
and define~$\phi_\e:=\phi*\rho_\e$, where~$\rho_\e(x):=\frac1{\e^n}\rho(\frac{x}\e)$, for any~$x\in\R^n$,
is a mollifying kernel, for any~$\e\in(0,1)$.
We also set~$u_\e:=u*\rho_\e$.
Then, if~$\dist({\rm{\supp}} \phi, \partial\Om)\gg \e$, we can use~\eqref{grp-1}
and make an integration by parts twice to obtain
that 
\begin{equation}\begin{split}\label{9t5y8vhjg9867}
0\ge \int_\Om \Delta u\,\Delta \phi_\e=\;&  \int_\Om \Delta u(\Delta \phi)*\rho_\e\\
=\;&
\int_\Om \Delta u(x)\left( \int_\Om\rho_\e(x-y)\Delta \phi(y)dy\right)dx\\
=\;&
\int_\Om \Delta \phi(y)\left( \int_\Om\rho_\e(x-y)\Delta u(x)dx\right)dy\\
=\;&
\int_\Om \Delta \phi\,\Delta u_\e\\
=\;&\sum_{i,j=1}^n
\int_{\Om} \phi_{ii} (u_\e)_{jj}\\
=\;&\sum_{i,j=1}^n
\int_{\p\Om}\phi_i (u_\e)_{jj}\nu^i-\sum_{i,j=1}^n\int_\Om\phi_i(u_\e)_{ijj}\\
=\;&\sum_{i,j=1}^n
\int_{\p\Om}\phi_i (u_\e)_{jj}\nu^i
-\sum_{i,j=1}^n\int_{\p\Om}\phi_i (u_\e)_{ij}\nu^j+\sum_{i,j=1}^n\int_\Om \phi_{ij}(u_\e)_{ij}\\
=\;&\sum_{i,j=1}^n
\int_\Om\phi_{ij}(u_\e)_{ij}.
\end{split}\end{equation}
Moreover, we observe that
$$  \lim_{\e\to0}\int_\Om\phi_{ij}(u_\e)_{ij}=\int_\Om\phi_{ij} u_{ij}.$$
{F}rom this and~\eqref{9t5y8vhjg9867}, we have that
\begin{equation}\label{grp-2}
\sum_{i,j=1}^n\int_\Om\phi_{ij} u_{ij}\le 0.
\end{equation}
Now, we choose~$\phi:=(u-m)\eta^2$, where~$m$ is as in~\eqref{Em},
and~$\eta $ is a standard cut-off
function supported in~$B_{2R}\subset\subset\Omega$, such that~$ \eta=1$ in~$B_{R}$ and~$
\eta =0$ outside~$B_{2R}$. Therefore, we see that~$
\phi\in W^{2,2}(\Om)\cap W^{1, 2}_0(\Om)$ and~$\phi\ge0$. 
With this choice,
$$\phi_{ij}=u_{ij}\eta^2+2u_i\eta_j\eta+2u_j\eta_i\eta
+(u-m)(\eta^2)_{ij}.$$  If we plug this into~\eqref{grp-2}, we have that
$$ \sum_{i,j=1}^n\int_\Om\Big(u_{ij}\eta^2+4u_i\eta_j\eta
+(u-m)(\eta^2)_{ij}\Big)\,u_{ij}\le 0.$$
That is, rearranging the terms and integrating by parts,
\begin{equation}\begin{split}\label{grp-3}
\sum_{i,j=1}^n\int_\Om u_{ij}^2\eta^2\le\; &-\sum_{i,j=1}^n
\int_\Om \Big(4u_{ij} u_i\eta_j\eta+(u-m)u_{ij}(\eta^2)_{ij}\Big)\\
=\;&
-\sum_{i,j=1}^n\int_\Om 4(u_{ij}\eta)u_i\eta_j+\sum_{i,j=1}^n
\int_\Om\Big( (u-m)u_i(\eta^2)_{ijj}+u_i u_j(\eta^2)_{ij}\Big)
\\
\le\;&
2\delta \sum_{i,j=1}^n\int_\Om u_{ij}^2\eta^2+\frac 8\delta\sum_{i,j=1}^n
\int_\Om u_i^2\eta_j^2
+\sum_{i,j=1}^n
\int_\Om\Big( (u-m)u_i(\eta^2)_{ijj}+u_i u_j(\eta^2)_{ij}\Big),
\end{split}\end{equation}
where the last line follows from a suitable application of
the H\"older inequality, for some~$\delta>0$. 

Now, by direct computations we have 
\begin{eqnarray*}
(\eta^2)_{ij}&=&2\eta_i\eta_j+2\eta\eta_{ij}\\
{\mbox{and }}\quad
(\eta^2)_{ijj}&=& 2\eta_{ij}\eta_j+2\eta_i\eta_{jj}
+2\eta_j\eta_{ij}+2\eta\eta_{ijj},
\end{eqnarray*}
and therefore 
\begin{equation*}
\big|(\eta^2)_{ij}\big|\le \frac C{R^2} \quad 
{\mbox{ and }}\quad \big|(\eta^2)_{ijj}\big|\le \frac C{R^3},
\end{equation*}
for some~$C>0$.

As a consequence, plugging this information into~\eqref{grp-3}
and using the H\"older inequality, we obtain that
\begin{equation}\begin{split}\label{iu67363r:0}
(1-2\delta )\sum_{i,j=1}^n\int_\Om u_{ij}^2\eta^2
\le\; &
\frac 8\delta\sum_{i,j=1}^n\int_\Om u_i^2\eta_j^2
+\sum_{i,j=1}^n\int_\Om \Big((u-m)u_i(\eta^2)_{ijj}+u_i u_j(\eta^2)_{ij}\Big)\\
\le\;&
\frac C{\delta R^2}\int_{B_{2R}}|\na u|^2+
\frac C{R^3}\int_{B_{2R}} (u-m)|\na u|+\frac C{R^2}\int_{B_{2R}}|\na u|^2\\
\le\;&
\left(1+\frac1\delta\right)\frac C{R^2}\int_{B_{2R}}|\na u|^2
+\frac C{R^4}\int_{B_{2R}}(u-m)^2,
\end{split}\end{equation}
up to renaming~$C$.
Since $\Delta u\ge- C$ (up to renaming constants,
recall Corollary~\ref{lem-subham}),
then from the Caccioppoli inequality (see e.g.~\eqref{ineq-Caccio-main})
we get that 
\[
\int_{B_{2R}}|\na u|^2\le \frac C{R^2}
\int_{B_{4R}}(u-m)^2+C\int_{B_{4R}}(u-m), 
\]
which implies that
\begin{equation}\label{iu67363r}
\frac{1}{R^{n+2}}\int_{B_{2R}}|\nabla u|^2\le\frac{C}{R^{n+4}}
\int_{B_{4R}}(u-m)^2+\frac{C}{R^{n+2}}\int_{B_{4R}}(u-m).
\end{equation}
Moreover, from~\eqref{iu67363r:0} and~\eqref{iu67363r}, we conclude that
$$
\frac{1-2\delta }{R^n}\sum_{i,j=1}^n\int_{B_R} u_{ij}^2\le 
\frac{1-2\delta }{R^n}\sum_{i,j=1}^n\int_\Om u_{ij}^2\eta^2\le 
\frac{C}{R^{n+4}}
\int_{B_{4R}}(u-m)^2+\frac{C}{R^{n+2}}\int_{B_{4R}}(u-m)
$$
up to renaming~$C>0$. 
Putting together this and~\eqref{iu67363r}, we obtain the desired estimate.
\end{proof}

\begin{corollary}\label{eq-Hessian:cor}
Let $n\ge 2$, $\delta>0$ and~$u$ be a minimizer of the functional~$J$
defined in~\eqref{defJ} in~$\Om$. Assume that~$B_{\bar R}\subset\subset \Om$.
Let~$x_0\in \fb u$ such that~$\nabla u(x_0)=0$ and~$\partial\{u>0\}$
is not $\delta$-rank-2 flat at~$x_0$ at any level~$r>0$
in the sense of Definition~\ref{def:flat}.
 
Then, there exist~$R_0\in(0,\bar R)$ and~$C>0$, depending only on~$n$,
$\|u\|_{W^{2,2}(\Om)}$ and~$\dist(B_{\bar R},\partial\Om)$, such that
\begin{equation*}
\frac1{R^{n+2}}\int_{B_{R}(x_0)}|\na  u|^2
+\frac1{R^n}\int_{B_{R}(x_0)}|D^2 u|^2
\le C,
\end{equation*}
for any~$R<R_0$.
\end{corollary}

\begin{proof}
The desired estimate follows from Lemma~\ref{eq-Hessian}
and the quadratic growth of~$u$, as given by Theorem~\ref{growth}.~\end{proof}

\section{A remark on the one-phase problem}

Here we show that the one-phase problem, as presented in Definition~\ref{LA:HDEF1ph},
and the analysis of the minimizers which happen to be nonnegative
are structurally very different questions.
Indeed, while a ``typical'' one-phase minimizer exhibits nontrivial open
regions in which it vanishes, the
free minimizers that are nonnegative do not show the same phenomena.
As a prototype result for this, we point out the following observation:

\begin{prop}\label{NOONE}
Suppose that $0\in \Omega$, $u\in C^{1,1}(\Omega)$ is such that~$u>0$ in~$\Omega\cap\{x_n>0\}$
and~$u=0$ in~$\Omega\cap\{x_n\le0\}$. Then, $u$ cannot be a local minimizer
for the functional~$J$ in~$\Omega$ in the class of admissible functions~$
\mathcal A$ given in~\eqref{ADMI}.
\end{prop}

\begin{proof} Without loss of generality, we can assume that~$B_2\subset\subset\Omega$.
Let~$\varphi\in C^\infty_0(B_2,[0,1])$ be such that~$\varphi=1$ in~$B_1$.
Let also~$\e\in(0,1)$ and~$u_\e:= u-\e\varphi$.

We observe that the regularity of~$u$ and the fact that~$u(x',0)=0\le u(y)$ for any~$x'$ such that~$(x',0)\in B_2$
and any~$y\in B_2$ give that,
for every~$x=(x',x_n)\in B_1$, 
$$ u(x)\le K\,x_n^2,$$
for some~$K>0$. Consequently, for every~$x\in B_1$
with~$|x_n|< \sqrt{\e/K}$
we have that
$$ u_\e(x)\le K\,x_n^2 -\e<0.$$
This gives that for any~$x\in (-1/n,1/n)^{n-1}\times(0,\sqrt{\e/K})=:Q_\e$, we have that
$$ u_\e(x)<0<u(x),$$
as long as~$\e>0$ is sufficiently small.

Notice also that~$u_\e\le u$ and so~$\{u_\e>0\}\subseteq\{u>0\}$.
Accordingly, computing the energy functional in~$B_2$,
\begin{eqnarray*}
J[u_\e]-J[u]&=&
\int_{B_2} \big(|\Delta u_\e|^2-|\Delta u|^2\big)
+|B_2\cap\{u_\e>0\}|-|B_2\cap\{ u>0\}|\\
&=&
\int_{B_2} \big(|\Delta u-\e\Delta\varphi|^2-|\Delta u|^2\big)
-|B_2\cap\{u_\e\le 0< u\}|\\
&\le&
\int_{B_2} \big(
\e^2|\Delta\varphi|^2-2\e\Delta u\Delta\varphi\big)
-|Q_\e|\\
&\le& C\e-\left(\frac{2}{n}\right)^{n-1}\sqrt{\frac{\e}{K}}\\
&<&0
\end{eqnarray*}
provided that~$\e$ is small enough. 
\end{proof}

\section{Proof of an auxiliary result}\label{PALLS3u545-APP}

For completeness, in this appendix we provide the proof of
Proposition~\ref{PALLS3u545}.

\begin{proof}[Proof of Proposition~\ref{PALLS3u545}] Given~$\delta>0$, let~$p\in\partial\Omega$ with~$|u(p)|>\delta$.
By~\eqref{GA72GA0234}, we can find~$\rho>0$ such that
\begin{equation}\label{APSdfkg2345r}
\overline\Om\cap B_\rho(p)\subset\left\{|u(p)|>\frac\delta2\right\}.\end{equation}
Let~$\phi\in C^\infty_0(B_\rho(p))$, with~$\phi=0$ along~$\partial\Om$. For each~$\e\in\R$ with~$|\e|<
\frac{\delta}{4(1+\|\phi\|_{L^\infty(\R^n)})}$, we
let~$u_\e:=u+\e\phi$. We observe that
\begin{equation}\label{SABerot4}
\chi_{\{u_\e>0\}}=\chi_{\{u>0\}}\qquad{\mbox{in }}\Om.\end{equation}
Indeed, if~$x\in\Om\setminus B_\rho(p)$ we have that~$\phi(x)=0$
and thus~$u_\e(x)=u(x)$, proving~\eqref{SABerot4}
in this case. If instead~$x\in \Om\cap B_\rho(p)$, by~\eqref{APSdfkg2345r}
we can assume that~$u(x)>\delta/2$ (the case~$u(x)<-\delta/2$ being similar).
Then,
$$ u_\e(x)\ge u(x)-\e\|\phi\|_{L^\infty}>\frac\delta2
-\frac{\delta\,\|\phi\|_{L^\infty}}{4(1+\|\phi\|_{L^\infty(\R^n)})}>\frac\delta4,$$
and hence
$$ \chi_{\{u_\e>0\}}(x)=1=\chi_{\{u>0\}}(x),$$
completing the proof of~\eqref{SABerot4}.

As a byproduct of~\eqref{SABerot4}, we have that
$$ 0\le J[u_\e]-J[u]=\int_{\Omega\cap B_\rho(p)} 
\big(|\Delta u+\e\Delta\phi|^2-|\Delta u|^2\big)=
\int_{\Omega\cap B_\rho(p)} 
\big(2\e\Delta u\Delta\phi+\e^2|\Delta\phi|^2\big)
$$
yielding that
\begin{equation}\label{CIABSDboq} \int_{\Omega\cap B_\rho(p)} \Delta u\Delta\phi=0.\end{equation}
That is, defining~$v:=\Delta u$, we have that~$v$
is weakly harmonic in~$\Omega\cap B_\rho(p)$, hence harmonic
in~$\Omega\cap B_\rho(p)$, and therefore~$v$
is smooth in~$\Omega\cap B_\rho(p)$, up to the boundary.
Therefore, we deduce from~\eqref{dodcvdfgm7mse}
and~\eqref{CIABSDboq} that
$$ 0=\int_{\Omega\cap B_\rho(p)} v\Delta\phi=
\int_{\Omega\cap B_\rho(p)} \big({\rm div}(v\nabla\phi)
-{\rm div}(\phi\nabla v)\big)=
\int_{(\partial\Omega)\cap B_\rho(p)} \big(v\partial_\nu\phi-
\phi\partial_\nu v\big)=\int_{(\partial\Omega)\cap B_\rho(p)} v\partial_\nu\phi.
$$
Hence, since~$v$ is continuous on~$(\partial\Omega)\cap B_\rho(p)$,
thanks to~\eqref{dodcvdfgm7mse}, we find that~$v(p)=0$.

By taking~$\delta$ arbitrary, we thus conclude that~$v=0$
on~$(\partial\Omega)\cap\{|u|>0\}$. This and~\eqref{dodcvdfgm7mse-BIS}
give that
\begin{equation}\label{verfjverf234}
{\mbox{$v=0$ along~$\partial\Om$.}}\end{equation}
Now we prove~\eqref{APS566dfkg2345r}
by arguing by contradiction: we define~$V:=-v=-\Delta u$ and we suppose that
\begin{equation*}
M:=\sup_\Om V>0.
\end{equation*}
Now we use~\eqref{dodcvdfgm7mse}
and we find some~$\rho>0$ such that~$V$ is continuous in a $\rho$-neighborhood of~$\partial\Omega$
that we denote by~$\mathcal{O}_\rho$.
Thus, $V$ is uniformly continuous in~$\mathcal{O}_{\rho/2}$.
In particular, there exists~$\delta\in\left(0,\frac\rho2\right)$ such that if~$x$, $y\in\mathcal{O}_\delta$
with~$|x-y|\le\delta$, then~$|V(x)-V(y)|\le\frac{M}2$.

Consequently, taking~$y\in\partial\Om$ and recalling~\eqref{verfjverf234}, we find that
\begin{equation} \label{22121verfjverf234}
|V(x)|\le\frac{M}2\quad{\mbox{ for every~$x\in\mathcal{O}_\delta$, }}\end{equation}
and, as a result,
\begin{equation} \label{121verfjverf234}
0<M=\sup_\Om V=\sup_{\Om\setminus \mathcal{O}_\delta}V.\end{equation}
Furthermore,
in view of Lemma~\ref{POBIA}, for every~$\phi\in C^\infty_0(\Omega,[0,+\infty))$,
$$\int_\Om V\,\Delta\phi=-\int_\Om\Delta u\,\Delta \phi\ge 0,$$
hence~$V$ is weakly subharmonic. {F}rom this, \eqref{121verfjverf234}
and Theorem~B in~\cite{MR177186}, we deduce that~$V=M$ a.e. in~$\Omega$.
This is in contradiction with~\eqref{22121verfjverf234}, hence the claim in~\eqref{APS566dfkg2345r}
is established.
\end{proof}

\section*{Acknowledgments}

The first and third authors are members of INdAM and AustMS.
This work has been supported by the Australian Research Council
Discovery Project DP170104880 NEW ``Nonlocal Equations at Work''
and by the DECRA Project DE180100957 ``PDEs, free boundaries and applications''.

\vfill

\end{document}